\documentclass[preprint,10pt]{elsarticle}

\usepackage{ulem}
\usepackage{bm}
\usepackage{mathtools}
\usepackage{graphicx}%
\usepackage{multirow}%
\usepackage{hyperref} 
\usepackage{amsmath, amsfonts}%
\usepackage{amsthm}%
\usepackage{mathrsfs}%
\usepackage{soul}
\usepackage{enumerate}
\usepackage{enumitem}
\usepackage[title]{appendix}%
\usepackage{xcolor}%
\usepackage{textcomp}%
\usepackage{manyfoot}%
\usepackage{booktabs}%
\usepackage{algorithm}%
\usepackage{algorithmicx}%
\usepackage{algpseudocode}%
\usepackage{listings}%
\usepackage{nicefrac}%
\usepackage{caption}
\usepackage{dsfont}
\usepackage{subcaption}
\usepackage{wrapfig}
\usepackage[left=3.5cm,right=3.5cm,top=3.0cm,bottom=3.0cm]{geometry}
\usepackage[T1]{fontenc}
\usepackage[utf8]{inputenc}

\newcommand{\bbeta}{{\bm \beta}}
\newcommand{\Rmu}{{\bm {\mathcal M}}}
\newcommand{\Rmuc}{{\mathcal   M}}
\newcommand{\tv}{{\tilde v}}
\newcommand{\e}{{\bm e}}
\newcommand{\E}{\mathcal{E}}
\newcommand{\Q}{\mathcal{Q}}
\newcommand{ \Fyl}{Filippov }
\newcommand{\extracontent}[1]{}

\newcommand{\rv}{\textcolor{blue}}
\newcommand{\ind}{\mathds{1}} 
\newcommand{\hide}[1]{}
\renewcommand{\t}{\theta}
\newcommand{\at}{{\theta'}}
\renewcommand{\S}{{\mathcal S}}
\newcommand{\A}{{\mathcal A}}
\renewcommand{\AA}{{\mathbb A}}
\renewcommand{\SS}{{\mathbb S}}
\newcommand{\VV}{{\mathbb V}}
\newcommand{\EE}{{\mathbb E}}
\newcommand{\LL}{{\mathbb L}}
\newcommand{\pert}{\bm{\mathcal U}}
\newcommand{\pertc}{{\mathcal U}^a}
\newcommand{\F}{{\mathcal F}}
\renewcommand{\L}{{\mathcal L}}
\newcommand{\B}{\overline{\mathcal B}}
\newcommand{\supp}{\mbox{supp}}
\newcommand{\Fl}{{Filippov }}

\newcommand{\bmu}{{\bm \mu}}
\renewcommand{\P}{{\mathcal P}}
\newcommand{\N} {{\mathcal N}}
\newcommand{\C} {{\mathcal C}}

\newcommand{\z}{{\bf z}}
\newcommand{\nhd}{neighborhood }

\newcommand{\Removecyclic}[1]{}

\newcommand{\bromgac}{{\omega}}
\newcommand{\bromga}{{{{\bm \omega}}}}
\newcommand{\Romga}{{\bm{\mathcal W}}}
\newcommand{\Romgac}{{{\mathcal W}}} 
\newcommand{\balpha}{{\bm \alpha}}
\newcommand{\co}{\overline{\operatorname{co}}}
\newcommand{\inuc}{ \tilde{\Romgac} }
\newcommand{\inu}{ \tilde{\Romga} }
\newcommand{\persolu}{\hat{\Romga}}
\newcommand{\eop}{\hfill \textcolor{black}{\rule{1.5ex}{1.5ex}}}
\newcommand{\D}{{\mathcal D}}
\renewcommand{\H}{{\mathcal H}}
\newcommand{\R}{{\mathcal R}}
\newcommand{\G}{{\mathcal G}}
\newcommand{\ta}{{\tilde a}}

\newcommand{\I}{{\mathcal I}}
\newcommand{\g}{\bm g}
\renewcommand{\b}{\bm b}

\newcommand{\ICT}{\mathbb{I}}

\newcommand{\modl}{\oplus}

\newcommand{\Z}{{\bm{\mathcal Z}}}
\newcommand{\V}{{\bm{\mathcal V}}}
\newcommand{\J}{{\mathcal J}}

\renewcommand{\part}{ \frac{\partial h}{\partial \bromga} }
\newcommand{\CMr}{{c_m}}
\newcommand{\CSr}{{c_s}}
\newcommand{\nuu}{{a_c}}
\newcommand{\nablah}{\nabla_\bromga}
\newcommand{\nablaht}{\nabla_{\bromga(t)}}
\newcommand{ \RR}{\mathbb{R}}
\newcommand{ \CC}{\mathbb{C}}
\newtheorem{definition}{Definition}%
\newtheorem{thm}{Theorem}
\newtheorem{lem}{Lemma}

\begin{document}

\begin{frontmatter}

\title{Multi-type random game dynamics: limits at discontinuities and cyclic limits}


\author[a]{Raghupati Vyas\corref{cor1}}
\ead{raghupati.vyas@iitb.ac.in}
\author[a]{Kousik Das}
\ead{20004007@iitb.ac.in}
\author[a]{Veeraruna Kavitha}
\ead{vkavitha@iitb.ac.in}
\author[b]{Souvik Roy}
\ead{souvik.2004@gmail.com}

\cortext[cor1]{Corresponding author}

\affiliation[a]{
  organization={Indian Institute of Technology Bombay},
  addressline={Powai},
  city={Mumbai},
  postcode={400076},
  state={Maharashtra},
  country={India}
}

\affiliation[b]{
  organization={Indian Statistical Institute},
  addressline={Kolkata},
  postcode={700108},
  state={West Bengal},
  country={India}
}

\begin{abstract}
We consider (random) strategic interactions in a large population consisting of a variety of players. A rational player chooses actions that maximize certain utility functions, while a behavioral player chooses actions based on preferences such as avoid-the-crowd or follow-the-majority. We specifically study a turn-by-turn dynamic process in which players choose their actions sequentially and once; the utilities are realized either immediately or at the end of the game.

In the literature, such dynamical systems are often analyzed using an appropriate approximating ordinary differential equation (ODE). However, the  ODEs approximating the dynamics with pure actions are typically discontinuous. We adopt a differential inclusion (DI) based stochastic-approximation framework to derive the limiting analysis. The limits of the dynamics are characterized through the internally chain transitive (ICT) sets. We identify the presence of non-classical zeros as potential limits of the dynamics, a phenomenon not observed in classical settings involving continuous ODEs. These new limits arise precisely at the points of discontinuity of the dynamics. We further provide the conditions under which cyclic outcomes may occur at the limit.

Finally, we study a queuing game with differential priority-based services  and examine the impact of the proportions of avoid-the-crowd and two types of rational populations on the long-run outcomes of the  strategic interactions. We identify potential point limits and establish the possibility of cyclic outcomes for certain parameter configurations.
\end{abstract}
\end{frontmatter}

\section{Introduction}
The initial strands of the literature investigating strategic interactions among agents predominantly describe the outcomes via static concepts, mainly the Nash Equilibrium (\citep{nash1950equilibrium}).
In later developments, there has been a shift in focus toward the dynamic strategic interactions that more faithfully represent the real-world behavior (see e.g., \cite{sandholm2010population,webb2007game}). A substantial body of literature captures scenarios where finitely many agents continually learn each other’s strategies in pursuit of optimal utilities; for instance, when agents focus on the recent moves, we have best-response dynamics, and when they track a significant history of the same, the resultant is referred to as fictitious play (e.g., \cite{brown1951iterative,shapley1963some,binmore2001does,zusai2018tempered}). 
This process unfolds iteratively as the  players respond to the system (or the actions of others) continually  over time.
 
Another prominent strand of literature examines large-population games, where the agents adapt their strategies gradually over time. Here, the individuals revise their actions intermittently--reflecting inertia--and make decisions in a myopic fashion, relying primarily on the prevailing social state (empirical measure on  population choices) rather than long-term foresight. The resulting   evolution of population measures over strategies is often modeled as replicator dynamics (see, e.g., \citep{taylor1978ESS, hofbauer1998evolutionary, sandholm2010population, Mertikopoulosetal}). Basically,  in replicator dynamics,  the strategies with above-average payoffs become more prevalent, while less successful ones decline, in each type of population .

We also consider sequential decision-making scenarios  with a large number of agents choosing actions based on their types and current social state. However, our model is drastically different --- agents make their choices one after the other and only once --- this is referred to as turn-by-turn dynamics in \cite{turn_by_turn}, where the authors analyze a two-action game, see also~\cite{alon2012sequential}.  Such game dynamics have been relatively less studied in the literature, but have many important applications.
For instance, customer service lines like concert queues (where customers arrive sequentially and choose a queue or service level based on the prevailing occupancy), online reviews (where users make purchase or rating decisions influenced by earlier reviews), 
traffic routing (where drivers sequentially select routes based on observed congestion), or healthcare access (where patients select treatment options once, based on availability or perceived popularity). In all these examples, agents act once and move forward, capturing the essence of turn-by-turn-one-shot framework of this paper. Clearly, in this framework there is no learning or iterative refinement for any agent. We rather focus on the system level evolution driven by these one-shot decisions and the corresponding inferences related to the system.

The core contribution of our work begins with the identification of a differential inclusion (DI) that effectively captures the random dynamics of the game (using results of \cite{benaim2005stochastic}). Building on this, we further establish the following:

$\bullet$ We first show that the solutions of the DI approximate the  stochastic iterates of the empirical measures of population  choices (almost surely).  Further,  with probability one,  the limit set of the stochastic iterates is an  internally  chain transitive set (ICTs) of the DI. The remaining results are derived for linear utility functions.

$\bullet$ We  derived some classical and some non-classical Filippov solutions to characterize some important  `limit ICT sets'. 

  $\bullet$  For two-action games, singletons are the only ICTs and so have established  almost sure convergence to point limits. 
 
 $\bullet$ Interestingly, for games with more than two-actions (and any finite number of types), non-singleton ICTs can arise \\\hspace*{5mm} that  exhibit cyclic behavior. Besides, singleton ICTs can also emerge at discontinuities.

 $\bullet$ Strikingly, all identified singleton ICT sets are  among  multi-type mean-field Nash equilibria, thus the random dynamics converge to one among the equilibrium, when the limit set is such a singleton.
    
 $\bullet$ Finally,  we provide a numerical procedure  to identify the singleton ICT sets and verify the possibility of cyclic ICTs using some graph theoretic tools, for any given finite generic game (with finite types and actions).

$\bullet$ Lastly, a concert queue   is analyzed, where three types of agents choose among three priority-based service options, differentiated by offered prices. Such systems can have cyclic preferences, some agents may perceive higher-priced options as   better, while free entry can be preferred over the highest price  ---  cyclic ICTs exist for such games.  
Interestingly,  cycles exist even for smaller fraction of  cyclic-preference agents, in presence of   avoid-the-crowd population.


In contrast to our work, most existing studies on game dynamics (\cite{boone2019darwin,juneja2013concert,Cressman2014,Mertikopoulosetal,KraTar24}) focus on   deterministic processes to model the evolution of population measures; further they typically consider mixed strategies and directly analyze using  (continuous) ordinary differential equations (ODEs).
These frameworks often overlook the stochastic nature of real-world decision-making, particularly the randomness   inherent because of  random arrival of heterogeneous agents; furthermore, the limiting DI has discontinuities   due to realistic consideration of      pure-action choices (the real-world agents seldom choose according to mixed strategies). Authors in \cite{turn_by_turn} do consider   the random system evolution with pure action choices; however they have a restricted study   with only two actions and two types. They utilize some direct probabilistic arguments, while  we analyze any finite system  in complete generality  using  DI-based framework,   where the population measures evolve in $(k-1)$-dimensional spaces with $k$-actions.  Authors in \cite{zusai2018tempered,rothsandholm2013} do consider evolution because of pure action choices and analyze using DIs, 
however, they do not consider random evolution. 
\hide{we construct an aggregate population measure of strategies that evolves over time; our approach directly models the stochastic iterates of action measures evolution, preserving the randomness from agent types \rv{and action selection}. Basically, we use a differential inclusions (DI) based stochastic approximation framework to analyze the limiting behavior, which generalize ODEs to accommodate set-valued dynamics. It is essential because agents use pure strategies, i.e., they choose single action rather than a mixture --- introduces discontinuities in the evolution of the empirical distribution. This approach enables the analysis of such a significantly general games with any finite number of actions and types where population measures reside in higher dimensions. In contrast, \cite{turn_by_turn} focuses on a restricted setting with two actions and two types, i.e., a one dimension case, and relies on direct probabilistic arguments for the limiting analysis, whereas we adopt the DI-based approach.}

Seminal works such as \cite{benaim2005stochastic,benaim2006stochastic,benaim2012, rothsandholm2013} highlight the significance of the DI framework in understanding complex dynamical systems that arise in stochastic approximation.
In particular, \cite{benaim2005stochastic} shows that the limit set of  population measure iterates is an  ICT set,   almost surely.
In this paper, we extend this line of research by identifying the explicit structure of certain ICT sets (as discussed above), given the game under consideration. For example, we derive cyclic ICT sets in any  finite dimensional space, by constructing non-classical solutions of the DI, that pass through discontinuities (in DI) and  exhibit cyclic nature;  while works like \cite{brown1951iterative,shapley1963some,matsui1992best,samuelson1997evolutionary,binmore2001does,boone2019darwin}
dealing with mixed strategies derive the cyclic outcomes (like Shapley polygons and limit cycles around Nash equilibria)  for three-action games, or two dimensional space,  using the well-known  Poincaré–Bendixson theorem.  To the best of our knowledge, we have not come across any  literature that extensively characterizes the limit-sets of   DIs to this extent (e.g., identifying cycles  through discontinuities).

\hide{Similar to the cyclic ICT sets characterized in our work, several learning-based studies have reported the existence of cyclical behaviors --- such as Shapley polygons and limit cycles around Nash equilibria (shown through the Poincaré–Bendixson theorem) (see, for example \cite{brown1951iterative,shapley1963some,matsui1992best,samuelson1997evolutionary,binmore2001does,boone2019darwin}). However, this line of research predominantly focuses on mixed-strategy dynamics and, more importantly, overlooks the inherent randomness present in the underlying evaluations.}

Our work also relates to another branch of the literature \cite{sandholm2010population,gigerenzer2002bounded,simon1955behavioral,camerer2003behavioral} that studies games with bounded rational and/or behavioral  agents; such a consideration is important, as once again real-world agents seldom make completely rational decisions based on rigorous computations. We discuss  this more in  subsection \ref{sub_sec_some_examples_app}; for now, we would like to mention that our dynamics involve  populations spanning  fully rational, bounded-rational, and behavioral agents. 

\section{A multi-type population Game} \label{sec_general model}
We consider a large population consisting of multiple types of interacting players and study   
  their interactions in  a mean-field (MF) kind of framework. The interactions in our model are 
  significantly
  different from those in the  traditional MF scenarios, in terms of the types of agents, in terms of the involvement of the agents, and also in terms of the realization of the utilities. Towards a realistic consideration, we consider a mixture of rational and behavioral agents -- the former agents choose actions to maximize certain utility functions, while the latter category chooses actions simply based on some behavioral patterns \cite{agarwal2024balancingrationalitysocialinfluence}. Furthermore, we consider scenarios where the agents interact with the system just once, derive their utilities, and leave the system (e.g., concert queues, as in \cite{juneja2009concert,juneja2013concert}). Alternatively, our framework can also include the scenarios where the agents interact once, but receive their utilities at the end of the game (e.g., a participation game as in \cite{turn_by_turn}).  Before delving further into the details of the actual dynamics that describe the system, we begin with the static version of the framework that can predict the possible limits of the dynamic interactions.  We would later discuss the relation between the static notions of the equilibrium and the limits of the dynamical system. 
 
Formally, \textit{the multi-type population game  is represented by   $\langle  \Theta, \balpha, \{\A_\t\}_{\t \in \Theta},  (u_\t)_{\t \in \Theta}\rangle$,} 
with the components defined as follows:

\noindent$\bullet$ The set $\Theta := \{\t_1,\t_2, \ldots, \t_n\}$ represents the finite types of players. The players interact with players of their own type as well as with players of all other types.   

\noindent$\bullet$ The population distribution across types in $\Theta$  is given by $\balpha := (\alpha_\t)_{\t \in \Theta} \in \P (\Theta)$;  specifically $\alpha_\t$ represents the proportion of players of type $\t$ among the population. In general,   $\P ({\mathcal X})$ denotes the set of probability measures on~$\mathcal X$.
 
 \noindent$\bullet$  The set $\A_\t =\{1, \ldots, k_\t  \}$ represents the finite set of actions available  to  players of type $\t$. 
    
\noindent$\bullet$ The vector $\boldsymbol{\mu}_\theta \in \mathcal{P}(\mathcal{A}_\t)$ represents the empirical distribution of the players of type $\theta$ over the action set $\mathcal{A}_\t$. We denote by $\mu_\theta^a$ the fraction of players of type $\theta$ who choose action $a \in \mathcal{A}_\t$. Thus, $\bmu_\t = (\mu_\t^a )_{a\in \A_\t}$.

\noindent$\bullet$  The function $u_\t$ represents the utility of a typical player of type $\t$,  which may vary across different types.
        As in a mean field game, the utility functions   depend on the actions  chosen by the players and the mean field (more precisely the empirical distribution) of the actions chosen by other players.
       A typical type-$\t$ player   derives   utility  $u_\t (a, \bmu_\t,  {\bmu}_{-\t} )$, when it chooses action $a \in \A_\t$ and when the mean field  of its own type is $\bmu_\t$, and that of the other types is ${\bmu}_{-\t} = \{\bmu_{\t'} : \t' \ne \t \}$ \textit{(agents aim to maximize their utilities)}. In particular, we consider  games with special structure, where the dependency is on aggregate measure\footnote{If required, one can easily extend the analysis of this paper to more general   $u_\t (a, \bmu_\t,  {\bmu}_{-\t} )$ functions.}:
\begin{equation}
\label{Eqn_special_game}
u_\t (a, \bmu_\t,  {\bmu}_{-\t} ) =   u_\t (a, \bromga), \text{ for all } \t \in \Theta, \text{ where } \bromga  :=  \sum_{\at \in \Theta } \alpha_{\at} \bmu_{\at} . 
\end{equation}
In the above $\bromga \in \P ( \A) $ is an \textit{`aggregate measure'} on the combined set of actions   $\A := \cup_\t \A_\t$; for notational simplicity,  we set $\mu_\t (a) = 0$ if $a \notin \A_\t$.    Here, $\bromga = (\bromgac^a)_{a \in \A}$, where $\bromgac^a$ is the fraction of population across all types that choose action~$a$. Next, we define the equilibrium for such games, we also define the  aggregate-MFE (mean-field equilibrium).
\begin{definition}{\bf[MT-MFE and aggregate-MFE]}   \label{defn_MT-MFE}
        A tuple $(\bmu^*_{\t})_{\t \in \Theta}$, where $\bmu_\theta^* \in \mathcal{P}(\A_\t)$ for each $\theta \in \Theta$, is said to be a \textit{ multi-type mean-field Nash equilibrium} (MT-MFE) of the game  $\langle  \Theta, \balpha, \{\A_\t\}_{\t \in \Theta},  (u_\t)_{\t \in \Theta}, \bbeta \rangle$, if  for each $\t \in \Theta$, the following   holds: 
        $$
        \mathrm{supp}(\bmu^*_\t) \subseteq {\rm Arg} \max_{a \in \A_\t} u_\t (a, \bromga^* ), \text{ for all } \t \in \Theta, \text{  where } \mathrm{supp}(\bmu^*_\t) := \{a \in \A_\t : \mu_{\t}^{*a}   > 0\}.
        $$ 
       We refer to the aggregate measure $\bromga^*$  as the aggregate-MFE,  if there exists an  MT-MFE  $(\bmu^*_{\t})_{\t \in \Theta}$ such that  $\bromga^* = \sum_{\t} \alpha_\t \bmu^*_\t$.
    \end{definition}
\subsection{Some examples and applications}\label{sub_sec_some_examples_app}
Our multi-type setting provides a sufficiently general framework that can cover interactions among a variety of agents. The framework can consider rational agents who choose actions with an aim to optimize a certain (type-dependent) utility function  $u_\t$. For example, for some fraction of the population visiting a service facility, the utility could be the time spent in the facility, while for some others  the utility could be driven by economic factors, and for a third category, it could be the certainty with which the job is completed.

There has been a shift in thinking, in the recent game-theoretic literature, from considering completely rational agents to more realistic bounded-rational agents, or for that matter, agents with some behavioral patterns \cite{agarwal2024balancingrationalitysocialinfluence, turn_by_turn, camerer2011behavioral, banerjee1992simple}. 
Such agents may either approximate their utilities rather than compute them accurately (referred to as bounded-rational agents in \cite{sandholm2010population}) or may simply follow some behavioral instincts (e.g., herding agents in \cite{agarwal2024balancingrationalitysocialinfluence, turn_by_turn, banerjee1992simple, morone2008simple, eyster2009rational}). Our framework is general enough to accommodate populations spanning fully rational, bounded-rational, and behavioral agents.
\subsubsection*{\underline{Behavioral interactions}}
In settings involving a large population, one may notice people imitating others, aligning with the majority by following their choices, or deliberately avoiding the majority, etc.  

For example, say the population of type ($\t=\text{avoid}$) prefers the least considered choice.    One can typically notice such behavioral patterns  when people choose less crowded restaurants or tourist spots, etc.  Clearly, such agents can be incorporated into our framework by defining the following (fictitious) utility function:
\vspace{-2mm}
\begin{eqnarray}\label{eqn_avoid_utility}
     u_{\t } (a, \bromga) = - \bromgac^a \ind_{\{a \in \A_\t\}}, \text{ when } \t = \mbox{avoid}.
\end{eqnarray}
It is important to note here that \textit{the utility functions of the behavioral agents only  capture  the factors that drive their decisions/choices, but may not be indicative of their true utilities.} However, such a consideration is sufficient, as the focus of our paper is primarily on the long-run behavior of the system dynamics driven by the strategic/behavioral choices.

Similarly, one can consider the population of type $\t =\mbox{prefer}$, who simply prefer a particular action, say $\ta$. Again, one can notice such trends commonly, for example, some people  always prefer traditional methods. Such a population can be incorporated by considering the type-utility function as:
$ u_\t (a, \bromga) =  \ind_{\{a = \ta\}},  \mbox{ when } \t = \mbox{prefer}.$
\subsubsection*{\underline{Herding and $\alpha$-rational Nash equilibrium}}\label{Subsec_herding_explain}
One can include the herding crowd in the population as in \cite{agarwal2024balancingrationalitysocialinfluence,turn_by_turn}.  This situation can be incorporated by considering the   utility function:
$ u_\t  (a, \bromga) =  \bromgac^{a} \ind_{\{a \in \A_\t\}}, \text{ when } \t = \mbox{herd}.$

The MT-MFE definition \ref{defn_MT-MFE} with herding and rational population (with $n = 2$ )    coincides with the recently defined $\alpha$-rational  Nash equilibrium (referred as $\alpha$-RNE in \cite{agarwal2024balancingrationalitysocialinfluence}). One can view Definition \ref{defn_MT-MFE} as multi-type $\alpha$-RNE. 

\section{Turn-by-turn dynamics among multi-type players } \label{sec_turn_by_turn}
Inspired by recent work~\cite{turn_by_turn},
we consider  turn-by-turn dynamics within 
our framework of a multi-type population game,  discussed in   section \ref{sec_general model}. In this setting, the players arrive sequentially over time, and each player's type is sampled  (independently across time slots) according to $\balpha$. They choose their actions based on their types, current trends, or the empirical distribution of the choices of the previous players, and participate only once.  
The players can derive utilities immediately after making a choice or   after the game is over; the former category (among the non-behavioral agents) represents rational players, while  the players in the latter scenario can be interpreted as myopic rational players that typically make choices based on current trends without being far-sighted (see, e.g., \cite{sandholm2010population}). 
  The former modeling is applicable to instances like queuing or routing problems, where the customers or drivers receive negative utility  (in terms of the travel or waiting times) immediately after their choices. On the other hand, there are games  where the  utility of a player depends upon  collective decisions of all others and is received in the end; for instance, in vaccination drives, the benefit 
  arises only after a substantial portion of the population is vaccinated, and in crowdfunding campaigns, the utility depends on reaching a target, etc., (see \cite{turn_by_turn}). 

The main aim of the paper is to study the culmination of such game dynamics, as $T$, the number of players (or turns) increases to infinity. 
Such a study helps in deriving an approximate analysis for finite games, when $T$ is sufficiently large.  We next describe the precise details of the dynamics. Formally, a player of random type $F(t)$, arrives at a turn/decision epoch  $t>0$,  according to the probability distribution:
$$
P(F(t) = \theta) = \alpha_\t \quad \text{for all } \theta \in \Theta.
$$ 
The player chooses an action $A(t) \in \A$ based on their type and the (available) history of actions chosen by previous players, which we describe in the immediate next. Let     
\begin{eqnarray}
\label{eqn_barN}
\Bar{N}_\t(t) := \frac{ \sum_{s\le t} \ind_{\{F(t) = \t\}}}{t}, \mbox{ and }  \Rmuc_\t^a (t) := \frac{ \sum_{s\le t} \ind_{\{F(t) = \t, A(t) = a \}} } {t \Bar{N}_\t(t)}     
\end{eqnarray}
respectively represent the fraction of agents  of type $\t$ and the fraction of type $\t$ players that have chosen action $a$, up to time $t$. Let $ \Rmu_{\t} (t) := (\Rmuc_{\t}^a(t))_a$ (we again set $\Rmuc_{\t}^a(t) = 0$ for $a \notin \A_\t$).
 Let $\Romga(t)$  represent the aggregate (random) population measure, as defined in \eqref{Eqn_special_game},  now only because of the actions chosen by the agents till epoch $t$: 
\begin{eqnarray}
     \Romga (t)  &=&  \sum_{\t \in \Theta } \alpha_{\t} \Rmu_{\t} (t). \label{Eqn_Romga}
\end{eqnarray}
Every agent chooses an action that maximizes its type-dependent utility function. If the agent  is of type $\t$, i.e., if  $F(t+1) = \t$, then it chooses an action $A(t+1)$ that maximizes  the utility function $u_\t$: 
\begin{equation}\label{ratioanl_decision}
    A(t+1)  = \min \left \{   {\rm Arg} \max_{a \in \A_\t} u_\t (a, \Romga(t) )  \right \}. 
\end{equation}
In the above, the \textit{ties are broken by assuming that the players  choose the smallest index among the maximizers in~\eqref{ratioanl_decision}.} As already mentioned, the utilities are realized by the player either immediately or after all the players have chosen their actions, i.e., at the end of the game (and such differences have no impact on the game dynamics). 

Next, we describe in detail the evolution of the aggregate population measure $\Romga(t)$ defined in \eqref{Eqn_Romga}.  Observe that its $a$-th component $\Romgac^a(t)$ is the fraction  of players within the overall population who have chosen action $a$, up to time $t$, and is given by (see \eqref{eqn_barN}):
\begin{equation*}\label{eqn_mu_bar}
    \Romgac^a(t) := \sum_{\t \in \Theta} \Rmuc^a_\t(t) \Bar{N}_\t(t), \mbox{ for all } a \in  \A = \cup_\t \A_\t. 
  \end{equation*}
In other words, one can view $\Romga(t) = (\Romgac^a(t))_{a \in  \A}$ as the (random) state of the system at turn epoch $t$. This vector evolves according to the following update rule (one can assume any value for the initial terms like  $ \Romga(0)$):
\begin{eqnarray}\label{eqn_iterates_nu}
   \Romgac^a(t+1) = \Romgac^a(t) + \frac{1}{t+1}(\ind_{\{A(t+1) = a\}}-\Romgac^a(t)), \mbox{ for all } a \in \A  \text{ and } t>0. 
\label{EQN_ITERATES_NU}
\end{eqnarray}
The above iterative scheme aligns with a stochastic approximation (SA) scheme (e.g., \cite{borkar2008stochastic}). Generally, the asymptotic analysis of such iterates can be derived using the ordinary differential equation (ODE) (e.g., \cite{kushner2003stochastic}), 
\begin{eqnarray}\label{eqn_ode_new_foot}
    \dot{\bromgac}^a =  {g^a}(\bromga), \mbox{ for all } a \in  \A, \mbox{ with, } \ 
    g^a(\bromga) \ := \ \mathrm{E}[\ind_{\{A(t+1) = a\}}-\Romgac^a(t) \mid \bm \Gamma(t)], 
\end{eqnarray}
 defined   using the conditional expectation of $(\ind_{\{A(t+1) = a\}}-\Romgac^a(t))$ with respect to $\bm \Gamma(t)    := \sigma\{\Romga(s): s \le t\}$,
the filtration that captures the history of $\Romga(t)$.
Basically ${\bm\Gamma}(t)$   denotes the smallest $\sigma$-algebra generated by $\Romga(s)$ for all $s \le t$. It is easy to observe that:
\begin{eqnarray}
   \hspace{-1mm} g^a(\bromga) \hspace{-2.5mm} &:=& \hspace{-2.5mm}\mathrm{E } [\ind_{\{A(t+1) = a\}}-\Romgac^a(t)) \mid \bm \Gamma(t)],   \small{\text{ (when $\Romga(t) = \bromga$ with $\Romgac^a(t) = \bromgac^a$)}}, \label{eqn_cond_exp}\\
         \hspace{-2.5mm}&=&\hspace{-2.5mm} \ \sum_{\t \in \Theta}\mathrm{E}\left[\mathrm{E}\left[ \ind_{\{F(t+1) = \t \}}\ind_{\{A(t+1) = a\}}\mid F(t+1),  \bm\Gamma(t)  \right] \mid \bm\Gamma(t) \right] - \bromgac^a, \label{} \nonumber\\
              \hspace{-2mm}&=&\hspace{-2mm} \ \sum_{\t \in \Theta} \alpha_\t \ind_{\{a \in \A_{\t}^*(\bromga) \}}   - \bromgac^a , \mbox{ (from } \eqref{ratioanl_decision}), \mbox{ for all } a \in  \A, \label{final_g_nu_bar}
\end{eqnarray}
where $\A_{\t}^*(\bromga)\ := \left \{ \min \ {\rm Arg } \max_{\ta \in \A_\t}  u_\t (\ta, \bromga) \right \}$ (recall the ties are broken in the favor of the smallest index in  \eqref{ratioanl_decision})
and hence the corresponding ODE is: 
\begin{eqnarray}
\label{eqn_ODE}
    \dot{\bromgac}^a \ = \ g^a(\bromga) \ = \ \sum_{\t \in \Theta} \alpha_\t \ind_{\{a \in \A_{\t}^*(\bromga)\}} - \bromgac^a \ = \ 
     \sum_{\t \in \Theta} \alpha_\t \Pi_{\ta \in \mathcal{A}} \ind_{\{h_\t^{a,\ta}(\bromga) \geq 0  \}} - \bromgac^a,  \text{ for all } a \in   \A, \label{eqn_general_ode_1} 
\end{eqnarray}
where  the functions $h^{a,\ta}_{\t}: \D \to \mathbb{R}$ are defined for all $\t$ and $a\ne \ta$   by (let $k := | \  \A|$):
\begin{equation}\label{eqn_h_ij}
    h_{\t}^{a,\ta}(\bromga) := u_{\t}(a,\bromga)-u_{\t}(\ta,\bromga), \text{ and } \D := \left \{ \bromga \in [0,1]^k:  \sum_a \bromgac^a = 1 \right \}. 
\end{equation}
In \eqref{eqn_ODE}, for each action $a$, the ODE\footnote{Domain $\D$ is clearly invariant for the ODE, for example if $\bromgac^a = 0$ for some $a$ then the corresponding RHS of ODE is non-negative.} function  $g^a:\D \to \RR$. One can write \eqref{eqn_ODE}  in the vector form using the corresponding vector function  $\g  : \D \to \RR^k$,
\begin{equation}\label{eqn_ODE2}
\hspace{1mm}
    \dot{\bromga} = {\bm g}(\bromga), \mbox{ with, } \g(\cdot) := (g^a (\cdot) )_{a \in     \A }.
\end{equation}
Observe from \eqref{eqn_ODE} that the function $ {\bm g}(\cdot)$ is discontinuous  with respect to $\bromga$ in general; however   it  equals a continuous $\g(\bromga) = \b  - \bromga$, for some vector $\b$   in some open sets ($\b$ depends on  open set)   that are not intersected by  `border'  sets,
\begin{equation}\label{eqn_fun_H_theta}
 \{ \H^{a,\ta}_\t  \}_{a \ne \ta, \t,} \mbox{ where, }
    \H^{a,\ta}_\t := \{\bromga: h^{a, \ta}_\t (\bromga) = 0\} \mbox{ for each } (a, \ta) \in \A_\t \times \A_\t \mbox{ with } a \ne \ta \mbox{ and } \t \in \Theta.
\end{equation}

\begin{wrapfigure}{r}{0.32\textwidth}
\vspace{-5mm}
    \centering
 \hspace{0.5mm}\includegraphics[trim={7cm  3.5cm 5cm 5.5cm},clip,scale=0.34]{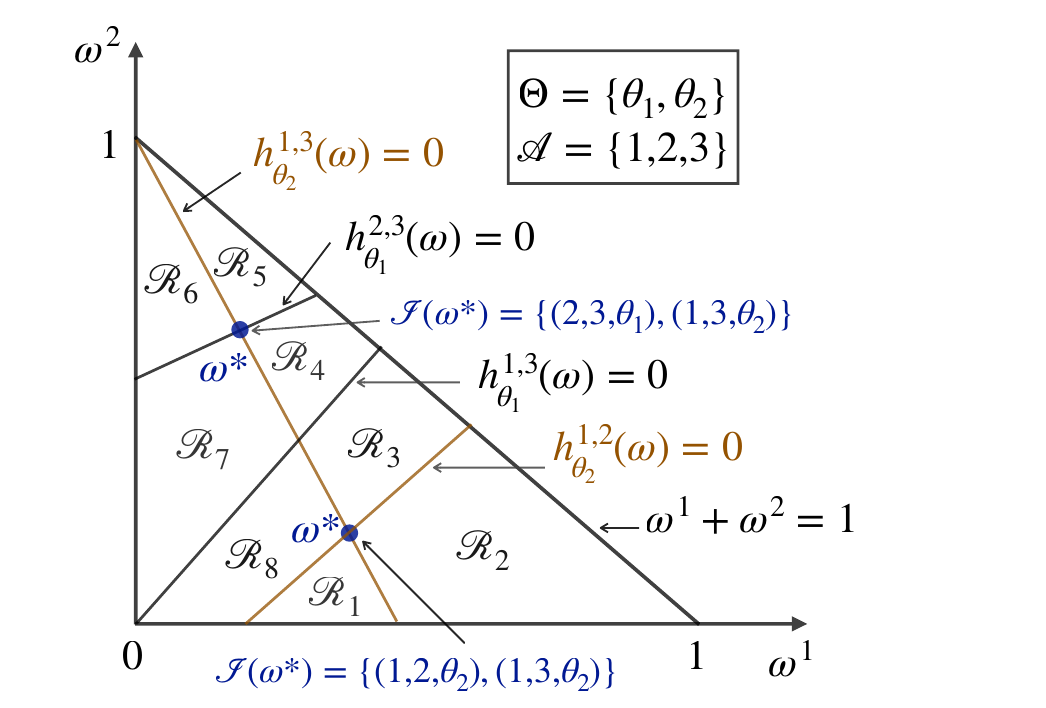}
    \captionsetup{type=figure}
    \vspace{-4mm}
    \caption{Representative picture for 2 types,  3 actions}
    \label{pictorial_regions_hfun}
     \vspace{-4mm}
\end{wrapfigure}

We have finitely many $\{\H^{a,\ta}_\t\}$ regions, as $n=|\Theta| < \infty$ and $k=|\cup_\t \A_\t| < \infty$. Further from \eqref{eqn_ODE}, the open regions  mentioned above (where $\g$ is continuous)  are bounded by some of the $\H^{a,\ta}_\t$ regions  --- 
\textit{basically in   each such region, agents of the same type always choose one particular action}. We represent these open sets\footnote{Each $\R_j$ is open in the domain $\D$ equipped with subspace topology from $\mathbb{R}^{k-1}$.} by $\{\R_j\}_{j \le m}$ and   \textit{assume $m < \infty$}, see the representative Figure
~\ref{pictorial_regions_hfun}.

Thus, one will have a partition of $\D = \H \cup_j \R_j$  with  $\H := \cup_{a,\ta, \t} \H^{a,\ta}_\t$ representing the `borders' region in $\D$, where some type-agents can be indifferent to a certain subset of actions. 
For clarity, we repeat that,  any type agent has one distinct `maximizer' action (unique maximizer in $\mbox{Arg} \max_a   u_\t(a, \bromga)$) at all $\bromga \in \D \setminus \H$,  and these actions  dictate the `region-specific'  $\b$ vectors mentioned earlier. 

As a result,  ODE \eqref{eqn_ODE2} can be discontinuous at $\bromga \in \H$ and hence the basic stochastic approximation results cannot be applied to study the stochastic iterates  \eqref{eqn_iterates_nu}. Therefore, one requires a more sophisticated framework to study the game dynamics \eqref{eqn_iterates_nu}, and we utilize a more general differential inclusion (DI)  based framework of \cite{benaim2005stochastic} to study the same. Towards this, we begin with identifying the DI associated with   ODE \eqref{eqn_ODE2}. 
Prior to that, we digress briefly to discuss an interesting application that can be modeled using the dynamics given in~\eqref{eqn_iterates_nu}.
\hide{
\begin{figure}[H]
\vspace{-5.7mm}
    \centering
    \begin{subfigure}[b]{0.35\textwidth}
        \centering
        \includegraphics[trim={0cm 11.4cm 10cm 11cm}, clip, scale=0.50]{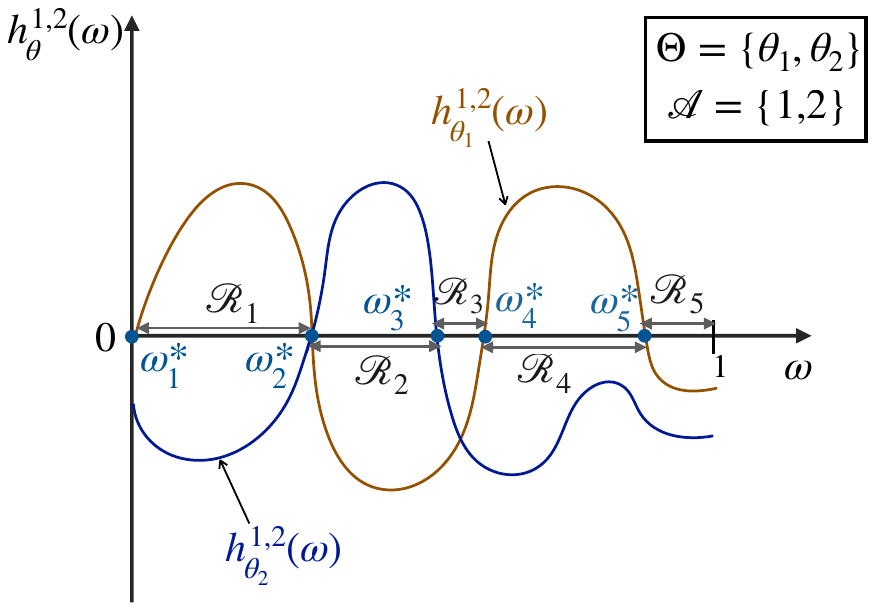}
            \vspace{-6mm}
        \caption{Two actions and two types}
        \label{fig:two_actions}
   \end{subfigure}
    \hfill
    \begin{subfigure}[b]{0.55\textwidth}
        \centering
        \includegraphics[trim={2.3cm 10cm 7cm 11.95cm}, clip, scale=0.50]{pictorial_regions_hfun.pdf}
            \vspace{-4mm}
        \caption{Three actions and two types}
        \label{fig:three_actions1}
    \end{subfigure}
    \vspace{-2mm}
    \caption{Representative pictures for different action and type configurations.}
    \label{DI interval figure}
    \vspace{-4mm}
\end{figure}}

\subsection{Queuing game with a variety of agents and   priorities-based services}\label{sub_sub_sec_queuing_application}
Queuing systems that cater to diverse customer preferences are commonly observed in various real-life scenarios, like visa application centers,   hospital appointment desks, etc. To study such scenarios, we consider a queuing system designed to accommodate customer preferences via three distinct types of queues. The first,  referred to as the standard queue, offers a cost-free option but involves potentially longer wait times due to higher congestion. The second, called the moderate queue, provides a balanced choice with reduced congestion cost reduced by a factor  $0<\rho < 1$, and at a moderate fee $p$. Finally, the third one is a high priority or premium queue that provides  maximum convenience ---  this eliminates the congestion  entirely but requires a highest premium fee $p_s$. Let $a=1$, $a=2$ and $a=3$ denote the choice of the standard, moderate and premium queues, respectively, i.e., the set of available actions is  $\A = \cup_\t \A_\t = \{1,2,3\}$.

To model and analyze the system more realistically, we consider a variety of customers. We begin  with the customers who make their choices directly based on the overall costs  of the three options provided by the system. This type of customer is referred to as 
  cost-rational (briefly  as $\CSr$-rational),  and their  utility function   is given by:
\begin{eqnarray}\label{eqn_queuing_game_utility_function}
   u_{\CSr}(a, \bromga) = 
   \left (  -\bromgac^1  \right ) \ind_{\{a = 1\}} +  
   \left (  -\rho \bromgac^2-p \right )  \ind_{\{a = 2\}} +  
   \left (  -p_s  \right ) \ind_{\{a = 3\}},  
\end{eqnarray}
where $\bromgac^i$ represents the aggregate congestion level in queue $i$ at the decision epoch. Let  $\t_1=\CSr$ be their type. 

Often one observes customers who make their choices by comparing with the choices of the others. Typically customers  prefer the moderate  over the standard queue, and the premium over the moderate queue,  because of  comfort reasons; customers can also rank  standard queue superior to premium queue for price considerations.  Thus the customers that compare with others and perceive the utility, for example while weighting the moderate queue option (or $a=2$),  will perceive positive utility by comparing with $\bromgac^1$  fraction  of people using standard queue and negative utility by comparing with $ \bromgac^3$ fraction of people using premium queue;  thus we model their utility by $\bromgac^1 - c \bromgac^3$, where $c$ is a parameter capturing the trade-off.  Perceived utilities  for other options  can be captured in a similar way,  and in all, one can capture the utility function of this type  of comparing $(\t_2 = \CMr$) agents  as  below:
\begin{eqnarray}
   u_{\CMr}(a, \bromga) = 
   \left (  \bromgac^3-\bromgac^2  \right ) \ind_{\{a = 1\}} +  
   \left ( \bromgac^1-c\bromgac^3 \right )  \ind_{\{a = 2\}} +  
   \left (  \bromgac^2-\bromgac^1  \right ) \ind_{\{a = 3\}}. \label{CMR}
\end{eqnarray}
 We view them also as rational agents with a different utility function, and refer them briefly as $\CMr$-rational agents. Since the premium queue costs a higher price, the trade-off factor $c \geq 1$ is applied only to $\bromgac^3$. Note that the utilities in \eqref{CMR} are inspired by the well-known rock-scissor-paper game (see \cite{narahari2014game}).

It is   not uncommon to observe in queuing systems that  the customer  choices depend upon some  inherent behavioral tendencies, as discussed in subsection \ref{sub_sec_some_examples_app}.  In particular,    some customers     avoid the congested queues;  thus we consider avoid-the-crowd   population as the third type with $\t_3=\nuu$, and their  utility function is given by \eqref{eqn_avoid_utility}. In all, we consider a queuing game with   $\Theta = \{\CSr, \CMr, \nuu\}$ as the set of customers types;  let $\alpha_\CSr,\alpha_\CMr$ and $\alpha_\nuu$ represent the respective fractions. 

In queuing systems, it is reasonable to assume that  the customers make their decisions sequentially, one after the other, and each makes a single choice; furthermore, only one customer decides at any given time; also  observe that the three utility functions   depend on the aggregate congestion vector $\bromga$. Hence, the above queuing system can be   analyzed using  the "aggregate mean-field" dependent turn by turn dynamics given in \eqref{eqn_iterates_nu}.  
 We study  this   game in subsection \ref{subsec_analysis_Queuingnetwork}, after analyzing the general game.

\section{ Differential inclusion (DI) based   stochastic approximation}\label{Sec_ident_DI}
The vector $\bromga = (\bromgac^a)_{a \in \A} $   belongs to $\mathbb{R}^k$ (as $|\A|=k$), however it satisfies    $\sum_{a} \bromgac^a = 1$. Thus, it is sufficient to analyze   the vectors in $\mathbb{R}^{k-1}$, by (say) replacing   $\bromgac^{k} = 1-\sum_{a  \ne k}\bromgac^a  $. With slight abuse of notation, we accordingly redefine the vector  $\bromga := (\bromgac^{1},\ldots,\bromgac^{k-1})$, the sets  $ \{ \H^{a,\ta}_\t  \}_{a \ne \ta, \t}$, domain $\D$ and the function $\bm{g}(\cdot)$. 
We begin with  a technical lemma which is  instrumental in identifying the differential inclusion (DI) that can approximate the dynamics \eqref{eqn_iterates_nu} (see   \ref{Appendix_DI}  for background on  DIs, in particular \eqref{DI_plane1};   and the  proof is in   \ref{proof_lemma_Filippov}):
\vspace{-0.1cm}
\begin{lem}\label{lemma_Filippov}
Consider a collection  of measurable sets 
  $\{ \C_{i\delta}\}_{i \le m}$ in   Euclidean space $\mathbb{R}^k$,    one for each $\delta \in (0, 1)$ and for some  finite integer $m \ge 2.$ 
Consider a collection of `centers' $ \{\g^c_i\}_{i\le m}$  that satisfy the following for each $i$ and $\delta$: (i)  $\g_i^c  \in \co \left ({ \cup_{i=1}^m \C_{i\delta} } \right ) $  for each $\delta$, where $\co$ denotes the closed convex hull, and  (ii)
  $\C_{i\delta} \subset \B_{i\delta}:=\B(\g^c_i, r_{i\delta})  = \{ \g: d(\g, \g^c_i) \le r_{i\delta}\} $ for some $r_{i\delta} < \infty$, where $d$ represents the euclidean distance.   
Further,  say $r_{i\delta} \to 0$ for each $i$ as $\delta \to 0$; then:
{\small$$\hspace{20MM}
\bigcap_{\delta} \co\left(\bigcup_{i=1}^{m}\C_{i\delta}\right) = \bigcap_{\delta} \co\left(\bigcup_{i=1}^{m} \B_{i\delta}\right) = \co \left( \g^c_1,\g^c_2,\ldots, \g^c_m \right).\hspace{40mm} \mbox{ \eop }
$$}
\end{lem}

\noindent Next, we study  the partition of the reduced domain $\D := \left \{ \bromga : \bromgac^a \ge 0~ \forall a , \sum_{a \ne k  } \bromgac^a \le 1 \right  \}$ of the function ${\bm g}(\cdot)$,   as a first step towards identifying the DI that can approximate dynamics \eqref{eqn_iterates_nu}.
Recall $\H= \cup_{a,\ta, \t} \H^{a,\ta}_\t$ is the set associated with the discontinuities in the function ${\bm g}(\cdot)$; also recall  $\H^{a,\ta}_\t$  in \eqref{eqn_fun_H_theta} is the border region, in the \nhd of which type-$\t$ players can switch between the decisions $a, \ta$  with a small change in aggregate measure $\bromga$ (if all other $a' \ne a,\ta$ are inferior for type-$\t$ population).  
As mentioned before, the domain $\D$ is  partitioned into   a finite number of open, connected and  disjoint regions denoted by $\R_{1}, \ldots, \R_{m}$, and $\H$:
\begin{equation}\label{eqn_calD}
    \D = \cup_{j=1}^{m} {\R}_j \cup \H.
\end{equation}
Note here, the common border between any pair of `adjacent' $\R_i,\R_j$ regions (i.e., those with intersection of closures, $\overline{\R}_i \cap \overline{\R}_j \ne \emptyset$)  lies in some  components of $\H$, see representative Figure \ref{pictorial_regions_hfun}. From  ODE \eqref{eqn_general_ode_1}-\eqref{eqn_ODE2},   there exist constants $\{\b_j\}_j$, such that  its  RHS  (right hand side)  has   the following special form for any  $\bromga \in \cup_j \R_j $: 
\begin{eqnarray}
   \g(\bromga) &= &  \sum_{j}   \b_j \ind_{\{\bromga \in \R_{j}\}}     -\bromga, \text{ where $\b_j = (b_j^1, \ldots, b_j^{k-1})$  for any } j , \text{ with, } \label{eqn_new_general_ode1}\\
   b_j^{a} &:=&  \sum_{\t } \alpha_\t  \prod_{\ta \ne a}\ind_{\{ h_\t^{a, \ta} ({\tilde \bromga}) > 0 \}}  \text{ for some }  {\tilde \bromga} \in \R_j,  \text{ and  all } a . \label{eqn_new_general_ode2}
\end{eqnarray}
In the above,  the definition of constant $\b_j$ is independent of   $ {\tilde \bromga} \in \R_j$ ---  each  $\R_j$   is a connected set on which none of the  $h^{a, \ta}_\t  $  functions   touch  zero (sign of each   $h^{a, \ta}_\t $ function remains the same   for all $ {\tilde \bromga} \in \R_j$, by continuity assumed in~\ref{A.2}).

We now set up the notations that would be instrumental in presenting the main results of the paper. 
\textit{Fix any $\bromga \in \H$ and all the discussions below are centered around it.} To begin with, define the following set of indices $\I (\bromga)$:
\begin{eqnarray}
 \I (\bromga) := \{ (a, \ta, \t) \in \A_\t \times \A_\t \times \Theta :a \ne \ta  \mbox{ and } h^{a,\ta}_\t  (\bromga) = 0  \}.    \label{Eqn_Iomega}
\end{eqnarray}
The cardinality   $|\I (\bromga)| $  represents the number of border curves intersecting 
 at point $\bromga$. Further, each such $\bromga$ is adjacent to  a sub-class of open regions $\{\R_j\}$, and   let,
\begin{eqnarray}
    \J(\bromga) := \{ j: \bromga \in \overline{\R}_j \}. \label{Eqn_Jomega}
\end{eqnarray}
Basically at  $\bromga$, there exist $|\J(\bromga)|$ regions  $\{\R_{j}\}_{j \in \J(\bromga)}$  among $\{\R_j\}_{j \le m}$ of \eqref{eqn_calD}, that intersect with closed balls $\B(\bromga, \delta)$  for all $0<\delta \le \bar \delta$ for some sufficiently small $\bar \delta > 0$, see representative Figure \ref{pictorial_regions_hfun}  (here $\B(\bromga, \delta) := \{\bromga' : d(\bromga, \bromga') \le \delta\}$).

For each $j_y \in \J(\bromga)$, choose any  sequence   $\{\bromga^y_{t}\}_{t \geq 1} \subset  \R_{j_y}$ that converges to $\bromga$.
 By \eqref{eqn_new_general_ode1},   the $ \lim_{t \to \infty} \g(\bromga^y_{t})$ exists and equals $ \b_{j_y} - \bromga$, which    is the same for any sequence in  $ \R_{j_y}$  that converges to $\bromga$. Let this limit be represented by:  
\begin{eqnarray}\label{eqn_g_limit}
    \g^\infty_y (\bromga ) :=
    \lim_{\substack{\bromga^y_{t} \to \bromga,   \bromga^y_{t} \in \R_{j_y}}} \g  ( \bromga^y_{t}) = \b_{j_y} - \bromga 
\end{eqnarray}
(observe here  one may   have     $\g^\infty_y (\bromga) \ne \g (\bromga)$, due to discontinuity at $\bromga$).
%
Thus, we have  exactly  $|\J(\bromga)|$ number of limits, $\{ \g_y^\infty(\bromga)\}_{j_y \in \J(\bromga)}.$  

Consider a small enough $\bar \delta > 0$ such that no   $h^{a,\ta}_\theta$ curve, other than those in $\I (\bromga)$, intersect the regions $\{\B(\bromga, \bar \delta) \cap \R_{j_y} \}_y$ and consider any $\delta \in (0, \bar \delta)$. Now, 
  define the following sets, one for each $j_y \in \J(\bromga)$:
  \begin{eqnarray}\label{eqn_c_j}
      \C_{j_y\delta}  := {\bm g}(\B(\bromga,\delta) \cap \R_{j_y}).
  \end{eqnarray}
Observe $\C_{j_y\delta}$ is 
 a   bounded set and  by definition in \eqref{eqn_g_limit}, we have 
 $\g^\infty_y \in  \overline{\C}_{j_y\delta}$, the closure;   hence,  there exists a radius $r_{y\delta}$  such that $\C_{j_y\delta} \subset \B(\g^\infty_y(\bromga), r_{y\delta})$. 
%
We now aim to apply Lemma \ref{lemma_Filippov}   for sets $\{\C_{j_y \delta}\}$  towards constructing the RHS of the desired DI, see \eqref{DI_plane1}. To proceed, we now choose an appropriate sequence of radii'   $r_{y\delta}$  that converge to $0$ as $\delta$ tends to $0$. Towards that, from \eqref{eqn_general_ode_1}, the function ${\bm g}(\cdot)$ is Lipschitz continuous in the region $\R_{j_y}$ with Lipschitz constant $1$:
$$
|{\bm g}(\bromga') - {\bm g}(\bromga'')| \leq |\bromga'-\bromga''| \leq \delta \mbox{ for all }  \bromga', \bromga'' \in \R_{j_y}.
$$
Observe that though $\bromga \notin \R_{j_y}$, it however belongs to the closure $\overline{{\R}}_{j_y}$ (by definition of $\H$ and as $\bromga\in \H$), and hence we have,    also see  \eqref{eqn_new_general_ode1} and \eqref{eqn_g_limit}: 
$
| {\bm g}_y^\infty(\bromga) - {\bm g}(\bromga')|  \leq \delta \mbox{ for all } 
\bromga' \in \R_{j_y}.
$

In other words, ${\bm g}(\bromga') \in  \B({\bm g}_y^\infty(\bromga),\delta)$ for all $\bromga' \in \R_{j_y} \cap \B(\bromga, \delta)$ implying  $\C_{j_y\delta} \subset \B(\g^\infty_y(\bromga), \delta)$. Thus, one can choose  $r_{y\delta} = \delta$ ensuring that $r_{y\delta} \to 0$ as $\delta \to 0.$ Finally, using Lemma \ref{lemma_Filippov} with   sets $\{\C_{\delta}\}_{j}$ as defined in \eqref{eqn_c_j}, the DI corresponding to   ODE \eqref{eqn_ODE2} is given by (see the definition of DI  \eqref{DI_plane1} in \ref{Appendix_DI}):
\begin{equation}\label{eqn_general_DI}
    \dot{\bromga}\in {\bm \G}(\bromga) = \begin{cases}
  \{  {\bm g}(\bromga) \}, & \mbox{ if } \bromga \in \D \setminus \H = \cup_j \R_j,\\

  \co\left \{{\bm g}^\infty_{1}(\bromga),\ldots,{\bm g}^\infty_{|\J(\bromga)|}(\bromga
  ), \g(\bromga) \right \}, &\mbox{ if } \bromga \in \H,
\end{cases}
\end{equation}
where ${\bm g}^\infty_y(\bromga) =  \b_{j_y} - \bromga$ is the limit defined in $\eqref{eqn_g_limit}$, and $\J (\bromga)$ is  defined in \eqref{Eqn_Jomega}.

{\bf  Assumptions:} We  require some  assumptions for the generic case (with any finite $k, n$), and will begin with stating four  of them; \textit{these are assumed throughout the paper}. The rest of them are stated, as and when   required.
\begin{enumerate}[label=\textbf{G.\arabic*}, ref=\textbf{G.\arabic*}]
    \item The function $\bromga \mapsto h_{\t}^{a,\ta}(\bromga)$, defined in \eqref{eqn_h_ij},  is continuous on $\D$ for all $(a,\ta, \t)$.
    \label{A.1}
    \item For every $\t \in \Theta$ and $a,\tilde{a} \in \A_\t$, the set $\H^{a,\ta}_\t$ is a curve/surface that divides the domain $\D$ into two disjoint open sets. We also assume $m < \infty$.
    \label{A.2}
  \item Assume $\b_j \notin \H$ for all $j \in \{1,\ldots,m\}$. \label{asm_bj_notin_regions}
   \item Assume $|\{ \bromga : |\I(\bromga)| > 1\}| < \infty.$ \label{asm_card_ind_set_finite}
\end{enumerate}
The continuity assumption is  considered typically in literature working with ODE-based techniques (e.g., \cite{turn_by_turn}), while the remaining assumptions are required for mathematical tractability. 

{\bf Stochastic approximation result:}
We now derive the asymptotic analysis of the   dynamics \eqref{eqn_iterates_nu}, under  \ref{A.1}-\ref{asm_card_ind_set_finite}. 
Recall that the ODE \eqref{eqn_ODE2} corresponding to  \eqref{eqn_iterates_nu} is discontinuous.  Thus,  we approximate the random trajectories of \eqref{eqn_iterates_nu} by the solutions of   DI \eqref{eqn_general_DI}, using   the results of \cite{benaim2005stochastic}.   Further analysis is obtained  by characterizing the limits and special limit-sets, called internally chain transitive (ICT) sets  of DI \eqref{eqn_general_DI}, see Definition \ref{Def_ICT_set} of Appendix \ref{Appendix_DI} . 

Define  the  time instants,  
$$
\tau(0):=0 \quad \text { and } \quad \tau(t) :=\sum_{i=1}^t \gamma(i) \quad \text { for all } t \geq 1,
$$
and consider the continuous-time affine interpolated process $\inu: \mathbb{R}_{+} \rightarrow \mathbb{R}^{k}$, with $\inu = (\inuc^a)_{a }$,  for \eqref{eqn_iterates_nu}  as below:
\begin{equation}\label{eqn_interpolated_traj}
    \inuc^a\left(\tau(t)+s\right) = \inuc^a(t) + s \frac{\inuc^a(t+1) - \inuc^a(t)}{\tau(t+1)-\tau(t)}, \quad s \in\left[0, \gamma(t+1)\right).
\end{equation}
Using \cite[Theorem 3.6, Propositions 1.3 and 1.4]{benaim2005stochastic}, and the definitions (perturbed solution and ICT set) from \ref{Appendix_DI}, we prove the following (proofs of this section are  in  \ref{proof_Thm_stoch_approx}).
\begin{thm}
{\bf [Stochastic approximation]}
\label{Thm_stoch_approx}
The following properties  hold almost surely for dynamics \eqref{eqn_iterates_nu}:
\begin{itemize}
    \item[(i)] The linear interpolated process $\inu (\cdot)$   of \eqref{eqn_interpolated_traj} is a bounded perturbed solution of  DI~\eqref{eqn_general_DI}. 
    \item[(ii)] The limit set  
  $L(\inu(\cdot) ):=\bigcap_{t \geq 0} \overline{\{\inu(s): s \geq t\}},$ is an ICT.  \eop
\end{itemize}
\end{thm}
We conclude this section, with a simple yet important consequence of the above theorem. 
\begin{lem}\label{lemma_singleton_ICT}
Let $\SS^*$ denote the sub-class of ICT  sets of   DI \eqref{eqn_general_DI}
that are singletons. Define the set of sample paths ($x$) of the dynamics \eqref{eqn_iterates_nu} whose limit sets are singleton ICT sets: 
\begin{eqnarray}
    \Omega:= \{ x :  L(\inu (x, \cdot))  \in \SS^*\} = \{ x :  L(\inu (x, \cdot)) = \left \{\bromga^*\}, \mbox{ where } \{\bromga^*\} \mbox{ is an ICT set} \right \}. \label{eqn_Omega_space}
\end{eqnarray}
For  all $x \in \Omega$, the  linear interpolated trajectory $\inu(x,t)$ as well as the original trajectory $\Romga(x,t)$ converge to $\bromga^*$ as $t\to \infty$, where  $\bromga^* \in L(\inu(x, \cdot))$. \eop
\end{lem}

In summary: (a) the interpolated process \eqref{eqn_interpolated_traj} forms  a perturbed solution for   DI \eqref{eqn_general_DI} --- \textit{implying the dynamics \eqref{eqn_iterates_nu} can be approximated by the solutions of the DI}; and (b)   the limit set of the perturbed solution and hence that of   \eqref{eqn_iterates_nu} is an ICT almost surely.
Thus, for further analysis, one would require to identify the ICT sets.  These sets are  like the equilibrium points or the limit-cycles of the classical ODEs; however they may not always stem from the `zeros' of the RHS of DI, like in classical  theory. They can instead result from some appropriately structured discontinuity in the RHS. \textit{One of the main results of this paper is to identify a sufficiently big class of such ICT sets which are relevant to study the asymptotic analysis of our behavior dynamics \eqref{eqn_iterates_nu}.} 
In particular,  the components of $\H$ will have a major role in this identification.

\section{Asymptotic analysis of \eqref{eqn_iterates_nu} via characterization of ICT sets}\label{sec_char_ICT_set}
We now consider asymptotic analysis of \eqref{eqn_iterates_nu}. Recently, the authors in \cite{turn_by_turn} analyzed similar dynamics with two actions and two types. 
\hide{
have analyzed turn-by-turn dynamics like in \eqref{eqn_iterates_nu}, for a special case with two actions $\A = \{1, 2\}$;  they consider two types of population,   the myopic rational population and    the herding crowd as described in subsection~\ref{Subsec_herding_explain}.}
In subsection \ref{sec_game_two_actions}, with two actions,  we significantly extend the analysis of \cite{turn_by_turn}  by considering a multi-type population with any finite number ($n$) of  types.  %
%
Later in subsection \ref{subsec_generic_results},  we further generalize to   consider any finite number $(k)$ of actions and illustrate the possibility of cycles at limit. 
%
%
Subsection \ref{subsec_game_three_actions} contains additional  results with $|\A|=3$, while 
in  subsection \ref{subsec_connect_ICT_MT_MFE}, the  singleton limits  are shown to be  aggregate-MFEs   of   Definition  \ref{defn_MT-MFE}. 

\subsection{Games with two actions}
\label{sec_game_two_actions}
In this case, the empirical distribution over actions is represented originally by $\bromga = (\bromgac^1,\bromgac^2)$ and in the reduced space by $\bromga = \bromgac^1$ (recall that one can replace $\bromgac^2 = 1- \bromgac^1$ to obtain a reduced space). Hence, the reduced domain is $\D=[0,1]$.
\textit{For simplicity in the analysis, we denote $\bromgac^1$ by $\bromgac$, and this notation is applicable only for this subsection}. We obtain almost sure convergence of the dynamics \eqref{eqn_iterates_nu} using Theorem \ref{Thm_stoch_approx} and Lemma \ref{lemma_singleton_ICT}. 

The ODE \eqref{eqn_general_ode_1} for this special case  simplifies to:
\begin{equation} \label{eqn_ODE_two_action}
     \dot{\bromgac} = \g(\bromgac) = \sum_{\t \in \Theta}\alpha_\t \ind_{\{h_{\t}^{1,2}(\bromgac) \geq 0\}} -\bromgac. 
\end{equation}
\begin{wrapfigure}{r}{0.45\textwidth}
\vspace{-4mm}
    \centering
   \includegraphics[trim={5cm  5cm 7.3cm 5.5cm},clip,scale=0.38]{Two_action_fun_h.pdf}
    \captionsetup{type=figure}
    \vspace{-6mm}
\caption{Representative picture for 2 types and   actions; here $\bromgac^*_1 = 0$, so $\R_0 = \emptyset$.}
    \label{fig:two_actions}
    \vspace{-6mm}
\end{wrapfigure}
We now obtain the simplified form of   DI \eqref{eqn_general_DI} and the corresponding ICT sets  which helps us in deriving the analysis with two actions. Towards this, first observe that the only possible pair of actions $(a,\ta)$ with $a\ne\ta$ defining the set of indices $\I(\bromgac)$  defined in \eqref{Eqn_Iomega} is $(1, 2)$ and this is true for any $\bromgac  \in \H = \cup_\t \H^{1,2}_\t = \cup_\t \{\bromga: h^{1, 2}_\t (\bromga) = 0\}. $
Hence,  \textit{the set of indices    $\I(\bromgac)$  can now be viewed as a subset of $\Theta$.}
 We require \textit{the following assumption specific to the case with two actions or equivalently the one-dimensional case} (like in~\cite{turn_by_turn}):
\begin{enumerate}[label=\textbf{O.1}, ref=\textbf{O.1}] 
 \item Assume $|\H| < \infty$.  
 \label{O.1} 
\end{enumerate}
Let $ \vartheta  :=|\H|$ and let   $ \H = \{\bromgac^*_1,   \cdots, \bromgac^*_\vartheta\}$  be arranged in increasing order, see Figure \ref{fig:two_actions}. Thus for the case  with two actions,   one can represent the $\{\R_j\}$ regions as below and as depicted in Figure \ref{fig:two_actions}:  
$$
\R_0 = [0, \bromgac^*_1),  \  \R_{\vartheta} = (\bromgac^*_\vartheta, 1] \mbox{ and }  \R_{i} := (\bromgac^*_{i}, \bromgac^*_{i+1}), \mbox{ for other } 1 < i \leq \vartheta.
$$
In the above, regions $\R_0$ and $\R_{\vartheta}$ are  empty respectively if  $\bromgac^*_1 = 0$ and $\bromgac^*_\vartheta = 1$; thus the number  $m$ of disjoint $\R_j$ regions is upper bounded by, $m \le \vartheta+1$. 
Now, consider any $\bromgac_j^* \in \H$ with $\bromgac^*_j \notin \{0,1\}$; then  $\I(\bromgac^*_j)$ of \eqref{Eqn_Iomega}, the subset of types of the population  that toggle their decisions around  $\bromgac^*_j $  is given by:
$
\I(\bromgac_j^*) = \{ \t \in \Theta: h^{1,2}_\t (\bromgac_j^*) = 0\}.
$
For such $\bromgac_j^*$, under the assumption \ref{O.1}, with $\R_{j_1}  := \R_{j-1}$ and $\R_{j_2} := \R_{j}  $ one can partition the set  $\I(\bromgac_j^*)$ as below (by  \ref{O.1}   the  function $h^{1,2}_\t$  for any $\t$ touches $0$  only at finitely many    $\bromgac \in \H$,  also recall each $h^{1,2}_\t$ is continuous):
\begin{eqnarray}    
\I^-(\bromgac_j^*) &=& \{\t \in \I(\bromgac_j^*):  h^{1,2}_{\t}(\bromgac_j^*) > 0 \mbox{ in } \R_{j_1}, h^{1,2}_{\t}(\bromgac_j^*) < 0 \mbox{ in } \R_{j_2}
\}, \label{eqn_I_minus_non_st} \\
\I^+(\bromgac_j^*) &=& \{\t \in \I(\bromgac_j^*):   h^{1,2}_{\t}(\bromgac_j^*) < 0 \mbox{ in } \R_{j_1}, h^{1,2}_{\t}(\bromgac_j^*) > 0 \mbox{ in }\R_{j_2}\}.\label{eqn_I_plus_non_st}
\end{eqnarray}
Also for each $j$, $\R_{j_1} $ and   $\R_{j_2}$ are two open (in subspace topology on $[0,1]$), connected and disjoint intervals. In case $\bromgac_1^*=0 \mbox{ or } $  $\bromgac_\vartheta^* = 1, $   we have only one open connected region   that surrounds it,   which is $\R_{1_2} = \R_1 = (0,\bromgac_{2}^*)$ or $\R_{\vartheta_1} =\R_{\vartheta-1}=(\bromgac_{\vartheta-1}^*,1)$ respectively; and the sets $\I^-(\bromgac_j^*)$ or $ \I^+(\bromgac_j^*)$, with $j=1$ or $j=\vartheta$ respectively,  are now given by: 
\begin{eqnarray}    
\I^+(0) &=& \{\t \in \I(\bromgac_j^*):  h^{1,2}_{\t}(\bromgac_j^*) > 0 \mbox{ in } \R_{j_2}
\},\ \ \ \
\I^-(0) \hspace{1mm}=\hspace{1mm} \emptyset, \label{eqn_I_minus_st} 
\\
\I^-(1) &=& \{\t \in \I(\bromgac_j^*):  h^{1,2}_{\t}(\bromgac_j^*) > 0 \mbox{ in } \R_{j_1}
\}, \ \ \ \
\I^+(1) \hspace{1mm}=\hspace{1mm} 
\emptyset.  \label{eqn_I_plus_st}
\end{eqnarray}
 Further, the RHS of the ODE \eqref{eqn_ODE_two_action}, for $\bromgac\in \R_{j_1} \cup \R_{j_2} $ equals:
\begin{eqnarray}  \label{eqn_fun_g_}
   \g(\bromgac) &=&  b_{j_1} \ind_{\{\bromgac \in \R_{j_1}\}} +  b_{j_2}   \ind_{\{\bromgac \in \R_{j_2}\}} -\bromgac, \text{ where }\\
  b_{j_1}  &:=& \sum_{\t \in \Theta\setminus \I(\bromgac^*_j) }  \alpha_\t \ind_{\{h_{\t}^{1,2}(\bromgac^*_j) > 0\}} + \sum_{\t  \in \I^-(\bromgac^*_j)}\alpha_\t,    
 \nonumber \ \  \\
  b_{j_2}  &:=&   \sum_{\t \in \Theta\setminus \I(\bromgac^*_j) } \alpha_\t \ind_{\{h_{\t}^{1,2}(\bromgac^*_j) > 0\}} 
  +   \sum_{\t \in   \I^+(\bromgac^*_j)}\alpha_\t.     \nonumber
\end{eqnarray}
The function $\g(\cdot)$ in \eqref{eqn_fun_g_} is  linear and continuous in the regions/intervals $\R_{j_1}$ and $\R_{j_2}$. Thus, the limits in \eqref{eqn_g_limit}, $\g^\infty_{j_1}(\bromgac_j^*)$ (left hand) and $\g^\infty_{j_2}(\bromgac_j^*)$ (right hand) of $\bromgac_j^*$ are as below: 
\begin{eqnarray}
    \g^\infty_{j_1}(\bromgac_j^*) &=& \g({\bromgac_j^*}^-)  := \lim_{\bromgac\uparrow \bromgac^*_j} \g(\bromgac) = \sum_{\t \notin \I(\bromgac^*_j)} \alpha_\t \ind_{\{h_{\t}^{1,2}(\bromgac^*_j) > 0\}} + \sum_{\t \in \I^-(\bromgac^*_j)} \alpha_\t - \bromgac^*_j, \label{eqn_limit_two_actions_di1}\\
    \g^\infty_{j_2}(\bromgac_j^*) &=& \g({\bromgac_j^*}^+)  := \lim_{\bromgac\downarrow \bromgac^*_j} \g(\bromgac) =  \sum_{\t \notin \I(\bromgac^*_j)} \alpha_\t \ind_{\{h_{\t}^{1,2}(\bromgac^*_j) > 0\}} + \sum_{\t \in \I^+(\bromgac^*_j)} \alpha_\t - \bromgac^*_j. \label{eqn_limit_two_actions_di2}
\end{eqnarray}
 Hence, the DI \eqref{eqn_general_DI} for the special case with two actions,  constructed using \eqref{eqn_fun_g_}-\eqref{eqn_limit_two_actions_di2}, is given by (where, for e.g.,  the convex hull $\co(a,b,c) := \left [\min\{a,b,c\}, \max\{a, b,c\} \right ]$ is an interval):
\begin{equation}\label{eqn_DI_with_two_actions}
    \dot{\bromgac}\in {\bm \G}(\bromgac) = \begin{cases}
  \{  {\bm g}(\bromgac) \}, & \mbox{ if } \bromgac \in \D \setminus \H,\\
\co( \g^\infty_{j_1}(\bromgac_j^*), \g(\bromgac^*_j), \g^\infty_{j_2}(\bromgac_j^*)), &\mbox{ for all } \bromgac=\bromgac_j^* \in \H \setminus \{0,1\} \\
\co(  \g(\bromgac^*_j), \g^\infty_{j_2}(\bromgac_j^*)), &\mbox{ if } \bromgac=0 \in \H  \\
\co( \g^\infty_{j_1}(\bromgac_j^*), \g(\bromgac^*_j) ), &\mbox{ if } \bromgac=1 \in \H   .
\end{cases}
\end{equation}
We now have the following result related to the dynamics  \eqref{eqn_iterates_nu} for   games with two actions (proof is in \ref{proof_lem_alm_sure_conv_two_act}).
 \begin{thm}{\bf [Almost sure convergence]}
    \label{lem_alm_sure_conv_two_act}
Consider $\A = \{1, 2\}$ and assume  \ref{O.1}. Then DI~\eqref{eqn_DI_with_two_actions} is not associated with any   non-singleton ICT set.  Further, the class $\SS^*$ of singleton ICT sets of Lemma \ref{lemma_singleton_ICT}   equals:
\begin{equation}
    \label{Eqn_ICTs_Two_actions}
\SS^* =    \left \{ \{\bromgac^*\} : \bromgac^* \in \mathbb{L} \right \} \mbox{, with }
\mathbb{L} = \left \{  \bromgac^* \notin \H:  \g(\bromgac^*) = 0  \right \} \cup \left \{   \bromgac^* \in \H :    0 \in   {\bm \G}(\bromgac^*)     \right  \}.
\end{equation}
Moreover, $|\SS^*| < \infty$ and  the    iterates in \eqref{eqn_iterates_nu} converge:  \ 
$\Romgac(t) \to \bromgac^* \text{ as }  t \to \infty, \text{ for some }   \bromgac^*  \in \mathbb{L}$ almost surely. \eop 
\end{thm}
\subsection{Generic games with finite actions} \label{subsec_generic_results}
We now consider any general game as in section \ref{sec_general model} with $k < \infty $   actions and $n<\infty$   types and primarily identify the ICT sets. In this subsection, we illustrate that games with more than two actions can have non-singleton ICT sets, and thereby establish the possibility of `cyclic' limits for dynamics \eqref{eqn_iterates_nu}.  

Recall, any region $ \R_j$ defined in section \ref{Sec_ident_DI} is the sub-region of the domain $\D$ in which each type of population chooses one particular action.  This `type-wise' action (of at maximum one type of population)  changes only when a solution/dynamics crosses over to a new $\R_i$ region, after crossing across some parts of the border set $\H \cap \overline{\R}_j$. 
  We begin with a first result that rules out the possibility of a non-singleton ICT set that lies completely inside the closure of one of these regions (with at least one point in the interior);   all proofs of this subsection are in \ref{subsec_append_Generic_game_proofs}. 
\begin{thm}
\label{Thm_gen_one}
Consider a set $\F \subset \overline{\R}_j$ for some $j \le m$, where $|\F| > 1$ and   $\F   \cap \R_j   \ne \emptyset$. Then $\F$ is not an ICT set.~\eop
\end{thm}

\noindent{\bf Remarks:}
The above Theorem leads to some important observations:
\begin{itemize}
    \item[(i)] Any ICT set $\F$ is either a singleton ($\F \in \SS^*$, the class of singleton ICT sets defined in Lemma \ref{lemma_singleton_ICT}) or a subset of $\H$ ($\F \subset \H$), or may span across multiple $\R_j$ regions.
   \item[(ii)] Consider any singleton ICT set $\{\bromga^*\}$ with  $\bromga^* \in \cup_j \R_j$; then $\bromga^*$  is a classical attractor that is both asymptotically and locally stable (see the proof of Theorem \ref{Thm_gen_one}). Specifically,   $\b_j$  for some $j$, is a classical attractor,  if $\b_j \in \R_j$. Thus the number of singleton  ICT sets in $\cup_j \R_j$  is upper bounded by $m$. 
\end{itemize}
 
We require the following additional assumption/condition for some results, while the   negation of the same implies some  other results, provided later (Lemma \ref{lem_three_action_on_H}); also note that the first condition is for a specific tuple $(a,\ta,\t)$.

\begin{enumerate}
[label={\textbf{G.5}}$(^{a, \ta}_\t)$, ref=\textbf{G.5}$(^{a, \ta}_\t)$, leftmargin=4.4em]  
    \item  \textit{Let  $\nablah$ represent the  gradient of the function  $h^{a, \ta}_\t(\cdot)$  at $\bromga$}. For  all $\bromga \in \H^{a,\ta} _\t   \cap \{\bromga': |\J (\bromga')| =2\}$,    assume
$ 
  (\nablah^T \cdot (\b_j-  \bromga)) (\nablah^T \cdot (\b_i -\bromga))     > 0 \mbox{ for  }  i,j \in \J (\bromga) .
  $  
\label{Asm_filippov_new} 
\end{enumerate}

\begin{enumerate}[label=\textbf{G.5}, ref=\textbf{G.5}]
\item The assumption \ref{Asm_filippov_new} is satisfied for all possible $(a, \ta, \t)$.\label{asm_gen_2_implication}
\end{enumerate}
Recall  $\J(\bromga)$ defined in \eqref{Eqn_Jomega} represents the  sub-class of open regions $\{\R_j\}$  that are adjacent to $\bromga$.  We have  $|\J(\bromga) | =2$,  for all  $\bromga \in \H^{a,\ta}_\t \setminus \{ \bromga : |\I(\bromga)| > 1\} $, and by assumption \ref{asm_card_ind_set_finite} the number of exceptions (with $|\J(\bromga) | >2$) is finite. We now prove a result related to the existence of a form of solution, which is instrumental in deriving the main results.  We refer to a solution as\textit{ constant solution if $\bromga(t) = \bromga$ for all $t$ for some $\bromga \in \D$}, and others as non-constant solutions. 

\begin{lem} {\bf [Solution only touches non-attracting (NA) border]}
\label{Lemma_no_Solution_on_H} 
Assume 
\ref{Asm_filippov_new} for some $(a, \ta, \t)$. 
Consider any non-constant solution $\bromga(t)$ of the DI \eqref{eqn_general_DI}. Then, there does not exist a time interval $[s_1, s_2]$  with $s_2 > s_1$ such that  
   $ \bromga(t) \in \H^{a,\ta}_\t$ for all $ t \in [s_1, s_2]$. \eop 
\end{lem}
Thus, under \ref{asm_gen_2_implication}, the DI has no solutions that live  on the border set $\H$ for a  non-trivial interval of time.

{\bf Linear Utilities:} 
We obtain the rest of the results for a special case with linear boundary $\{h^{a,\ta}_\t\}$ functions. There are many games that fall under such a category, for example the queuing game of
 subsection \ref{sub_sub_sec_queuing_application}; we believe one can extend    majority of the results  to non-linear boundary    functions, however, that would be considered as a part of future research. 
\textit{For the rest of the paper, we consider affine linear   functions as below (for any $(a,\ta, \t)$)}:
\begin{eqnarray}
\label{Eqn_h_linear}
    h^{a,\ta}_\t(\bromga) :=  \bromgac^1- \sum_{a'= 2}^{k-1} \eta^{a,\ta, a'}_\t  \bromgac^{a'} +  c^{a,\ta}_\t,   \mbox{  for some constants } \left\{\eta^{a,\ta, a'}_\t  \right\}_{a'\ne 1} \text{ and }  c^{a,\ta}_\t. 
\end{eqnarray}
By simple algebra, the gradient term in \ref{Asm_filippov_new} simplifies to the following and we also have for any $ \lambda \in [0,1]$: 
\begin{eqnarray}
    (\nablah^T \cdot (\b_j-  \bromga)) (\nablah^T \cdot (\b_i -\bromga)) &=& h^{a,\ta}_\t(\b_j) \nonumber h^{a,\ta}_\t(\b_i), \text{ and } \\
     h^{a,\ta}_\t(\bromga_1 \lambda +\bromga_2 (1-\lambda) )  &=& \lambda h^{a,\ta}_\t(\bromga_1) +(1-\lambda) h^{a,\ta}_\t(\bromga_2).
    \label{Eqn_linear_h_satisfies}
\end{eqnarray}

\subsubsection{\underline{Solutions of DI \eqref{eqn_general_DI}}}
In this subsection, we will derive some important solutions of DI \eqref{eqn_general_DI}, which would be instrumental in further analysis.
 As already mentioned, we now consider affine linear $\{h^{a,\ta}_\t\}$ functions as in \eqref{Eqn_h_linear}-\eqref{Eqn_linear_h_satisfies}. With this, \textit{the open region $\R_j$ is convex for each $j$.}  
To begin with, 
if $\b_j \in \R_j$ for some $ 1 \le j\le m$ and initial condition $\bromga(0) \in \R_j$, then the unique solution of the DI \eqref{eqn_general_DI} is as below:
\begin{eqnarray}
    \bromga(t) = (\bromga(0) - \b_{j})e^{-t} + \b_{j} \mbox{ for all } t \geq 0. \label{Eqn_DI_Solution_in_Rj}
\end{eqnarray}
For any $t$, the solution $\bromga(t)$ is a convex combination of $\bromga(0)$ and $\b_j$, and hence the solution remains in $\R_j$ for all $t \ge 0$, 
by convexity of $\R_j$. 
Clearly such a  $\b_j$ is an attractor: it is locally and  asymptotically stable in the classical sense (see also proof of Theorem \ref{Thm_gen_one}). Let the set of such \textit{classical attractors} and the corresponding sub-class of ICT sets  be denoted by $\LL_c^*$ and $\SS_c^*$, respectively:
\begin{eqnarray}\label{eqn_set_AA}
\LL_c^* := \{ \b_j :  j \in \AA\}  \text { and } \SS_c^* :=  \{ \{\b_j\} : j\in\AA\},   \text{ where }   \AA:= \{ j :  \b_j \in \R_j\}.   
\end{eqnarray}
Also observe,    $\SS^*_c  \subset \SS^*$, the class of singleton ICT sets defined in Lemma \ref{lemma_singleton_ICT}.
 Now, consider any initial condition  $\bromga(0) \in \R_{j_1}$ for some $j_1 \notin \AA$.  Again clearly,  
 there exists a unique solution that starts from $\bromga(0)$, as in \eqref{Eqn_DI_Solution_in_Rj}; however, now the solution is not valid for all $t$. More precisely, the solution is exactly as in \eqref{Eqn_DI_Solution_in_Rj}, but only for   $t \le \tau_{j_1} $ where exit time, 
\begin{eqnarray}\label{eqn_hititing_time1}
     \tau_{j_1} = \tau_{j_1} (\bromga(0)) :=  \inf \left \{ t : (\bromga(0) - \b_{j_1})e^{-t} + \b_{j_1}  \notin \R_{j_1} \right \}. 
 \end{eqnarray}
By boundedness of $\R_{j_1}$ and as $\b_{j_1} \notin \R_{j_1}$, we have
$\tau_{j_1} < \infty$. By continuity of  solution $\bromga(\cdot)$,  $\bromga_{j_1} := \bromga(\tau_{j_1})    \in \H \cap  {\overline \R}_{j_1}$. 
Say $ | \J(\bromga _{j_1}) |= 2$ with $\J(\bromga _{j_1}) = \{j_1, j_2\}$.  Extend  the function  $\bromga(t)$ beyond $t > \tau_1$,  as below (with $\tau_{j_0}=0, \  \bromga_{j_0} := \bromga(0)$):
\begin{eqnarray}\label{eqn_bromega_in_regions}
    \bromga(t) =  \sum_{i=1}^{2} \left ( (\bromga_{j_{i-1}} - \b_{j_{i}}) e^{-(t- \tau_{j_{i-1}})}  +  \b_{j_{i}} \right ) \ind_{\{  t \in [\tau_{j_{i-1}},\tau_{j_{i}}] \} },   \text{ for some $\tau_{j_2} > \tau_{j_1}$}. 
\label{Eqn_solution_two_regions} 
\end{eqnarray}
We will first show that $\bromga(t) \in \R_{j_2}$ for all $t\in (\tau_{j_1}, \tau_{j_2})$ and identify an appropriate $\tau_{j_2}$ in  the following, to later prove that \eqref{eqn_bromega_in_regions} is also a solution of the DI \eqref{eqn_general_DI}:

\begin{lem}{\bf [Solution through NA border]} \label{lem_solu_tau1_tau2} Assume \ref{asm_gen_2_implication} and that $\J(\bromga_{j_1}) = \{j_1,j_2\}$.
Consider  the function $\bromga(\cdot)$ defined in \eqref{Eqn_solution_two_regions}, then we have   $\bromga(t) \in \R_{j_2}$  for all $t \in (\tau_{j_1}, \tau_{j_2}) $, where:
    \begin{eqnarray}
    \label{Eqn_tau_two}
  \tau_{j_2} = \tau_{j_2} (\bromga(0)) :=  \inf \left \{  t > \tau_{j_1} :  \bromga(t) \notin \R_{j_2}   \right \}, \mbox{ with }  \tau_{j_2} = \infty \mbox{ when } \bromga (t) \in \R_{j_2}  \ \forall \ t  > \tau_{j_1}. \hspace{2mm} \mbox{\eop}
\end{eqnarray}
\end{lem}
Now, from the above Lemma, we have that $\bromga(t) \in  \R_{j_2} $ for all $t \in (\tau_{j_1}, \tau_{j_2})$. Thus, the derivative of $\bromga (\cdot)$ or the function in  \eqref{Eqn_solution_two_regions}  satisfies the DI \eqref{eqn_general_DI}, even for all $t \in (\tau_{j_1}, \tau_{j_2})$;  moreover, $\bromga(\cdot)$ is (trivially) absolutely continuous.
Thus, $\bromga (\cdot)$ is a solution of the DI \eqref{eqn_general_DI} for initial condition $\bromga_{j_0} \in \R_{j_1}$ and for all $t \le \tau_{j_2}$.   By 
continuing  similarly, a  solution 
of~\eqref{eqn_general_DI} with initial condition $\bromga(0)$ and  for any given time interval  $[0, \tau$], under \ref{asm_gen_2_implication}, is given by: 
\begin{eqnarray}\label{eqn_solution_under_G.2}
   \bromga(t) = \sum_{i=1}^{L} \left((\bromga_{j_{i-1}} - \b_{j_{i}}) e^{-\left(t- \tau_{j_{i-1}}\right)} +  \b_{j_{i}}\right)\ind_{\{t \in [\tau_{j_{i-1}},\tau_{j_{i}}]\}},   \mbox{  for any } t \in [0, \tau]\mbox{ with }    \label{solution_DI_general}
        \end{eqnarray}
$\bromga_{j_0} := \bromga(0) \mbox{ and }   \tau_{j_0} =0$,  where recursively on $\kappa \le L$, 
        \begin{itemize}
            \item [{\bf (s.1)}]  $
         \tau_{j_{\kappa}} = \tau_{j_\kappa} (\bromga_{j_0}) :=  \inf \left \{  t > \tau_{j_{\kappa-1}} :  \bromga(t) \notin \R_{j_\kappa}   \right \}$ represents the epoch at which the solution exits the $\R_{j_\kappa}$ region, when the set is non-empty; else,  we set  $\tau_{j_\kappa} = \infty$, implying the solution stays in $\R_{j_\kappa}$ forever; 
        
         \item   [{\bf (s.2)}]  $\bromga_{j_\kappa} := \bromga(\tau_{j_\kappa})$ is the   DI solution at the above exit epoch $\tau_{j_\kappa}$; 
         \item   [{\bf (s.3)}] 
        $\J(\bromga_{j_\kappa}) = \{j_\kappa, j_{\kappa+1}\}$ with $\R_{j_{\kappa+1}}$ representing the region that DI solution enters at $\tau_{j_\kappa}$; further $\bromga_{j_\kappa}$ is in the border between $\R_{j_{\kappa}}$  and $\R_{j_{\kappa+1}}$,  and touches upon $\H^{a_{j_\kappa},\ta_{j_\kappa}}_{\t_{j_\kappa}}$ for some $(a_{j_\kappa},\ta_{j_\kappa},\t_{j_\kappa})$, basically $h^{a_{j_\kappa},\ta_{j_\kappa}}_{\t_{j_\kappa}} (\bromga_{j_\kappa}) =~0$; 
         \item  [{\bf (s.4)}]  $L \ :=  \min\{ \kappa' : \tau_{j_{\kappa'}} > \tau\}$ equals (maximum) number of the regions traversed by the DI-solution  in interval~$[0, \tau]$.
     \end{itemize}
Let  $h_{j_\kappa} \text{ and } \H_{j_\kappa}$  briefly represent, $h^{a_{j_\kappa}, \ta_{j_\kappa}}_{\t_{j_\kappa}} \text{ and } \H^{a_{j_\kappa},\ta_{j_\kappa}}_{\t_{j_\kappa}}$ respectively. 
By Lemma \ref{lem_solu_tau1_tau2},   the above solution exists when,
  \begin{eqnarray}
  \label{Eqn_conditions_for_solution}
    h_{j_\kappa}   (\b_{j_{\kappa}} )   h_{j_\kappa}  (\b_{j_{\kappa-1}} )  >  0   \mbox{ and }   %
      \J(\bromga_{j_\kappa}) = \{j_\kappa, j_{\kappa+1}\}  \mbox{ for all }  \kappa \in \{1, \cdots, L\}. 
  \end{eqnarray}
The above form of  solution is possible under \ref{asm_gen_2_implication} and only  when $|\J(\bromga_{j_\kappa}) |= 2$ for each $\kappa \le L$. 
Also, recall by Lemma \ref{Lemma_no_Solution_on_H}, there exists no solution that lives on $\H$ for any  (non-zero length) interval of time.

The solution \eqref{eqn_solution_under_G.2} is valid when the initial condition is in the interior of some $\R_j$, i.e., for  $\bromga(0) \in \R_j$ for some $j$. For initial conditions on boundary, i.e.,  \textit{if $\bromga(0) \in  \H^{a,\ta}_\t$  for some boundary,  one can obtain solution exactly as in \eqref{eqn_solution_under_G.2} by Lemma \ref{lem_solu_tau1_tau2_start_line_H} provided in Appendix \ref{subsec_append_Generic_game_proofs}.}

Interestingly, from  \eqref{solution_DI_general},  the solution can  traverse several regions $\{\R_j\}$ during its course. This, coupled with the fact that there are finitely many   $\{\R_j\}$, hints at the possibility of a cycle ---  either a solution gets absorbed in one of the regions or it may return to the initial region $\R_{j_1}$ (the set $\AA$,  in \eqref{eqn_set_AA},  precisely represents the sub-class of absorbing regions, also see \eqref{Eqn_DI_Solution_in_Rj}).  
We will now show that such a possibility is not vacuous and, in fact, also show that the resulting cycle is a non-singleton ICT  set, in the immediate next.

 \subsubsection{\underline{Cyclic ICT sets}}  \label{subsub_sec_Cyclic ICT sets} 
We now describe one scenario, constructed using \eqref{eqn_solution_under_G.2}, under which a solution of the DI \eqref{eqn_general_DI} moves in a cyclic path. We basically identify a `cyclic DI-solution' that forms an ICT set under the following condition: 
\begin{enumerate}[label=\textbf{C.1}, ref=\textbf{C.1}]
\item Assume there exists a solution as in \eqref{eqn_solution_under_G.2}, with two special features: (a)   $(l+1)$-th region coincides with the first region,  $\R_{j_{l+1}} = \R_{j_1}$ for some $l > 2$; and  (b)   $\bromga_{j_{l}} = \bromga(0)$,  i.e.,   the solution starts  in the  boundary $\bromga(0) \in \H^{a_{_l}, \ta_l}_{\t_l}$. \label{asm_cycle_one}
\end{enumerate}
Represent the quantities corresponding to the special solution in \ref{asm_cycle_one} as $\bromga^*_{j_0}, \bromga^*_{j_1}, \cdots, \bromga^*_{j_{l-1}}$, $\tau^*_{j_1}, \cdots \tau^*_{j_l}$ with $\tau_{j_0} = 0$; observe $\bromga^*_{j_0} = \bromga(0) = \bromga^*_{j_l}$. If such a special cyclic solution exists,  it becomes a non-singleton ICT set (proof is given in Appendix \ref{proof_Cyclic_ICT}.)   Observe that, one will require $l > 2$ in \ref{asm_cycle_one} as the solution can flow only in one direction between any two points as seen from \eqref{eqn_solution_under_G.2} or as will be seen in another variety of solution provided in the next subsection. 
\begin{thm}\label{Cyclic_ICT}
    Assume \ref{asm_gen_2_implication}, \ref{asm_cycle_one} and \ref{asm_bj_notin_regions}. Then, the $l$-cycle, denoted by $\CC^* $, and defined as below,  forms an ICT set:
\begin{eqnarray}
\label{Eqn_Ic_cycle}
    \CC^* := \cup_{i=1}^{l}{\mathcal S}_{i},  \mbox{ where } {\mathcal S}_{i} := \left\{(\bromga_{j_{i-1}}^*-\b_{j_{i}})e^{-\left(t-\tau^*_{j_{i-1}}\right)} + \b_{j_{i}}:\tau^*_{j_{i-1}}\leq t \leq \tau^*_{j_{i}}\right\}.  
\end{eqnarray}
\end{thm}
\subsection{Game with three actions:  solution and ICT sets on $\H$}
\label{subsec_game_three_actions}
In this subsection, we take a brief detour to derive an   additional result for the special case with three actions (we believe one can easily extend the same to a more general scenario with $k = |\cup_\t \A_\t| < \infty$, however this would require some technical details); this result provides us impetus to derive new type of singleton ICT sets for the general case,  provided in the later part of the section.

In the previous subsection \ref{subsec_generic_results}, under   assumptions  \ref{asm_card_ind_set_finite} and \ref{asm_gen_2_implication}, we established that there  exists no solution on $\H$, see Lemma \ref{Lemma_no_Solution_on_H}. We now establish the existence of solution  to  DI \eqref{eqn_general_DI} that moves along some $\H^{a,\ta}_\t$ border set (a line for the case with three actions),  \textit{under an assumption that negates \ref{Asm_filippov_new}} (proof is   in  \ref{proof_lem_three_action_on_H}).  

\begin{lem}\label{lem_three_action_on_H} {\bf [Solution on attracting border]}
(i) Consider any $\bromga' \in \H^{a,\ta}_\t$ for some $(a, \ta, \t)$ that satisfies  $\J(\bromga') = \{i, j\}$ for some $i \ne j$ and $h^{a,\ta}_\t (\b_i) h^{a,\ta}_\t (\b_j)  < 0$. 
 Then the following is a solution of the  DI \eqref{eqn_general_DI} with initial condition $\bromga(0) = \bromga'$, 
\begin{eqnarray}
     \label{Eqn_Solution_on_H}
\bromga(t) \hspace{-2mm}&=& \hspace{-2mm} \bromga^\infty_{ij}  + (\bromga(0) -\bromga^\infty_{ij})  e^{-t}  \mbox{ for $t \in [0, \tau]$, where } \tau := \inf \{ t:  \J(\bromga(t)) \ne \{i,j\}  \}, and  \\
\bromga_{ij}^\infty \hspace{-2mm}&=& \hspace{-2mm}  \b_i - (\b_i-\b_j)\frac{h^{a,\ta}_\t(\b_i)}{h^{a,\ta}_\t(\b_i)-h^{a,\ta}_\t(\b_j)}. \label{eqn_limit_sol_on_h}
\end{eqnarray}
(ii) Further, assume  $\J(\bromga^\infty_{ij}) = \{i,j\}$ for   $\bromga^\infty_{ij}$ of \eqref{eqn_limit_sol_on_h}. 
Then $\{\bromga^\infty_{ij}\}$ is a singleton ICT set.  \eop
 \end{lem}
\hide{ {\color{blue} We prove that if \ref{Asm_filippov_new} is true for some $(a,\ta,\t)$ then  DI \eqref{eqn_general_DI}  has no solutions that lies on $\H^{a,\ta}_\t$ for a  non-trivial interval of time unless $\bromga(t) = \bromga'$ for all $t$ and for some $\bromga'\in \H^{a,\ta}_\t$ with $\J (\bromga') > 2$. Conversely also, we prove  that if \ref{Asm_filippov_new}  is not true for some $(a,\ta,\t)$ then there exists a solution of DI \eqref{eqn_general_DI} on the $\H^{a,\ta}_\t$  when the solution starts from a point on that line.  }}
 Thus, to summarize, we have a singleton ICT set consisting of a point different from the classical attractors.
  Observe from \eqref{eqn_limit_sol_on_h} that the limit $\bromga_{ij}^\infty \in \co\{\b_i, \b_j\}$ and that $\bromga_{ij}^\infty \in \H^{a,\ta}_\t$. In fact, as seen from the proof of part (ii) in  \ref{proof_lem_three_action_on_H}, we only require the above  conditions of Lemma \ref{lem_three_action_on_H}  to ensure $\{\bromga^\infty_{ij}\}$ is an ICT set.  
More precisely, the new set of singleton ICT sets (different from the ones arising because of classical attractors) results from points in the following set:
\begin{eqnarray}\label{eqn_set_H_int_con_comb}
    {\mathbb L}_{2,f}^* := 
  \{\bromga \in \H  : \bromga \in \co\{\b_j : j \in \J(\bromga)\}, \ |\J(\bromga)| =2  \} \mbox{ and } \SS_{2,f}^* := \left\{ \{ \bromga \} :  \bromga \in  {\mathbb L}_{2,f} \right\}.  
\end{eqnarray}
Basically, the intersection points of the line segments of  $\H$ with the lines    $ \overleftrightarrow{\b_i,\b_j}$ (corresponding to some adjacent regions) constitute these  singleton ICT sets (`{\it limits at discontinuity}' as $\bromga(t)$ is discontinuous at intersecting points)
--- we refer these  as  `\textit{\Fl attractors}' and so use the   subscript $f$ in~\eqref{eqn_set_H_int_con_comb}.  

 {\bf Additional singleton ICT sets for generic finite games:}
One can immediately extend the ICT sets of \eqref{eqn_set_H_int_con_comb} to a more general case, including all the  finite games (under the negation of \ref{asm_gen_2_implication}). 
Consider any $\bromga^*$ that satisfies, $\bromga^* \in \H \cap \co\{ \b_j : j \in \J(\bromga^*) \}$, irrespective of cardinalities $|\J(\bromga^*)|$ and $k = |\cup_\t \A_\t|$.  This implies $\bromga^*$ is  some convex combination of $\{\b_j\}_j$ (see also \eqref{eqn_general_DI}), i.e.,
$$
\bromga^* =  \sum_{j \in \J(\bromga^*)} \lambda_j \b_j \ \text { implying } \ {\bf 0} = \sum_{j \in \J(\bromga^*)}  \lambda_j(\b_j -\bromga^*), 
$$
$\text{which in turn implies } \mathbf{0} \in \G(\bromga^*) = \co\{ \g_i^\infty ; i \in \J(\bromga^*) \}$. Hence, clearly $\bromga(t) = \bromga^*$ for all $t$ is a solution and therefore $\{\bromga^*\}$ is an ICT set.  In all, we have the following class of  ICT sets arising from the Filippov attractors (obtained by removing the restriction $|\J(\bromga)| = 2$ and $k=3$ in~\eqref{eqn_set_H_int_con_comb}): 
\begin{eqnarray}\label{eqn_limit_sol_h_line_set}
    \mathbb{L}^*_f := \left  \{ \bromga^* \in \H  :  \bromga^* \in   \co\{ \b_j : j \in \J(\bromga^*) \}  \right \}  \mbox{ and } \ \SS^*_f := \left  \{ \{ \bromga^*  \} :\bromga^* \in  \mathbb{L}^*_f  \right \}.
\end{eqnarray}

\subsection{Connection of singleton ICT sets with MT-MFE}\label{subsec_connect_ICT_MT_MFE}
We now show that the attractors \eqref{eqn_set_AA} and \eqref{eqn_limit_sol_h_line_set} are  indeed   aggregate-MFEs that lead to  MT-MFEs, proof in \ref{append_connec_ict_equil}.
\begin{thm}{\bf [Singleton ICTs are aggregate-MFEs]} \label{thm_connection_MT_MFE}
If $\bromga^*  \in \LL^*_c \cup \LL^*_f $, then   $\bromga^*$ is an aggregate-MFE, Definition \ref{defn_MT-MFE} . \eop
\end{thm} 
\section{Methodology for deriving outcomes of the game-dynamics}
\label{sec_summary_of_results}
Given a game $\langle  \Theta, \balpha, \{\A_\t\}_{\t \in \Theta},  (u_\t)_{\t \in \Theta}\rangle$,  by virtue of the results of the previous sections, one can make several inferences related to transient and asymptotic analysis of the corresponding dynamics. In this section, we first summarize the analytical results and then provide a systematic numerical procedure to derive the outcomes.  
\subsection{Analytical approach}\label{subsec-anal_approo}
In section \ref{Sec_ident_DI}, it is convenient to partition  $\D$  
into disjoint $\{\R_j\}$ regions,
using  $\{\H_{\t}^{a,\ta}\}$ boundaries --- basically, the decision of any type remains the same in each $\R_j$ region. This aspect facilitated the required analysis. However, some  $\{\H_{\t}^{a,\ta}\}$  boundaries may not alter the decisions  of any type --- for example, in  two neighboring regions separated by some $\H^{a,\ta}_\t$, 
neither of the  two actions $a, \ta$ may be optimal for type $\t$ population.  One can combine such `adjacent'  $\R_j$ regions for a better description of the results, and we immediately consider the same.

Consider disjoint regions $\{\Q_v \}$, each of which is identified  by unique $\e_v = (e_{v,1}, \cdots e_{v,n})$ vectors, as below:
\begin{equation}
    \Q_v = \left \{ \bromga \in \D :  u_\t (e_{v, \t}, \bromga) >    u_\t (a', \bromga)  \mbox{ for all } a' \ne e_{v, \t},   a' \in \A_\t,  \text{ and all } \t \right \}.
\label{Eqn_Qv}
\end{equation}
Observe immediately that $\cup_v \Q_v \supset \cup_j \R_j  $, which is the sub-region of $\D$ where the population of each type has a unique choice.  Let  $\V :=\{ v : \Q_v \ne \emptyset\}$, represent the class of non-empty region vertices. We now define a natural notion of adjacent regions using $n_c (v,\tv) := \sum_\t \ind_{ e_{v,\t} \ne e_{\tv, \t}}$, that represents  the number of preference changes across two regions $\Q_v$ and $\Q_\tv$.  \textit{We say two regions indexed by $v, \tv \in \V$ are adjacent  when  $n_c(v,\tv) = 1$}, i.e., the two regions differ in the choice of exactly one type. For such adjacent regions, there is a unique type $\t$ for which $e_{v,\t} \ne e_{\tv,\t}$ and 
let 
\begin{eqnarray}
\label{Eqn_Qv_details}
a := e_{v,\t}, \ \ta := e_{\tv,\t},   \    h_{(v,\tv)} := h^{a,\ta}_\t, \mbox { and } \H_{(v,\tv)} := \H^{a,\ta}_\t,     
\end{eqnarray}
note that $\H_{(v,\tv)}$ is the hyperplane separating $\Q_v$ and $\Q_\tv$. It is important to note here that,  two pairs of adjacent regions can have boundary on the same $\H$ line. 
 To analyze the game, we first define for each $v\in \V$, vector $\b_v := \{b_v^a\} $ with $ b_v^a := \sum_\t \alpha_\t \ind_{\{e_{v,\t} = a\}} $ for each $a \in \A_\t$, and 
 then identify two distinct classes of adjacent regions based on \ref{asm_gen_2_implication} assumption  (see  also \ref{asm_bj_notin_regions}):
\begin{eqnarray}
    \I^c &=& \left \{ (v,\tv):  n_c(v,\tv) = 1, \ h_{(v,\tv)}(\b_v)h_{(v,\tv)} (\b_{\tv})>0  \right \}, \text { and } \label{eqn_set_ind_Ic} \\
     \I^* &=& \left  \{ (v,\tv): n_c(v,\tv) = 1, \ h_{(v,\tv)}(\b_v)h_{(v,\tv)}(\b_{\tv}) < 0  \right \}. \label{eqn_set_ind_Istar}
\end{eqnarray}
By Theorem \ref{Thm_stoch_approx}, (i) the interpolated stochastic trajectory representing the dynamics \eqref{eqn_iterates_nu} is a perturbed solution  (see Definition \ref{def_perturbed_sol}) of DI \eqref{eqn_general_DI},  while the limit sets are ICT by part (ii); thus the solutions of the DI provide insights into  the trajectories  of the dynamics while the ICT sets characterize the limit behavior. We now summarize the related 
results.

\noindent{\underline{Non-constant DI solutions:}}
We focus on non-constant solutions of the DI \eqref{eqn_general_DI} --- by Lemma \ref{Lemma_no_Solution_on_H}, such solutions cannot lie on the boundary $\E_{v,\tv} := \H_{(v,\tv)} \cap \overline{\Q}_v \cap \overline{\Q}_{\tv} $  for any pair $(v, \tv) \in \I^{c}$, see \eqref{eqn_set_ind_Ic}.     
On the other hand, by Lemma \ref{lem_three_action_on_H}  there exists a DI solution on the boundary, $\E_{v,\tv}  $  for any $(v,\tv) \in \I^*.$ 
Thus, for a given game, one can immediately identify the following: 
 (a) we either have solutions of type \eqref{solution_DI_general} that can cross through adjacent regions $\Q_v$, $\Q_{\tv}$ with $(v,\tv) \in \I^c$; (b) or have solution  \eqref{Eqn_Solution_on_H} on the `attracting' boundary region, $\E^* := \cup_{(v,\tv) \in \I^*} \E_{v,\tv}$ (this result  is proved  only for $k=3$, but we anticipate it to be  true for general $k$). One can also have solutions with some parts as in \eqref{solution_DI_general} and some as in \eqref{Eqn_Solution_on_H}.

\noindent{\underline{Classical and Filippov attractors:}}  We refer to singleton ICT sets as attractors  (with a
slightly erroneous interpretation, as will be discussed below). The set $\SS^*_c $ in  \eqref{eqn_set_AA} represents a sub-class of such ICT sets, derived using the   `classical attractors'; it is easy to verify that one can rewrite this set using $\Q$ regions as:  $\LL^*_c := \{ \b_v : \b_v \in \Q_v\}$ and $\SS^*_c := \{\{\b_v\}: \b_v \in \LL^*_c\}$. 
In a similar way, the second sub-class of attractors,  
  the   \Fyl attractors of \eqref{eqn_limit_sol_h_line_set} that arise primarily because  of discontinuities in dynamics, can be rephrased as below (see \eqref{Eqn_Jomega}):
\begin{eqnarray}
    \mathbb{L}^*_f &:=& \left  \{ \bromga^* \in \H  :  \bromga^* \in   \co\{ \b_v : v \in \J(\bromga^*) \}  \right \},  \mbox{ with } \label{Eqn_SSF}\\
    \J'(\bromga) &:=& \{ v \in \V: \bromga \in \overline{\Q}_v \} \text{ for any } \bromga \in \H. \nonumber
\end{eqnarray} 
 At this juncture, we would like to mention that the use of the word `attractor' is not exactly appropriate for these ICT sets. These sets are more like the equilibrium points of classical dynamic system literature, can have characteristics like `repellers' or `attractors' or `saddle points' (see, e.g., \cite{strogatz2024nonlinear}); we nonetheless choose to call them as attractors, as many of them  become asymptotic limits of stochastic iterates, see Lemma \ref{lemma_singleton_ICT} and the numerical results. 

\noindent{\underline{Cyclic ICT sets and Region-Vertex (RV) graph:}} 
\textit{We now construct a directed graph, with an aim to identify the conditions for non-existence of cycles.} Consider a graph with the vertex set as $\V$,   
the edge set as  $\EE_c = \I^c $ and direction is from $v\to v'$ if $h_{(v,v')} (\b_v)  <0$ for all $(v,v') \in \EE_c$ (one can work with reverse direction also, i.e., with $h_{(v,v')} (\b_v)  > 0$ for all $(v,v') \in \EE_c$);  we basically place edges   on those 
  `adjacent regions' $(v, \tv)$,  whose separating boundary $\H_{(v,\tv)}$  cannot house any  DI solution  (rather any  solution  that touches such a  boundary only passes through it  as established in Lemma \ref{Lemma_no_Solution_on_H}, also see \eqref{eqn_set_ind_Ic}). 
  \textit{We refer to the directed   graph $(\V,\EE_c)$ as the Region-Vertex (RV) graph.}   \textit{If the RV  graph contains no cycles, then no cyclic ICT sets as in Theorem \ref{Cyclic_ICT} exist} --- it is clearly not possible to satisfy\footnote{Any $\Q_v$ region is union of some finitely many $\R_j$ regions; some of these constituent $\R_j$ regions can be absorbing  by convexity (for   regions with $\b_j \in \R_j$ solution starting in $\R_j$ remains in $\R_j$ forever by \eqref{Eqn_DI_Solution_in_Rj}), however some boundaries of such $Q_v$ can also participate in a cycle.} assumption \ref{asm_cycle_one} with no cycles in the RV graph.

 On the other hand,  the presence of a cycle in RV graph indicates the possibility of cyclic ICT sets. Conditions for further verifying the existence of cyclic ICT sets (i.e., to satisfy assumption \ref{asm_cycle_one} completely) are provided in Appendix~\ref{subsub_sec_cond_exist_cycle} for the games with three actions. \textit{At this juncture, we would like to remark that the discussions of this paper related to cyclic limit sets are around the cyclic ICT sets identified in this paper;  one might have other varieties of cyclic ICT sets and could be of interest for future investigation.} 
Next, we provide a numerical procedure for analysis of the game:

\subsection{ Numerical procedure}
\label{sec_numerical_procedure}
The numerical procedure consists of the following three steps:

\noindent\textbf{Step 1:} We begin by identifying the unique $\e_v$ vectors corresponding to each $\Q_v$ region in the domain $\D$. To this end, we first enumerate all possible combinations of $\e$-vectors, i.e., vectors $\e = (e_\t)_{\t \in \Theta}$ with $e_\t \in \A_\t$. Among these, we retain only those $v$ for which the corresponding $\Q_v$ region  is nonempty, 
to obtain the vertex set $\V$. Towards this, we solve the following linear program (LP),  for each $\e_v$:
\begin{eqnarray*}
\max_{\bromga} \Lambda(\bromga)   \quad \text{s.t.} \quad h^{e_{v,\t}, a'}_\t(\bromga) > 0 \ \forall \ a' \ne e_{v,\t},\ \forall \ \t;\ \ \sum_a \bromgac^a = 1;\ \ 0 \le \bromgac^a \le 1\ \forall \ a,
\end{eqnarray*}
 where $\Lambda(\cdot)$  is any linear function. If this LP admits a feasible solution for some $\e_v$, we retain  it  and include the corresponding index $v$ in   $\V$ --- the LP-feasibility  implies the existence of at least one $\bromga \in \D$ satisfying the following:   for each type $\t$  the action $e_{v,\t}$   provides strictly higher utility than all other actions.
Then compute $\b_v$ for each  $v \in \V$ (recall $ b_v^a := \sum_\t \alpha_\t \ind_{\{e_{v,\t} = a\}} $  $\forall$ $a$). 

\noindent\textbf{Step 2:} We next identify the class of singleton attractors   $  \SS^*_c \cup \SS^*_f $,   defined in  \eqref{eqn_set_AA}, \eqref{eqn_limit_sol_h_line_set}, using the   following procedure: 

\textbf{(i)} To obtain    $\SS^*_c$, the set of classical attractors of \eqref{eqn_set_AA}, simply  verify if  $\{\b_v\} \in \SS^*_c$ for each $v \in \V$ --- by numerically checking  if  $\b_v \in \Q_v$ and this in turn is achieved  by verifying through the functions $\{h^{a,\ta}_\t(\cdot)\}$ defined in  \eqref{Eqn_h_linear}:
$$h_\t^{e_{v,\t},a'}(\b_v) >0 \ \forall \ a' \neq e_{v,\t} \ \forall  \ \t.$$ 

\textbf{(ii)} 
Next we identify   $\SS^*_{>,f} := \SS^*_f \setminus \SS^*_{2,f}$ defined using \eqref{eqn_set_H_int_con_comb}-\eqref{eqn_limit_sol_h_line_set}. This is achieved in two sub-steps: 
 
$\quad$ (a) To begin with,  if $\{\bromga^*\} \in \SS^*_{>,f}$, then  $\bromga^* \in \H$ with $|\J'(\bromga^*)| > 2$ -- any such  $\bromga^* $, with $|\J'(\bromga^*)| > 2$,  is  a common point of more than two hyperplanes among $\{ \H_{(v,\tv)}\}$, each of which separates a pair of adjacent regions --- thus each such $\bromga^*$ implies the existence of a cycle of length $\ell > 2$ in the 
undirected graph $(\V, \E_a)$ with edge set $\E_a := \{(v, \tv) : n_c(v,\tv) = 1\}$ --- the existence of such cycles can be verified using numerical methods for graphs.  For each such cycle $\C = (v_1, \cdots, v_\ell,v_1)$  with   $\ell > 2$, we identify one potential  candidate  for  $  \SS^*_{>,f}$    as the unique $\bromga^*$ that satisfies,
\begin{eqnarray}
\bromga^* \in \H_{(v_\ell, v_{1})} \cap   \left(\cap_{i=1 }^{\ell-1} \H_{(v_i, v_{i+1})}\right) \text{ where } v_i \in \C = (v_1, \cdots, v_\ell, v_1). \label{Eqn_cycle}
\end{eqnarray}

$\quad$(b) To determine the subset of   the potential candidates identified in step (a), that  belong to
$\SS^*_{> ,f}$ --- 
from \eqref{eqn_limit_sol_h_line_set}, $\{\bromga^*\} \in \SS^*_{> ,f}$, if further  $\bromga^* \in   \co\{ \b_v : v \in \J'(\bromga^*)\}$. Such a $\bromga^*$ can be rewritten as $\bromga^* = \sum_{v \in \J'(\bromga^*)} \lambda_v \b_v$,  for some $\lambda_v$ satisfying $ 0\leq \lambda_v \leq1$, $\sum_{v} \lambda_v = 1$. This inclusion  can thus be verified  for each potential $\bromga^*$  of step (a)   by checking the feasibility of the following LP (constructed using  the cycle of  \eqref{Eqn_cycle} corresponding to $\bromga^*$):
\begin{eqnarray}\label{eqn_lp_num_proc}
     && \min_{\{\lambda_i\}_{i \in \{1,\ldots,\ell\}}} \Lambda(\lambda_i) \quad \text{s.t.} \nonumber\\ && h_{(v_i,v_{i+1})}\left(\sum_{v} \lambda_v \b_v\right) = 0\ \forall \ i < \ell; \ \ h_{(v_\ell,v_1)}\left(\sum_{v} \lambda_v \b_v\right) = 0;\ \
    0 \le \lambda_v \le 1; \ \ \sum_{v} \lambda_v = 1.\qquad
 \end{eqnarray}
In the above $\Lambda(\cdot)$ is again any dummy linear   function and see \eqref{Eqn_Qv_details}. 

\noindent\textbf{(iii)} To derive $\SS^*_{2,f}$ of \eqref{eqn_set_H_int_con_comb}, consider each $(v,\tv) \in \I^*$ and first compute the point of intersection $\bromga^*$ of  the line $\overleftrightarrow{\b_v, \b_\tv}$ and $\H_{(v,\tv)}$, i.e., compute $ \bromga^* \in \overleftrightarrow{\b_v, \b_\tv} \cap \H_{(v,\tv)}$. This point is obtained by   calculating,  
\begin{equation*}
\lambda = \left|\frac{h_{(v,\tv)}(\b_\tv)}{h_{(v,\tv)}(\b_v) - h_{(v,\tv)}(\b_\tv)}\right| \text{ and } 
    \bromga^* = \lambda \b_v + (1-\lambda)\b_\tv. 
\end{equation*}

\begin{wrapfigure}{r}{0.32\textwidth}
\vspace{-4.8mm}
    \centering
    \includegraphics[trim={0cm 4cm 1cm 4.4cm},clip,scale=0.19]{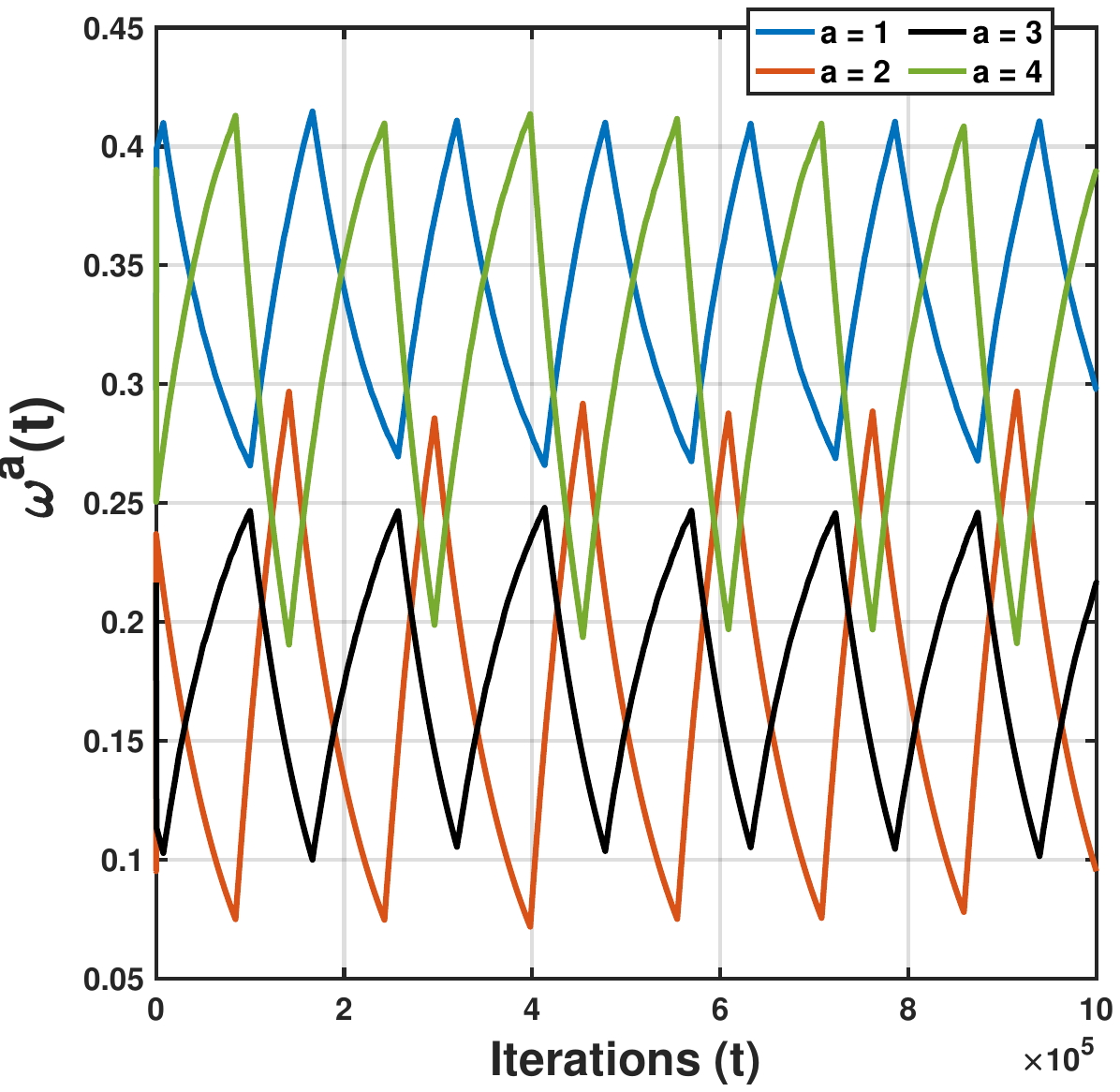}
    \captionsetup{type=figure}
    \vspace{-2mm}
    \caption{Cycle in an instance}
    \label{fig:four_actions_cycle}
    \vspace{-5mm}
\end{wrapfigure}
We then verify if $\bromga^*$ belongs to $\H_{(v,\tv)}$ (let the type that changed action between the adjacent  regions $\Q_v$ and $\Q_\tv$ be $\t$, see \eqref{Eqn_Qv_details}). Indeed $\bromga^* 
\in \H_{(v,\tv)}$   if: (a) the common action  $  e_{v, \t'} = e_{\tv, \t'}$  performs superior compared to other actions in $\Q_v\cup \Q_\t$  for all  $\t' \ne \t$,  and  b)  if for type $\t$, the action $e_{v, \t}$  is superior in $\Q_v$ and   the action $e_{\tv, \t}$ is  superior in  $\Q_{\tv}$.

$\noindent$\textbf{Step 3:} Finally, we determine numerically whether the directed graph $(\V, \EE_c)$ contains directed cycles using the function `allcycles($\cdot$)'. 
If no such cycle exists, this confirms the absence of cyclic ICT set. 

\paragraph{\bf Numerical 
instances} 
We generate several random instances of games (basically the parameters) and utilize the procedure described above to obtain the various possible outcomes. We then generate several sample paths of dynamics \eqref{eqn_iterates_nu} for the   random games 
--- in all the cases,  the limits of the sample paths are one among the  outcomes predicted. 
In  Figure \ref{fig:four_actions_cycle}, we plot one such sample path corresponding to a random instance with four actions, whose RV graph has cycles.  As can be seen, 
the sample path indeed exhibits cyclic behavior. 
\textit{Interestingly,   the period of the cycle increases with $t$, but the cycles persist forever.} 
\subsection{Analysis of queuing game}\label{subsec_analysis_Queuingnetwork}
We now  analyze  the queuing game introduced in Subsection \ref{sub_sub_sec_queuing_application} under two scenarios.
\subsubsection{\underline{With two types of rational agents}} 
Consider the game with only  $\CSr$  and $\CMr$ rational customers.
Each agent type has three available actions, implying a total of $2*3=6$ potential border lines $\{\H_\t^{a,\ta}\}$.

We consider an interesting case study, where the $\CSr$-rational costs are equal for all the three options at full load:  we basically consider that $p_s = 1 = \rho + p$ in \eqref{eqn_queuing_game_utility_function}, by virtue of which we have  $u_\CSr(1, (1,0) ) = u_\CSr(2, (0, 1) )  = u_\CSr(3, (0,0) )  = -1$. Hence, both $\H^{1,3}_\CSr$ and $\H^{2,3}_\CSr$ regions just touch the boundary of $\D$ and thus do not contribute to the analysis.  Therefore, for any $\alpha_\CSr <  1$, it is not difficult to identify that action~$3$ is strictly dominated by other actions for the $\CSr$-rational crowd under $p_s = 1 = \rho + p$, see \eqref{eqn_queuing_game_utility_function}. Thus we consider (effectively) $\A_\CSr = \{1,2\}$.

The borders constructed using \eqref{eqn_h_ij}, \eqref{eqn_queuing_game_utility_function}   and \eqref{CMR}, corresponding to $\CSr$ and $\CMr$ rational agents,  are   as below:
\begin{eqnarray}\label{eqn_cs_cm_bound_lines1}
h^{1,2}_{\CSr}(\bromga) \hspace{-2mm} &=& \hspace{-2mm} -\bromgac^1+\rho\bromgac^2+p, \
 h^{1,2}_{\CMr}(\bromga) = -\bromgac^1  - \bromgac^2 +  (c+1)\bromgac^3, \\  h^{2,3}_{\CMr}(\bromga) \hspace{-2mm} &=& \hspace{-2mm}  2\bromgac^1-\bromgac^2  -c\bromgac^3 ,  \ h^{3,1}_{\CMr}(\bromga) = -\bromgac^3  -\bromgac^1 +  2\bromgac^2.\label{eqn_cs_cm_bound_lines2}
\end{eqnarray}
 Further, considering the $\Q$-regions as described in section \ref{sec_summary_of_results}, the domain $\D$ is partitioned into five $\Q$-regions using the above functions (see
Figure \ref{fig:first_two} (Left)); each of them are identified by the following $\{\e_v\}$ and $\{\b_v\}$ vectors:
{\small
\begin{eqnarray}\label{eqn_b_val_CS_CM1}
   \e_1 \hspace{-2mm} &=& \hspace{-2mm} (1,3),  \ \e_2 = (1,1), \ \e_3 = (1,2), \ \e_4 = (2,2),  \ \e_5 = (2, 1);\\
   \b_1 \hspace{-2mm} &=& \hspace{-2mm} (\alpha_\CSr,0), \ \b_2 = (1,0), \ \b_3 = (\alpha_\CSr,\alpha_\CMr), \  \b_4 = (0,1),  \ \b_5 = (\alpha_\CMr, \alpha_\CSr).\label{eqn_b_val_CS_CM2}
\end{eqnarray}}
By Theorem \ref{Thm_stoch_approx}, the limit set of the iterates in \eqref{eqn_iterates_nu} is an ICT set. These ICT limit sets can be characterized using our theoretical results summarized in  section \ref{sec_summary_of_results}. Following this methodology, we obtain the following (proof in \ref{sec_append_Analysis_of_queuing_game}):
\begin{thm} {\bf[With two rationals]}\label{thm_appli_1}
Assume
$\rho < \nicefrac{1}{2}$ and $p+\rho = p_s =1$. The  $\LL_c^*$ of \eqref{eqn_set_AA} and $\LL^*_f$ of \eqref{eqn_limit_sol_h_line_set} are given by:
    \begin{itemize}
        \item[(i)] the set $\LL_c^*=\{(\alpha_\CSr,\alpha_\CMr)\}$  iff 
        $\frac{\rho}{1+\rho} < \alpha_\CMr < \nicefrac{2}{3}$,
        \item[(ii)]   (a) $\left\{ (\frac{1}{1+\rho},\frac{\rho}{1+\rho})\right\} \in \LL^*_f$ \hspace{1mm} iff   \hspace{1mm} $\alpha_\CMr < \frac{\rho}{1+\rho}$,
       \ \ \ \      (b) $\left\{(1,0)\right\} \in \LL^*_f$ \hspace{1mm} iff   \hspace{1mm}   $c \to \infty$ and $\alpha_\CMr = 1$,\\
            (c) $\left\{\left(\frac{2c+1}{3(c+2)},\frac{1}{3}\right)\right\} \in \LL^*_f$ iff $ \frac{c+5}{3(c+2)} < \alpha_\CMr$.
    \end{itemize}
Moreover, the RV graph has cycle(s) iff
$\nicefrac{2}{3}<\alpha_\CMr$, i.e., the stochastic iterates in \eqref{eqn_iterates_nu} may evolve cyclically in the limit. \eop
\end{thm}

\begin{figure*}[ht]
\vspace{-5mm}
\centering

\begin{minipage}{\textwidth}
    \centering
    \begin{subfigure}{0.45\textwidth}
        \centering
        \includegraphics[trim={4.4cm 4cm 9cm 5.4cm}, clip, scale=0.20]{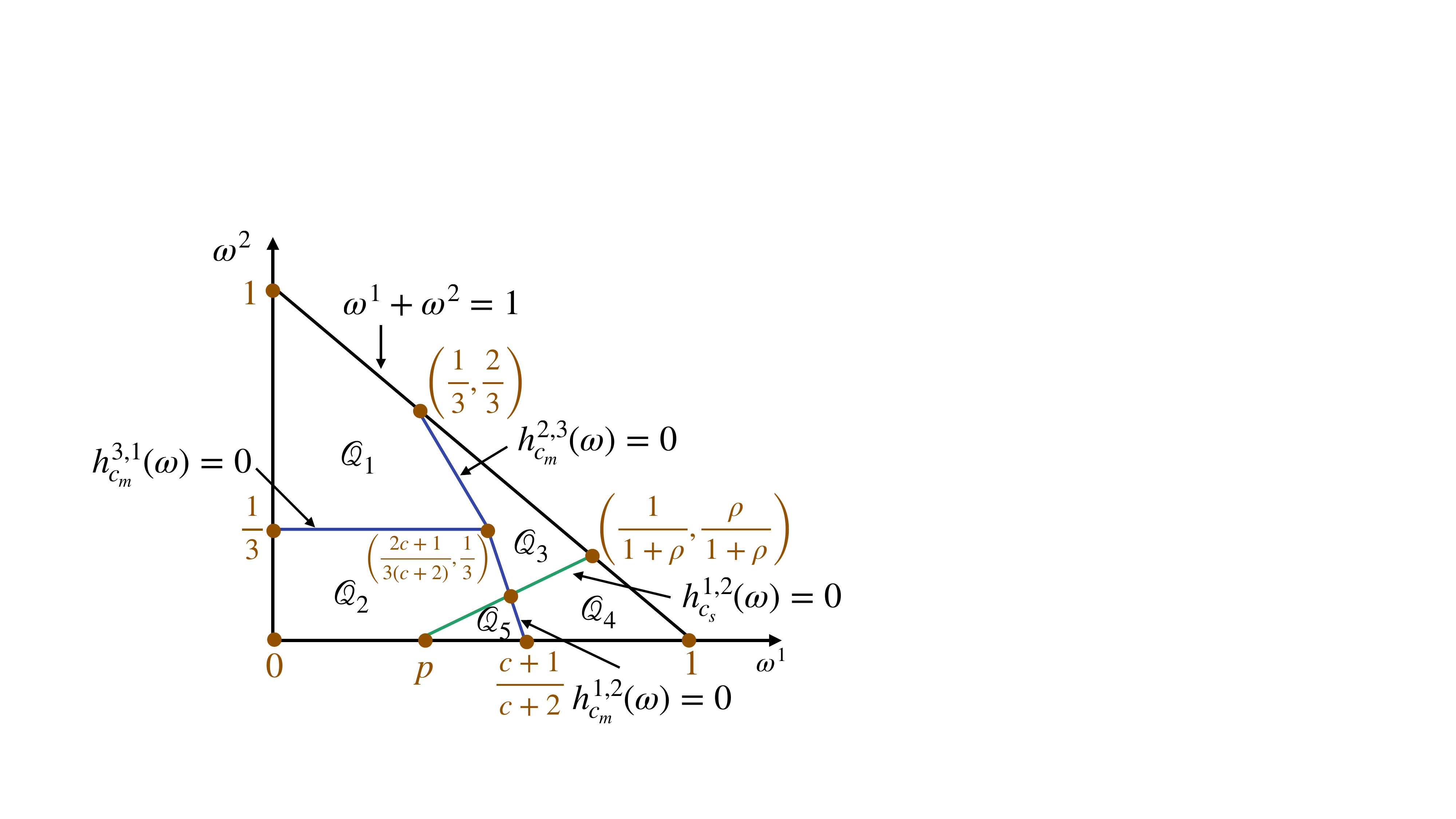}
        \label{fig:cs_cm}
    \end{subfigure}
    \hfill
    \begin{subfigure}{0.45\textwidth}
        \centering
        \includegraphics[trim={1cm 7cm 2cm 5cm}, clip, scale=0.34]{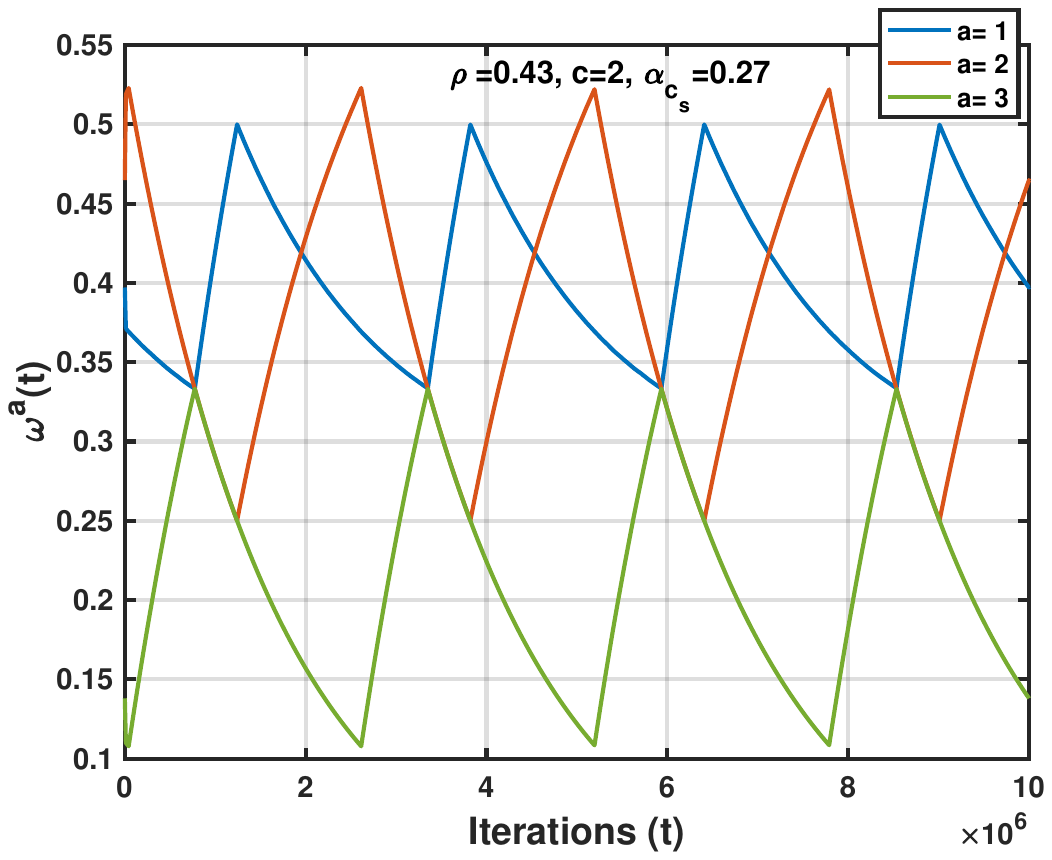}
        \label{fig:three_actions_cycle_two_type}
    \end{subfigure}

    \vspace{-1mm}
    \caption{{\bf Queuing game \ref{sub_sub_sec_queuing_application} with two types of rational agents:}
    $\{\Q_v\}$ regions (Left), Numerical cycle (Right)}
    \label{fig:first_two}
\end{minipage}

\begin{minipage}{\textwidth}
    \centering
    \begin{subfigure}{0.45\textwidth}
        \centering
        \includegraphics[trim={3cm 4cm 12cm 5.5cm}, clip, scale=0.20]{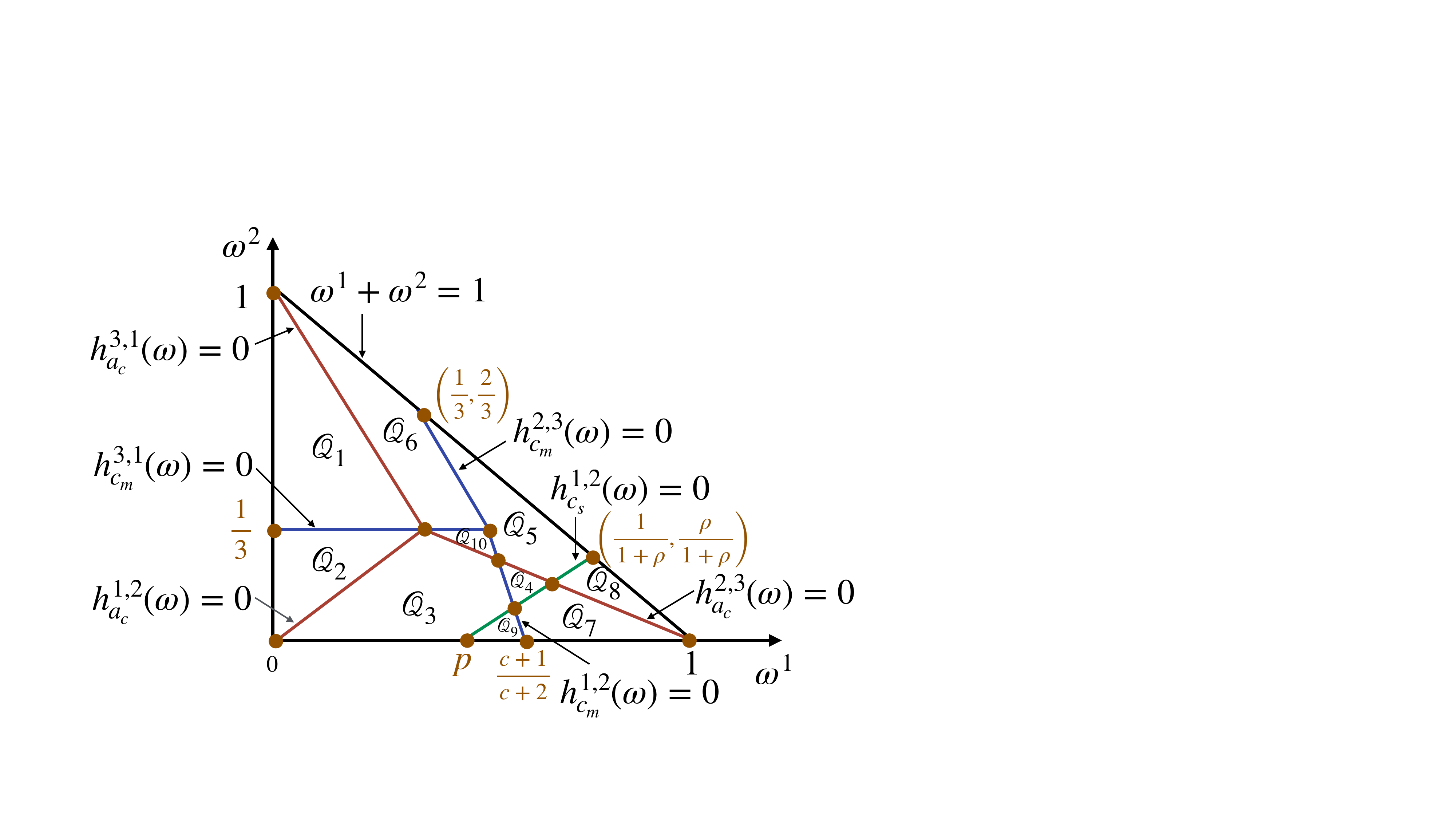}
        \label{fig:rational_avoid_compare}
    \end{subfigure}
    \hfill
    \begin{subfigure}{0.45\textwidth}
        \centering
        \includegraphics[trim={0cm 5cm 2cm 6.5cm}, clip, scale=0.28]{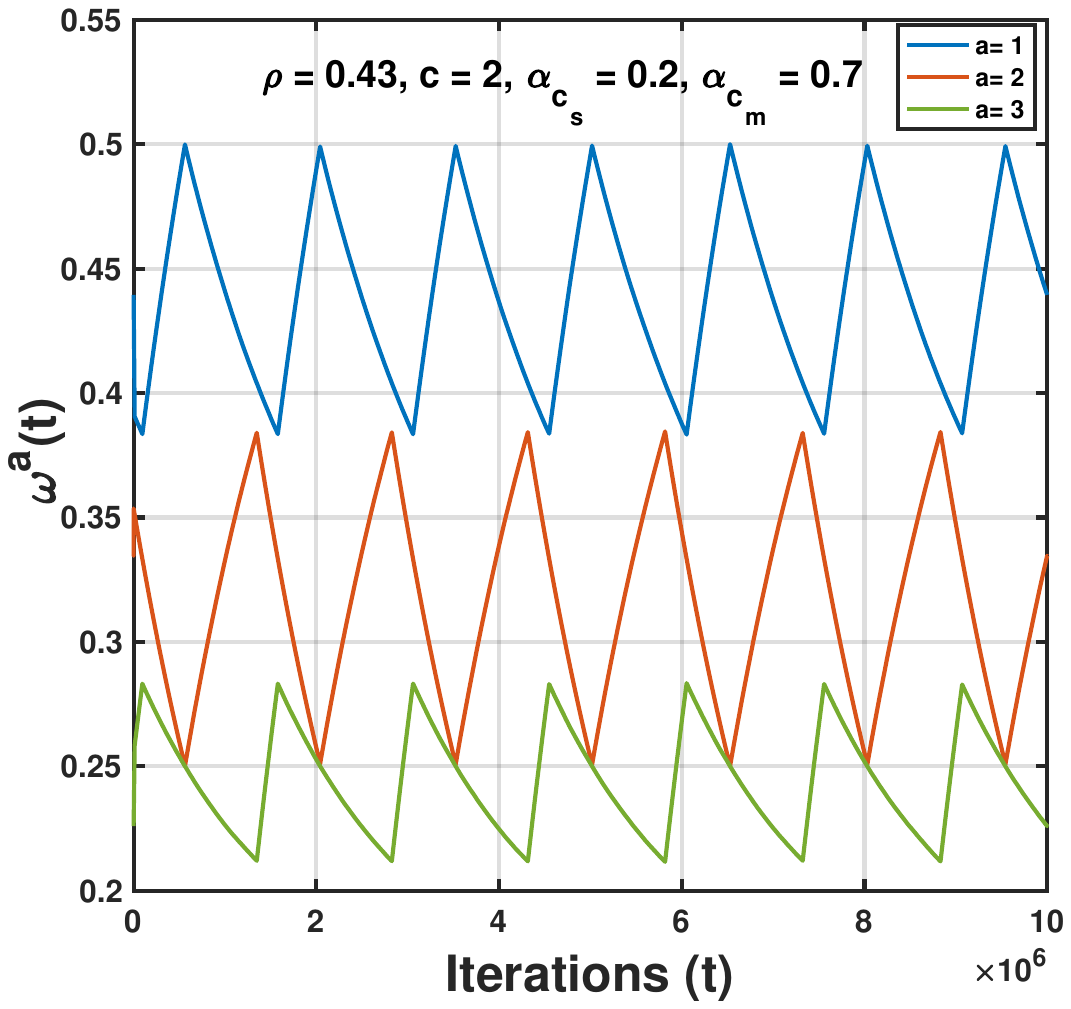}
        \label{fig:three_actions_cycle_three_type}
    \end{subfigure}

    \vspace{-2mm}
    \caption{{\bf Queuing game \ref{sub_sub_sec_queuing_application} with three types of agents:}
    $\{\Q_v\}$ regions (Left), Numerical cycle (Right)}
    \label{fig:next_two}
\end{minipage}
\vspace{-5mm}
\end{figure*}

\subsubsection{\underline{With three types of agents}} 
Consider the $\nuu$-agents along with two types of rational agents in the game. Each type has three action choices implying a total of  $3 * 3 = 9$ potential border lines $\{\H_\t^{a,\ta}\}$. However, similar to the previous case, we have $\A_\CSr = \{1,2\}$. Consequently, the borders corresponding to the $\CSr$ and $\CMr$ rational type agents are given by \eqref{eqn_cs_cm_bound_lines1}-\eqref{eqn_cs_cm_bound_lines2} 
and corresponding to  $\nuu$ type agents using  \eqref{eqn_avoid_utility} are given by $h^{a,\ta}_{\nuu} (\bromga) = \bromgac^\ta - \bromgac^a, \  \forall \ a\ne \ta$.
  %
Here, the domain $\D$ partitions into ten $\Q$-regions with boundary lines as in Figure~\ref{fig:next_two} (Left), and with  $\{\e_v\}$ and $\{\b_v\}$ vectors:
{\small
\begin{eqnarray}
\e_1\hspace{-2mm}&=&\hspace{-2mm}(1,1,3),\ \e_2 = (1,1,1),\ \e_3 = (1,1,2),\ \e_4 = (1,2,2),\ \e_5 = (1,2,3),\ \e_6 = (1,3,3), \label{eqn_gen_e_vec_cs_cm_anu1}\\
\e_7\hspace{-2mm}&=&\hspace{-2mm}(3,3,3),\ 
  \e_8=(2,2,3),\ \e_9=(2,1,2),\ \e_{10}=(1,1,3); \ \b_1=(\alpha_\CSr+\alpha_\nuu,0),\  \b_2 = (1,0), \label{eqn_gen_e_vec_cs_cm_anu2}\\
 \b_3 \hspace{-2mm}&=&\hspace{-2mm} (\alpha_\CSr+\alpha_\CMr,\alpha_\nuu),\ \b_4 = (\alpha_\CSr,\alpha_\nuu+\alpha_\CMr),\ \b_5 = (\alpha_\CSr,\alpha_\CMr), \ \b_6 = (\alpha_\CSr,0),\  \b_7 = (0,1),\\
    \b_8 \hspace{-2mm}&=&\hspace{-2mm} (0,\alpha_\CSr+\alpha_\CMr), \ \b_9 = (\alpha_\CMr,\alpha_\CSr+\alpha_\nuu),\  \b_{10} = (\alpha_\CSr+\alpha_\CMr,0).
 \label{eqn_b_vec2} \label{eqn_gen_e_vec_cs_cm_anu3}
\end{eqnarray}}
\hide{
The dynamics \eqref{eqn_iterates_nu} of the game in cases (i) ($l=5$) and (ii) ($l=10$) are governed by the ODE (see \eqref{eqn_new_general_ode2}):
\begin{eqnarray}
\dot{\bromga} \hspace{-2mm}&=&  \hspace{-2mm}
\g(\bromga) = \sum_{j=1}^{l} \b_j \ind_{\{\bromga \in \Q_j\}} - \bromga, \mbox{ where $\b_j$ vectors are given by \eqref{eqn_b_val_CS_CM1}-\eqref{eqn_b_val_CS_CM2} for case (i) and by \eqref{eqn_b_vec_21}-\eqref{eqn_b_vec2} for case (ii).}
\end{eqnarray}
One can derive the DI corresponding to the above ODEs using \eqref{eqn_general_DI}.}
We once again use the procedure section \ref{sec_summary_of_results}  to  obtain the following (proof in \ref{sec_append_Analysis_of_queuing_game}):
\begin{thm} {\bf[With all three types]}\label{thm_appli_2}
Assume \ref{asm_bj_notin_regions}. Assume $\rho < \nicefrac{2}{(c+1)}$ 
and $p+\rho = p_s =1$. Then the following statements holds for  $\LL_c^*$ and $\LL^*_f$ (with $\alpha_{sm} := \alpha_\CSr+\alpha_\CMr)$:
\begin{itemize}
        \item[(i)]the set $  \mathbb{L}^*_c= \{(\alpha_\CSr,\alpha_\CMr)\}$ iff 
        $\ \alpha_\nuu < \min\left\{\frac{1}{c+2}, \alpha_\CMr\right\} ,  (\rho+1) \alpha_\CMr+\alpha_\nuu > \rho,  3\alpha_\CMr + (c+2)\alpha_\nuu <2$, 
        \item[(ii)] For $\LL^*_f$, we have: \  \   (a) $\left(\alpha_\CSr,\frac{1- \alpha_\CSr}{2}\right) \in \LL^*_f$ \hspace{1mm} iff \hspace{1mm} $\alpha_\nuu \neq 0, \ \  \alpha_\nuu > \alpha_\CMr, \ \ \frac{c}{c+2} < \alpha_\CSr < \frac{2-\rho}{2+\rho}$,\\ (b) $ \left(\frac{1-\rho\alpha_\nuu}{1+\rho},\frac{\rho-\alpha_\nuu}{1+\rho}\right) \in \mathbb{L}^*_f$ iff  \hspace{1mm} $\alpha_\nuu < \frac{\rho}{2+\rho}, \ \rho \alpha_\CMr + p < \alpha_\CSr$,  \  \   
            (c) $ \left(\alpha_{sm} ,\frac{\alpha_\nuu}{2}\right) \in \mathbb{L}^*_f$ iff  \hspace{1mm} $\alpha_\nuu \neq 0, \ \frac{1}{3} < \alpha_{sm}  < \frac{c}{c+2}$.
        \hide{
        \begin{itemize}
            \item[(a)] $ \left(\alpha_\CSr,\frac{1- \alpha_\CSr}{2}\right) \in \LL^*_f$ iff \hspace{1mm} $\alpha_\nuu \neq 0, \alpha_\nuu > \alpha_\CMr, \ \frac{c}{c+2} < \alpha_\CSr < \frac{2-\rho}{2+\rho}$,

            \item[(b)] $ \left(\frac{1-\rho\alpha_\nuu}{1+\rho},\frac{\rho-\alpha_\nuu}{1+\rho}\right) \in \mathbb{L}^*_f$ iff  \hspace{1mm} $\alpha_\nuu < \frac{\rho}{2+\rho}, \ \rho \alpha_\CMr + p < \alpha_\CSr$,

            \item[(c)] $ \left(\alpha_\CSr+\alpha_\CMr,\frac{\alpha_\nuu}{2}\right) \in \mathbb{L}^*_f$ iff  \hspace{1mm} $\alpha_\nuu \neq 0, \ \frac{1}{3} < \alpha_\CSr+\alpha_\CMr < \frac{c}{c+2}$.
        \end{itemize}}
    \end{itemize}
Moreover, the RV graph has cycle(s) iff 
$ \alpha_\CSr < \min\left\{\nicefrac{1}{2},\nicefrac{(c \alpha_\nuu+\alpha_\CMr)}{2}\right\},
\ \ \alpha_\CMr > \nicefrac{1}{2}$, in which case the stochastic iterates in \eqref{eqn_mu_bar} can evolve cyclically in the limit. \eop
\end{thm}

\vspace*{-3mm}
\begin{table}[ht]
\scalebox{0.92}{
\begin{tabular}{|l|l|l|}
\hline
\textbf{Limits of SA based scheme \eqref{eqn_iterates_nu}} 
& 
\begin{tabular}[c]{@{}l@{}}\textbf{Conditions for $\CSr$} \\ \textbf{and $\CMr$ rational}  \\ \textbf{agents ($\rho < \nicefrac{1}{2}$)}\end{tabular}
& \begin{tabular}[c]{@{}l@{}}\textbf{Conditions for $\CSr$, $\CMr$ rationals} \\ \textbf{and $\nuu$ agents $(\rho < \nicefrac{2}{(c+1)})$}\end{tabular} \\ \hline

\begin{tabular}[c]{@{}l@{}}Filippov Attractor:\\ 
{\scriptsize$\left(\dfrac{1-\rho\alpha_\nuu}{1+\rho},\dfrac{\rho-\alpha_\nuu}{1+\rho}\right) 
\xrightarrow{\alpha_\nuu \downarrow 0} 
\left(\dfrac{1}{1+\rho},\dfrac{\rho}{1+\rho}\right)$}
\end{tabular}
& $\alpha_\CMr < \dfrac{\rho}{1+\rho}$ 
& \begin{tabular}[c]{@{}l@{}}$\alpha_\nuu < \frac{\rho}{2+\rho}$, \\ $\rho \alpha_\CMr + 1-\rho< 1-\alpha_\nuu- \alpha_\CMr$ \end{tabular} \\ \hline

\begin{tabular}[c]{@{}l@{}}Classical Attractor:
$(\alpha_\CSr,\alpha_\CMr)$ \end{tabular} 
& $\dfrac{\rho}{1+\rho} < \alpha_\CMr < \tfrac{2}{3}$ 
& \begin{tabular}[c]{@{}l@{}}$\alpha_\nuu < \min\!\left\{\nicefrac{1}{(c+2)}, \alpha_\CMr\right\},$  \\
$\alpha_\CMr+\frac{\alpha_\nuu}{\rho+1} > \dfrac{\rho}{\rho+1},$ \\ 
$\alpha_\CMr+\nicefrac{(c+2)\alpha_\nuu}{3} < \nicefrac{2}{3}$ \end{tabular}
\\ \hline

Cycle in RV graph 
& $\alpha_\CMr > \nicefrac{2}{3}$ 
& \begin{tabular}[c]{@{}l@{}}$\alpha_\CSr < \min\!\left\{\frac{1}{2}, \frac{(c-1)\alpha_\nuu+1-\alpha_\CSr}{2}\right\},$ \\
$\alpha_\CMr > \nicefrac{1}{2}$ \end{tabular} \\ \hline
\end{tabular}}
\caption{Comparison of queuing game results with and without $\nuu$ or avoid-the-crowd agents}
\label{tabel_for_results}
\end{table}
{\bf Remarks about outcomes:}
To begin with,  the results of Theorem \ref{thm_appli_2} ({\it Filippov}  and classical attractors and cyclic behavior) coincide with that in Theorem \ref{thm_appli_1}, when $\alpha_{a_c} \downarrow 0$.  For ease of reference, some results are tabulated in Table \ref{tabel_for_results}. \\
$\bullet$ There exists at maximum one classical  attractor (row $2$), where all $\CSr$-agents use standard queue,  $\CMr$-agents use moderate queue and $\nuu$-agents (if any) use premium queue $(a=3)$. This is a nicely balanced equilibrium with all types having absolute preferences (or pure strategies) and occurs when none of the types are negligible in proportion. \\
$\bullet$ When cycles are possible (row 3), the classical attractor does not exist (row 2) --- basically when cycles exist, one cannot have convergence to MFEs in pure strategies (recall Theorem \ref{thm_connection_MT_MFE}). However one may have cyclic outcomes or convergence to Filippov attractors depending upon the initial conditions. 
We illustrate two instances of queuing dynamics (with $\alpha_{\CMr}$ sufficiently large) that exhibit cyclic behavior in Figures \ref{fig:first_two} and \ref{fig:next_two}. \\
$\bullet$ Cycles can occur only in systems with more than $\nicefrac{2}{3}$ of the population making choices by comparison with others  as in \eqref{CMR}; this is the case when $\alpha_{a_c}=0$; however, in the presence of avoid-the-crowd population, one can have cycles even with $\alpha_{\CMr}$ just above $\nicefrac{1}{2}$. Thus a simple avoid-the-crowd trait can accentuate the cyclic behavior, generated by compare population with cyclic preferences;   on the other hand,  $\CSr$-agents conscious of direct costs diminish the cyclic tendencies.      
One can analyze various other variants of this game  in a similar manner and for other parameter regimes. 
\section{Conclusions}
This paper considers multi-type population games with (any) finite number of action choices and agent types.
In particular, we consider a sequential turn-by-turn   decision making model where the players make choices  one after the other,  participating only once. 
The preferences of different types of agents can be different, however, they are all  influenced by the same system --- we capture this effectively by having a common  measure that represents the aggregate of all population choices. Typically agents interacting once with  the system, make pure or deterministic choices in a myopic manner --- we again capture this using the aggregate measure and by modeling the decisions as the pure choices that maximize a type-dependent utility function of the current empirical measure. Such dynamics represent several real-world scenarios like concert queues, traffic-routing decisions,  online reviews, and participation games. 

To study such complex dynamics, that exhibit discontinuities (pure choices suddenly switch when the empirical distribution of aggregate choices crosses a measure-zero curve)  and randomness (agents of different types can arrive in any random order), we had to resort to differential inclusion (DI) based stochastic approximation framework (classical ODE-based framework is not sufficient). 
We primarily provide explicit characterization of some limit sets via the internally chain transitive sets of the driving DI, given the parameters of the game.  

Unlike the  majority of the literature that characterizes cyclic outcomes of smooth dynamical systems predominantly in two-dimensional systems, our results aid in identifying cyclic outcomes in higher dimensions --- this is achieved by constructing appropriate solutions of DI that pass through discontinuities. We also identify non-classical attractors that are not zeros of the driving field (as in the classical sense), but rather occur at points where a few action-choices balance out for a few agent types; basically these occur at discontinuities of the field or at switching points of the field.  

Beyond the theoretical analysis, we propose a numerical procedure to identify the singleton ICT sets or point limits,  for any given finite game. We also provide a graph theoretic method to identify the possibility of cyclic ICT sets --  this helped us in  illustrating the the emergence of cyclic outcomes in games with more than 3 actions. Finally, we consider  a queuing game and illustrate the emergence of cyclic outcomes for scenarios, dominated by agents with cyclic-preferences and avoid-the-crowd agents.  

\bibliographystyle{elsarticle-num}
\bibliography{reference1}

\appendix
\vspace{-3mm}
\section{Summary of key notations}
\begin{table}[htbp]
\centering
\scalebox{0.97}{
\begin{tabular}{|c|c|}
\hline
{\bf Notation}             & { \bf Description}                                                                                        \\ \hline
$\Theta$             & set of finite types                                                           \\ \hline
$\balpha$                  & population distribution across types in $\Theta$                                                             \\ \hline
$\A_\t$                & set of finite actions of type $\t$
\\ \hline
$\mathcal{P}(\Theta)/ \mathcal{P}(\A_\t)$        & set of probability measures on $\Theta$/$\A_\t$                                                                             \\ \hline
$\bmu_\t$                & empirical distribution of the players of type $\t$ over $\A_\t$    \\ \hline
${\mu}_\t^a$                & fraction of type $\t$ players who chose action $a$ \\ \hline
$u_\t$                & utility of a player of type $\t$                                                        \\ \hline
$\Bar{N}_\t(t)$           & fraction of agents of type $\t$ up to time $t$                          \\ \hline
$\Rmuc_\t^a (t) $             & fraction of type $\t$ players that have chosen action $a$ up to time $t$                          \\ \hline
$\Romga (t)$               & random aggregate population measure till time  epoch $t$        \\ \hline
${\bm\Gamma}(t) $               & history of $\Romga (t)$ till time $t$                                              \\ \hline
$h^{a,\ta}_\t$               & difference in utility functions for actions $a$ and $\ta$ of type $\t$ player                                                            \\ \hline
$\H^{a,\ta}_\t$ & set of points where the function $h^{a,\ta}_\t(\cdot)$ is zero (boundary region)
                                    \\ \hline
   $\H$ & union of boundary regions $\H^{a,\ta}_\t$  for all $(a,\ta,\t)$                               \\ \hline
    $\D$   & domain of the functions $\g(\cdot)$ and $h(\cdot)$        \\ \hline
    $\I(\bromga)$ & set of indices $(a,\ta,\t)$ for which $h^{a,\ta}_\t(\bromga)=0$ \\ \hline
 $\J(\bromga)$ & subclass of regions $\{\R_j\}$ which are adjacent to $\bromga$
    \\ \hline
    $\inu (\cdot)$ & linear interpolated process  \\ \hline
    $\C_{i\delta}$ & measurable set in euclidean space ${\mathbb R}^k$ for some $i$ and $\delta \in (0,1)$ \\ \hline
    $\B(\g_i,r)$ & closed ball of radius $r$ and centered at $\g_i$ for some $i$  \\ \hline
    $\co(\g_1,\ldots,\g_m)$ & convex hull of $\g_1,\ldots,\g_m$ \\ \hline
    $\nablah$ & gradient of the function  $h^{a, \ta}_\t(\cdot)$  at $\bromga$ \\ \hline
    $\SS^* $ & subclass of ICT  sets that are singletons \\ \hline
     $\LL_c^*$ & set of classical attractors \\ \hline
     $\SS_c^*$ & class of ICT sets arising from $\LL_c^*$ \\ \hline
     $\LL_f^*$ & set of Filippov attractors \\ \hline
      $\SS_f^*$ & class of ICT sets arising from  $\LL_f^*$ \\ \hline
       $\CC^* $ & cyclic ICT set  \\ \hline
 \end{tabular}}
\end{table}
\section{Appendix: Background on Differential Inclusions}\label{Appendix_DI}

\noindent Given an ODE in $\mathbb{R}^k$ such that 
\begin{equation}\label{eqn_append_ODE}
    \dot{\bromga} = \g(\bromga).
\end{equation}
The differential inclusion associated with the above ODE is defined below as in \cite{filippov2013differential}:
\begin{equation}\label{eqn_append_DI_Def}
    \dot{\bromga} \in \G(\bromga),
\end{equation}
where $\G$ is a set valued function which maps each point $\bromga \in \mathbb{R}^k$ to a set $\G(\bromga) \subset \mathbb{R}^k$, which is mathematically defined as:

\begin{equation}\label{DI_plane1}
   \dot{\bromga} \in \G(\bromga) = \bigcap_{N \in \mathcal{N}} \bigcap_{\delta > 0} \co \big\{\g(\bromga') : \|\bromga - \bromga'\| < \delta \setminus N\big\},
\end{equation}
where $\N$ represents the family of measure zero sets, and $\co$ denotes the closed convex hull. We need some assumptions on $\G$ as in \cite{benaim2005stochastic} as below.

\noindent {\bf Assumptions on $\G$:}
\begin{itemize}
    \item[(i)] $\G$ is a closed set-valued map.
    \item[(ii)] $\G(\bromga)$ a nonempty compact convex subset of $\mathbb{R}^k$ for all $\bromga \in \mathbb{R}^k$.
    \item[(iii)] There exists $c>0$ such that for all $\bromga \in \mathbb{R}^k$:
    $$
    \sup_{z \in \G(\bromga)}\|z\| \leq c(1+\|\bromga\|),
    $$
    where $\|\cdot\|$ denotes the norm on $\mathbb{R}^k$.
\end{itemize}

We now reproduce some important definitions from \cite{filippov2013differential} and \cite{benaim2005stochastic}, relevant for our paper. 
\begin{definition}\label{def_DI_solution}
    A solution for the differential inclusion given in \eqref{eqn_append_DI_Def} with initial point $\bromga(0) \in \mathbb{R}^k$ is an absolutely continuous mapping $\bromga: \mathbb{R} \to \mathbb{R}^k$ such that
    $$
    \dot{\bromga}(t) \in \G(\bromga(t)),
    $$
    for almost all $t \in \mathbb{R}$.
\end{definition}
     
\begin{definition}{[Perturbed solutions]} \label{def_perturbed_sol}
A continuous function $\persolu : \mathbb{R}_+ = [0, \infty) \to \mathbb{R}^k$ is called a perturbed solution of the DI given in \eqref{eqn_append_DI_Def}, if it satisfies the following set of conditions:
\begin{enumerate}
    \item[(i)] $\persolu$ is absolutely continuous.
    \item[(ii)] There exists a locally integrable function $\pert(\cdot)$ such that:
    \begin{enumerate}
        \item $\lim_{t \to \infty} \sup_{0 \leq v \leq T} \left\| \int_t^{t+v} \pert(s) \, ds \right\| = 0, \text{ for all } T > 0$, and
        \item 
        $\dot{\persolu}(t) - \pert(t) \in \G^{\delta(t)}(\persolu(t)) \text{ for almost every } t > 0,$
        for some function $\delta : [0, \infty) \to \mathbb{R}$ with $\delta(t) \to 0$ as $t \to \infty$. Here, 
        $$
        \G^\delta(z) := \{ \persolu \in \mathbb{R}^k : \exists \  z' : \|z' - z\| < \delta, \, d(\persolu, \G(z')) < \delta \},
        $$
        and 
        $$
        d(\persolu, {\mathcal D}) = \inf_{c \in {\mathcal D}} \|\persolu - c\|.
        $$
    \end{enumerate}
\end{enumerate}
\end{definition}

\begin{definition}{[Internally chain transitive (ICT) set]}\label{Def_ICT_set}
    A set $\ICT$ is said to be ICT, provided that $\ICT$ is compact and there exists a chain between $z$ and $\tilde{z}$ in $\ICT$, for all $z,\tilde{z} \in \ICT$.

     We say that there exists a chain between $z$ and $\tilde{z}$ in $\ICT$, if for every $\epsilon > 0$ and $T > 0$, there exists an integer $l \in \mathbb{N}$, solutions $\bromga_1,\bromga_2, \ldots, \bromga_l$ to DI \eqref{eqn_append_DI_Def}, and real numbers $t_1, t_2, \ldots, t_l$ greater than $T$ such that:
\begin{enumerate}
    \item[(a)] $\bromga_i(s) \in \ICT$ for all $0 \leq s \leq t_i$ and for all $i = 1, \ldots, l$,
    \item[(b)] $\|\bromga_i(t_i) - \bromga_{i+1}(0)\| \leq \epsilon$ for all $i = 1, \ldots, l-1$,
    \item[(c)] $\|\bromga_1(0) - z\| \leq \epsilon$ and $\|\bromga_l(t_l) - \tilde{z}\| \leq \epsilon$.
\end{enumerate}
The sequence $(\bromga_1, \ldots, \bromga_l)$ is called an $(\varepsilon, T)$ chain in $\ICT$ from $z$ to $\tilde{z}$ for $\G$.
\end{definition}

\begin{definition}{[Robbins-Monro conditions]}\label{def_robbins_monro}
The stochastic process $\left\{\Romgac^a(t)\right\}_{t > 0}$ given by \eqref{eqn_bromega_DI_iterates} satisfies the Robbins-Monro condition with martingale difference noise (see \cite{kushner2003stochastic}), if its characteristics satisfy the following:
 \begin{itemize}
     \item[(a)] $\left\{\gamma(t)\right\}_{t > 0}$ is a deterministic sequence.
     \item[(b)] $\left\{\pertc(t)\right\}$ is measurable with respect to ${\bm\Gamma}(t)$ for each $t >0$.
     \item[(c)] $\mathrm{E}\left[\pertc(t+1) \mid {\bm\Gamma}(t)\right]=0$.
 \end{itemize}
\end{definition}
\section{Appendix: Stochastic Approximation Proofs} \label{Appendix_stoch_Approx_proofs}
\noindent \textbf{Proof of Lemma \ref{lemma_Filippov}: } \label{proof_lemma_Filippov}
    By hypothesis,   as $\delta \to 0$,  $\B_{i\delta} \to \{\g^c_i\}$  for each $i$ and thus  $\bigcup_{i=1}^{m} \B_{i\delta} \to {\mathcal O}:=\{\g^c_1, \cdots, \g^c_m\}$. Clearly, $\co (\bigcup_{i=1}^{m} \B_{i\delta}) \to \co \left({\mathcal O}\right)$.

Since $\C_{i\delta} \subseteq \B_{i\delta}$ for all $(i, \delta)$, we trivially have $\bigcup_{i=1}^{m}\C_{i\delta} \subset  \bigcup_{i=1}^{m}\B_{i\delta}$ and  then $\co\left({\bigcup_{i=1}^{m}\C_{i\delta}}\right) \subset \co\left({{\bigcup_{i=1}^{m}\B_{i\delta}}}\right)$; this implies 
 $$
 \bigcap_{\delta} \co\left(\bigcup_{i=1}^{m}\C_{i\delta}\right) \subset \bigcap_{\delta} \co\left(\bigcup_{i=1}^{m} \B_{i\delta}\right) = \co(\mathcal O).$$

Again by hypothesis, ${\mathcal O} \subset \co\left(\bigcup_{i=1}^{m }\C_{i\delta}  \right) $ for each $\delta$, and thus  $ \co({\mathcal O}) \subset  \bigcap_{\delta} \co\left(\bigcup_{i=1}^{m}\C_{i\delta}\right)$.
\eop

\vspace{3mm}

\noindent\textbf{Proof of Theorem \ref{Thm_stoch_approx}: }  \label{proof_Thm_stoch_approx}
  To prove the result, we first apply proposition 1.4 and then proposition 1.3   of \cite{benaim2005stochastic}, by showing that our stochastic trajectory satisfies  Robbins-Monro
conditions (\ref{def_robbins_monro}, also see \cite{benaim2005stochastic})). Towards that, for ease of explanation, we reproduce and rewrite the equations \eqref{eqn_iterates_nu} and \eqref{eqn_cond_exp}   as  below and provide equivalent representation of sigma-algebra $\Gamma(\cdot)$:
\begin{eqnarray}
 \hspace{-8mm} \Romgac^a(t+1) \hspace{-2mm}&=& \hspace{-2mm}\Romgac^a(t) + \gamma(t+1)(\ind_{\{A(t+1) = a\}}-\Romgac^a(t)), \forall a \in \A, \text{ with } \gamma(t):=\frac{1}{t},\ t>0 \text{ and } \label{eqn_iterates_nu_proof}\\
    \hspace{-8mm} g^a(\Romga(t)) \hspace{-2mm}&:=& \hspace{-2mm} \mathrm{E}[\ind_{\{A(t+1) = a\}}-\Romgac^a(t) \mid \bm \Gamma(t)], \text{ where }\label{eqn_cond_exp_proof}\\
  \hspace{-8mm} {\bm\Gamma}(t) \hspace{-2mm}&:=& \hspace{-2mm} \sigma\{\Romga(s): s \le t\} \stackrel{`a\text{'}}{=} \sigma \{\Romga(0),A(s):s \leq t\}   \text{ for all } t > 0. \label{def_gamma_tt} 
\end{eqnarray}
The equality $`a$' follows straightforwardly from \eqref{eqn_iterates_nu_proof}. Then, martingale difference noise denoted by $\pertc(\cdot)$, for any $a \in \A$, is defined as below (see, for instance,  \cite[Chapter 4-5]{kushner2003stochastic}).
\begin{eqnarray}\label{eqn_martingale_difference}
    \pertc(t+1) \hspace{-2mm}&:=& \hspace{-2mm} \ind_{\{A(t+1) = a\}} - \Romgac^a(t) - \mathrm{E}[\ind_{\{A(t+1) = a\}}-\Romgac^a(t)) \mid \bm\Gamma(t)], \text{ for all } t > 0, \\
         \hspace{-2mm}&=& \hspace{-2mm} \ind_{\{A(t+1) = a\}} - \Romgac^a(t) - g^a(\Romga(t)), \text{ for all } t > 0. \label{eqn_mar_diff}
    \end{eqnarray}
   Using \eqref{eqn_iterates_nu_proof} and \eqref{eqn_mar_diff}, we get:
    \begin{equation*}
     \Romgac^a(t+1) = \Romgac^a(t) + \gamma(t+1)[\pertc(t+1)+ g^a(\Romga(t))], \mbox{ for all } a \in \A.
\end{equation*}
 The above can be rewritten as the following inclusion  (constructed using singleton sets): 
\begin{equation}\label{eqn_bromega_DI_iterates}
     \Romgac^a(t+1) - \Romgac^a(t) - \gamma(t+1)\pertc(t+1) \in \gamma(t+1)\{g^a(\Romga(t))\}, \text{ for all } a \in \A.
\end{equation}
To apply \cite[Proposition 1.4]{benaim2005stochastic}, we first show that the stochastic process $\{\Romgac^a(t)\}_{t > 0}$ given by \eqref{eqn_bromega_DI_iterates} satisfies the Robbins-Monro conditions of Definition \ref{def_robbins_monro} with martingale difference noise \eqref{eqn_martingale_difference} as below.
\begin{itemize}
    \item[(a)] As $\gamma(t)=\nicefrac{1}{t}$ (see \eqref{eqn_iterates_nu_proof}), the step size $\{\gamma(t)\}_{t \geq 1}$ is a deterministic sequence, also  $\substack{\sum_{t}\nicefrac{1}{t^2} < \infty}$.
    \item[(b)] $\pertc(t)$ defined in  \eqref{eqn_martingale_difference} is $\bm \Gamma(t)$ measurable using   \eqref{eqn_cond_exp_proof}-\eqref{def_gamma_tt}, as $\Gamma(t-1) \subset \Gamma(t)$.
Also,  by    \eqref{eqn_martingale_difference}, we have 
$ \mathrm{E}[\pertc(t+1) \mid {\bm\Gamma}(t)] = 0$.
\end{itemize}
Finally,  to apply \cite[Proposition 1.4]{benaim2005stochastic},  it is left to show that $E[\pertc(t+1)^2] < \infty$ for each $a,t$. From \eqref{eqn_martingale_difference} it suffices to prove   $E[\Romgac^a(t)^2 ] < \infty $ for each $a$ and $t$. 
But this follows from  \eqref{eqn_iterates_nu_proof} because:   when $\Romgac^a(0) \in [0, 1]$ for all $a$,   by induction  for each $t$, we have $\Romgac^a(t) \in [0,1]$ for each $a$ --- observe for each $a,$  $|\Romgac^a(t)| = |(1-\gamma(t) ) \Romgac^a(t-1) + \gamma(t) \ind_{\{A(t)\}} | \le 1$, as $\gamma(t) = \nicefrac{1}{t}$.
Now by applying   \cite[Proposition 1.4]{benaim2005stochastic}, we have:
\begin{equation}
    \lim_{l \to \infty} f^T(l)=0 \text{ almost surely}, \text{ for all } T > 0, \label{eqn_lim_f_l_T}
\end{equation}
where
\begin{eqnarray*}
    f^T(l) &=& \sup \left\{\left\|\sum_{t=l}^{\kappa-1}\gamma(t+1)\pertc(t+1)\right\|:  \kappa = l+1,\ldots,m(\zeta_{l+1}+T)\right\}, \text{ and }\\
    m(t) &=& \sup\{l\geq0: t \geq \zeta_l\} \text{ with } \zeta_l := \sum_{t=1}^{l}\gamma(t).
\end{eqnarray*}
Hence the hypothesis of  \cite[Proposition 1.3]{benaim2005stochastic} is  satisfied and therefore the linear interpolated process $\inu (\cdot)$, defined in  \eqref{eqn_interpolated_traj} is a bounded perturbed solution of the DI \eqref{eqn_general_DI} for all the sample paths for which  \eqref{eqn_lim_f_l_T} is true; in other words  in almost sure sense, as \eqref{eqn_lim_f_l_T}  is true almost surely.

Finally,  using \cite[Theorem 3.6]{benaim2005stochastic}, the limit set $L(\inu(t))$ of linear interpolated trajectory $\inu(t)$,
is ICT almost surely.  
\eop

\vspace{3mm}
 
\noindent\textbf{Proof of Lemma \ref{lemma_singleton_ICT}:}\label{proof_lemma_singleton_ICT}
Consider $x \in \Omega$ (defined in \eqref{eqn_Omega_space}) such that the limit set associated with $x$ is $L(\inu(x, \cdot)) = \{\bromga^*\}$. Assume, for contradiction,  that the linear interpolated trajectory  $\inu(x, t) \not\to \bromga^*$ as $t\to \infty$. Since $\inu(x, t)$ (see \eqref{eqn_interpolated_traj}) is continuous with respect to $t$, therefore, there exists $\epsilon > 0$, and a sequence $\{t_n\}_{n \geq 1}$ such that $t_n \to \infty$ and
\begin{equation}\label{eqn_single_contra}
    |\inu(x,t_n)- \bromga^*| \geq \epsilon \mbox{ for all } n \geq 1.
\end{equation}
 Observe that the sequence $\{\inu(x,t_n)\}_{n\ge 1}$ is bounded, even as $t_n \to \infty$. Therefore, by the Bolzano -Weierstrass theorem (see \cite{rudin1976principles}), there exists a convergent subsequence $\{\inu( x,t_{n_\kappa})\}_{\kappa \geq 1}$ with a limit $\bromga'$. Also, by the definition of the limit set, we know that $\bromga' \in L(\inu(x, \cdot))$.

However, $L(\inu(x, \cdot)) = \{\bromga^*\}$, implying $ \bromga' = \bromga^*$. Thus, for any $\epsilon>0$, there exists $K$ such that, 
$$
|\inu(x,t_{n_\kappa}) - \bromga^*| < \epsilon, \text{ for all } \kappa>K.
$$ 
This contradicts \eqref{eqn_single_contra}. Therefore, $\inu(x,t)$ converges to $ \bromga^*$ as $t \to \infty$.     
\eop

\section{Appendix: Characterization of ICT sets proofs}
\subsection{Game with two actions}
\textbf{Proof of Theorem \ref{lem_alm_sure_conv_two_act}: }  \label{proof_lem_alm_sure_conv_two_act}
Let   $\mathbb{S}' := \{\{\bromgac^*\} : \bromgac^* \in \mathbb{L} \}$   represent the class in the RHS of \eqref{Eqn_ICTs_Two_actions} with $\LL$ defined in the same RHS. And 
let the class $\SS^*$  of singleton ICT sets,  defined in Lemma ~\ref{lemma_singleton_ICT}, be  represented briefly by $\SS$. We will first show that $\SS = \SS'$.  Then  we will show  that all ICT sets are singleton sets. And then the result follows from Theorem~\ref{Thm_stoch_approx} and Lemma~\ref{lemma_singleton_ICT}.
%
%
%
%
%
Now the proof proceeds in the following steps.
\\
\underline{\textbf{Step 1:} To prove that  $\mathbb{S}$ coincides with $\mathbb{S}'$, i.e., $\mathbb{S} = \mathbb{S}'$.}

We first show that $\mathbb{S}' \subset \mathbb{S}$.  
To this end, let $\{\bromgac^*\} \in \mathbb{S}'$.  
From the DI~\eqref{eqn_DI_with_two_actions}, observe the following:\\
If $\bromgac^* \notin \mathcal{H}$ which implies $\g(\bromgac^*) = 0$, then the constant function $\bromgac(t) = \bromgac^*$ for all $t \geq 0$ is a solution to the DI~\eqref{eqn_DI_with_two_actions}. Therefore, $\{\bromgac^*\}$ is an ICT set.
If $\bromgac^* \in \mathcal{H}$ with $0 \in [\g_1^\infty(\bromgac^*), \g_2^\infty(\bromgac^*)]$, then by similar reasoning, $\bromgac(t) = \bromgac^*$ is also a solution of the DI~\eqref{eqn_DI_with_two_actions}, and hence $\{\bromgac^*\}$ is an ICT set.
Thus, in both cases, $\{\bromgac^*\} \in \mathbb{S}$, and we conclude that $\mathbb{S}' \subset \mathbb{S}$.

To prove the converse, suppose, for contradiction, that $\{\bromgac^*\} \in \mathbb{S}$ but $\{\bromgac^*\} \notin \mathbb{S}'$. 
We first consider the case where $\bromgac^* \notin \mathcal{H}$. Then, by the structure of $\SS'$, $\g(\bromgac^*) \neq 0$. Since $\g$ is continuous on $\cup_j \mathcal{R}_j$, there exists a neighborhood $\B(\bromgac^*,\delta)$ (for some $\delta > 0$) such that $\g(\bromgac)$ is either strictly positive or strictly negative for all $\bromgac \in \B(\bromgac^*,\delta)$ (see also the definition of the limits $\g_y^\infty(\bromgac)$ in~\eqref{eqn_g_limit}). 
In this case, due to the strictly positive or negative drift in the neighborhood of $\bromgac^*$, the singleton $\{\bromgac^*\}$ cannot be invariant. Hence, by~\cite[Lemma 3.5]{benaim2005stochastic}, it cannot be an ICT set. This leads to a contradiction.

Now consider the case where $\bromgac^* \in \mathcal{H}$, but $\{\bromgac^*\} \notin \mathbb{S}'$, which implies $0 \notin [\g_1^\infty(\bromgac^*), \g_2^\infty(\bromgac^*)]$. In this situation, the set of possible derivatives at $\bromgac^*$ is either strictly positive or strictly negative, and thus the constant trajectory $\bromgac(t) = \bromgac^*$ cannot be a solution to the DI~\eqref{eqn_DI_with_two_actions}. This again contradicts the assumption that $\{\bromgac^*\}$ is an ICT set.
Therefore, our assumption must be false, and it follows that $\mathbb{S} \subset \mathbb{S}'$.

\noindent\underline{\textbf{Step 2:} To show that $|\mathbb{S}| < \infty$.}

From~\eqref{eqn_fun_g_}, observe that in each interval $\mathcal{R}_j$, the function $\g(\bromgac)$ takes the form $\g(\bromgac) = b_j - \bromgac$, where $b_j$ is a constant. Therefore, the number of solutions to the equation $\g(\bromgac^*) = 0$ is at most $m$, i.e.,
\[
\left| \left\{ \bromgac^* : \g(\bromgac^*) = 0 \right\} \right| \leq m < |\mathcal{H}| + 1 < \infty,
\]
by Assumption~\ref{O.1}. 
Hence, the number of singleton ICT sets is finite.

\noindent\underline{\textbf{Step 3:} Non-singleton sets are not ICT sets.}

Suppose, for contradiction, that ${\mathcal F}$ is an ICT set that includes a nontrivial interval. Then there exists some $j$ such that the set ${\mathcal I} := {\mathcal F} \cap \mathcal{R}_j$ is non-empty. By Assumption~\ref{O.1}, we may further shrink ${\mathcal I}$, if necessary, so that
\[
\mathcal{I} \cap \left\{ \bromgac : \g(\bromgac) = 0 \right\} = \emptyset.
\]
Since $\g(\bromgac)$ has a constant sign on $\mathcal{I}$, any solution trajectory of the ODE starting at a point $a \in \mathcal{I}$ can move only in one direction, say, from $a$ to another point $b \in \mathcal{I}$, but not in the reverse direction. This violates the invariance and two-sided accessibility requirements stated in Definition~\ref{Def_ICT_set} of an ICT set.

By similar reasoning, no set $\mathcal{F} = \{\bromgac_i\}_{i > 1}$, whether finite, countable, or uncountable, can form an ICT set unless each element is isolated and satisfies the conditions for singleton ICT sets established earlier. 
\hfill$\square$

\subsection{Generic games with finite actions} \label{subsec_append_Generic_game_proofs}

Some proofs of this section   depend upon some common observations given below.
 By   definition, the sign of function $h^{a,\ta}_\t(\bromga)$ does not change at any $\bromga \in \R_j$, while it can change across different regions; this is true 
 for any $(a,\ta,\t)$. More precisely, the following statements are true for any $(a,\ta,\t)$ and $j \ne j'$:
\begin{eqnarray}\label{eqn_same_diff_sign in_regions}
    h^{a,\ta}_\t (\bromga_1) h^{a,\ta}_\t (\bromga_2)  &>& 0 \mbox{ for all }   \bromga_1 , \bromga_2 \in \R_j,  \mbox{ for any }  j,  \\ 
       h^{a,\ta}_\t (\bromga_1) h^{a,\ta}_\t (\bromga_2) & < & 0   \mbox{ for all } \bromga_1 \in \R_j,   \  \bromga_2 \in \R_{j'} \mbox{ implies } \H^{a, \ta}_\t  \cap  \overline{\R}_j \cap \overline{\R}_{j'} \ne \emptyset; \nonumber 
\end{eqnarray}
basically, the sign change (by continuity) across $  \R_j, \R_{j'} $ is possible only  if    the corresponding $\H^{a,\ta}_\t$ is in their common boundary (also observe  $h^{a,\ta}_\t(\bromga) \ne 0$ if $\bromga \in \cup_j \R_j$).

\begin{lem}\label{lem_three_act_not_zero}
Consider any set $\F \subset \overline{\R}_j$ for some $1\le j\le m$. If $\bromga \in \F \cap \R_j $   with $\g(\bromga) \ne 0 $, then $\F$ is not an ICT set. 
\end{lem}
\textbf{Proof.}
    From \eqref{eqn_new_general_ode1}, 
   $g^a(\bromga') = b^a_{j} - \bromgac'^a$  for all $a$ and all $\bromga' = (\bromgac'^1, \cdots, \bromgac'^{k-1}) \in \R_j$, where $\{b^a_{j}\}$ are constants;  by the hypothesis of the lemma, there exists a ball $\B(\bromga,\delta) \subset \R_j$ for some $\delta > 0$ such that $\mbox{sign}(g^a(\bromga'))$ is the same (and not equal to 0) for all $\bromga' \in \B(\bromga,\delta)$ and all  $a$.   Consider  any  other $\tilde{\bromga} \in \F$ with  $\tilde{\bromga} \neq \bromga'$; then  there does not exist any $(\epsilon,T)$ chain,  from  $\tilde{\bromga}$ to $\bromga$,   for every sufficiently small $\epsilon > 0$ and for any $T>0$ (see the Definition \ref{Def_ICT_set}); observe that any solution starting in $\R_j$ and till it hits the boundary of $\R_j$, has the following structure:
   $$
   \bromga(t) =(\bromga(0) - \b_j)e^{-t} + \b_j,
   $$and hence if one has a solution from some $\bromga'$ to $\bromga''$ (basically if $\bromga(0)= \bromga'$ and $\bromga(t) = \bromga''$ for some $t$), then one can not have a solution (which fully lies in $\R_j$) from $\bromga''$ to $\bromga'$. For the same reasons,  if $\F = \{\bromga\}$ is a singleton, then it is not invariant (see \cite[Lemma 3.5]{benaim2005stochastic}). Thus, $\F$ is not an ICT set. \eop
   
\vspace{3mm}
\noindent\textbf{Proof of Theorem \ref{Thm_gen_one}:}  To begin with, if $\g(\bromga) \neq 0$ for some $\bromga \in  \F \cap   \R_j$, then, by Lemma \ref{lem_three_act_not_zero} of \ref{subsec_append_Generic_game_proofs}, the set $\F$ is not ICT.

Now, suppose $\g(\bromga) = 0$ for some  $\bromga \in \F \cap \R_j$. From \eqref{eqn_new_general_ode1}--\eqref{eqn_new_general_ode2}, we have $\dot{\bromga}^a= b^a_{j} - \bromgac^a$ for all $a \in \R_j$, which implies $\bromga = \b_j$. 
Furthermore, $\bromga$ is an asymptotically and locally stable attractor, with the entire $\R_j$ as its domain of attraction. Therefore, any solution starting from $\bromga$ satisfies $\bromga(t) = \bromga$ for all $t$.

Now, consider any other $\bromga' \in \F$ (such a point exists since $|\F| > 1$). Then, for any $\epsilon <  d( \bromga, \bromga')$, where $d$ denotes the Euclidean distance, there exists no $(\epsilon,T)$-path from $\bromga$ to $\bromga'$, as per Definition \ref{Def_ICT_set}. Thus, $\F$ is not an ICT.
\eop

\vspace{3mm}

\noindent\textbf{Proof of Lemma \ref{Lemma_no_Solution_on_H}:} \label{proof_Lemma_no_Solution_on_H}
Assume, for the sake of contradiction, that there exists a non-constant solution $\bromga(t)$ of the DI \eqref{eqn_general_DI}, such that 
$\bromga(t) \in \H^{a,\ta}_\t$ for all $t \in [s_1, s_2]$, for some $s_2 > s_1$. Since the solution is non-constant, there exist $t, s \in [s_1, s_2]$ with $t \ne s$ such that $\bromga(t) \ne \bromga(s)$. By continuity of solution to the DI, this implies that $\bromga([s_1, s_2])$ contains a non-singleton connected set.

By further invoking Assumption \ref{asm_card_ind_set_finite}, there exists a sub-interval of $[s_1, s_2]$ over which the index set $\J(\bromga(t)) = \{i, j\}$ remains fixed for some pair of region indices $i, j$. With a slight abuse of notation, we denote this sub-interval again by $[s_1, s_2]$.

%

Thus, in all, we have the following for some fixed pair of region indices $\{i,j\}$:
\begin{equation}
   h^{a,\ta}_\t(\bromga(t)) = 0 \quad \text{and} \quad \J(\bromga(t)) = \{i, j\} \quad \text{for all } t \in [s_1, s_2]. \label{lemma41}
\end{equation}
Recall that for any $t$ in the above interval, $\bromga(t) \in \H^{a,\ta}_\t$, and therefore, $\bm \G(\bromga(t))$ is the convex hull formed by the vectors $\b_i$ and $\b_j$. Consequently, the DI \eqref{eqn_general_DI} simplifies to:
\begin{eqnarray}
    \bm \G (\bromga(t)) = \left\{ \lambda \b_i + (1 - \lambda) \b_j - \bromga(t) : \lambda \in [0,1] \right\} \quad \text{for all } t \in [s_1, s_2]. \label{lemma411}
\end{eqnarray}
Furthermore, using the first part of equation \eqref{lemma41}, we apply the chain rule (composition rule) to obtain for all $t \in [s_1, s_2]$:
\begin{eqnarray} \label{eqn_der_h_inn_pro}
    0 = \frac{d}{dt} \left( h^{a,\ta}_\t(\bromga(t)) \right) = \nablaht^T \dot{\bromga}(t).
\end{eqnarray}
By the definition of the DI, we have $\dot{\bromga}(t) \in \bm \G(\bromga(t))$ for almost all $t \in [s_1, s_2]$. For such $t$, define $\V(t) := \dot{\bromga}(t)$. Then from \eqref{lemma411}, there exists $\lambda(t) \in [0,1]$ such that:
\[
\V(t) = \lambda(t) \b_i + (1 - \lambda(t)) \b_j - \bromga(t) 
= \lambda(t) (\b_i - \bromga(t)) + (1 - \lambda(t)) (\b_j - \bromga(t)).
\]
Substituting into \eqref{eqn_der_h_inn_pro}, we obtain:
\[
0 = \nablaht^T \V(t) = \lambda(t) \nablaht^T (\b_i - \bromga(t)) + (1 - \lambda(t)) \nablaht^T (\b_j - \bromga(t)).
\]
However, by Assumption \ref{Asm_filippov_new}, we have that both inner products $\nablaht^T (\b_i - \bromga(t))$ and $\nablaht^T (\b_j - \bromga(t))$ are of the same sign and strictly non-zero, which implies:
\[
\nablaht^T \V(t) \ne 0,
\]
leading to a contradiction (since $\lambda(t) \in [0,1]$ and the convex combination of two non-zero quantities of the same sign cannot be zero).\eop

\vspace{3mm}

 \noindent\textbf{Proof of Lemma \ref{lem_solu_tau1_tau2}: }\label{proof_lem_solu_tau1_tau2}
Under the given hypothesis on $\J(\bromga_{j_1})$, we have $|\I(\bromga_{j_1})| = 1$, and let $\I(\bromga_{j_1}) = \{(a, \ta, \t)\}$. As a first step, we aim to show that the function $\bromga(\cdot)$, defined in \eqref{eqn_bromega_in_regions}, satisfies $\bromga(t) \in \R_{j_2}$ for some $t > \tau_{j_1}$, and subsequently we will establish that this inclusion holds for all $t \in (\tau_{j_1}, \tau_{j_2})$.

By continuity of the extended function \eqref{eqn_bromega_in_regions}, and using the hypothesis that $\J(\bromga_{j_1}) = \{j_1, j_2\}$, it follows that $\bromga(t) \in \R_{j_1} \cup \R_{j_2}$ for all $t \in (\tau_{j_1}, \tau_{j_1} + \epsilon)$, for some $\epsilon > 0$ sufficiently small. 
If, for some such $t$, it holds that $h^{a,\ta}_\t(\bromga_{j_0}) \cdot h^{a,\ta}_\t(\bromga(t)) < 0$, then by \eqref{eqn_same_diff_sign in_regions} we conclude that $\bromga(t) \notin \R_{j_1}$, which implies $\bromga(t) \in \R_{j_2}$ (recalling the initial condition $\bromga_{j_0} \in \R_{j_1}$).
Thus, it suffices to establish that the sign of $h^{a,\ta}_\t(\bromga(t))$ changes from that of $h^{a,\ta}_\t(\bromga_{j_0})$.

%

By definition, $h^{a, \ta}_\t(\bromga_{j_1}) = 0$, and hence:
\begin{eqnarray*}
0 &=& h^{a, \ta}_\t (\bromga_{j_0}) \cdot h^{a, \ta}_\t (\bromga_{j_1}) \\
&=& h^{a, \ta}_\t (\bromga_{j_0}) \cdot \left[ h^{a,\ta}_\t \left( (\bromga_{j_0} - \b_{j_1}) e^{-\tau_{j_1}} + \b_{j_1} \right) \right] \qquad \text{(from \eqref{eqn_bromega_in_regions})} \\
&=& h^{a, \ta}_\t (\bromga_{j_0})^2 e^{-\tau_{j_1}} + h^{a, \ta}_\t (\bromga_{j_0}) \cdot h^{a, \ta}_\t(\b_{j_1}) (1 - e^{-\tau_{j_1}}) \qquad \text{(from \eqref{Eqn_linear_h_satisfies})}.
\end{eqnarray*}
This yields:
\[
h^{a, \ta}_\t (\bromga_{j_0}) \cdot h^{a, \ta}_\t(\b_{j_1}) < 0,
\]
which implies:
\[
h^{a, \ta}_\t (\bromga_{j_0}) \cdot h^{a, \ta}_\t (\b_{j_2}) < 0,
\]
because $h^{a, \ta}_\t(\b_{j_1}) \cdot h^{a, \ta}_\t(\b_{j_2}) > 0$ by \ref{asm_gen_2_implication} and \eqref{Eqn_linear_h_satisfies}.

Finally, by direct computation, the function $\bromga(\cdot)$ defined in \eqref{eqn_bromega_in_regions} satisfies the following for all $t > \tau_{j_1}$:
\begin{eqnarray}
h^{a, \ta}_\t(\bromga_{j_0}) \cdot h^{a, \ta}_\t(\bromga(t)) 
&=& h^{a, \ta}_\t(\bromga_{j_0}) \cdot \left[ h^{a, \ta}_\t \left( (\bromga_{j_1} - \b_{j_2}) e^{-(t - \tau_{j_1})} + \b_{j_2} \right) \right] \nonumber \\
&=& h^{a, \ta}_\t(\bromga_{j_0}) \cdot h^{a, \ta}_\t(\b_{j_2}) (1 - e^{-(t - \tau_{j_1})}) < 0, \label{Eqn_oppo_sign}
\end{eqnarray}
where the final inequality holds since $h^{a, \ta}_\t(\bromga_{j_1}) = 0$.

Therefore, the sign of $h^{a, \ta}_\t(\bromga(t))$ differs from that of $h^{a, \ta}_\t(\bromga_{j_0})$, implying $\bromga(t) \in \R_{j_2}$ for some $t \in (\tau_{j_1}, \tau_{j_1} + \epsilon)$; denote one such $t$ by $t'$.
Now, for $t \in (t', \tau_{j_2})$, the function $\bromga(\cdot)$ in \eqref{eqn_bromega_in_regions} can be interpreted as the solution of the DI \eqref{eqn_general_DI} with initial point $\bromga(t') \in \R_{j_2}$, as specified in \eqref{Eqn_DI_Solution_in_Rj}.

This completes the proof (refer to the definition of the exit time $\tau_{j_2}$ as provided in the hypothesis of the lemma).\eop

\begin{lem}\label{lem_solu_tau1_tau2_start_line_H}
Assume \ref{Asm_filippov_new}. Consider  $\bromga(0) \in  \H^{a,\ta}_\t$  for some  $(a,\ta, \t)$, and assume $|\I (\bromga(0))| = 1$. Then there exists a solution exactly as in \eqref{eqn_solution_under_G.2} with initial condition $\bromga(0)$ and with $j_1 \in \J (\bromga(0))$. 
\end{lem}
\textbf{Proof.}
By hypothesis, say $\J(\bromga(0)) = \{i,j\}$ for some $i\ne j$. Define two functions $\tilde{\bromga}(s)$ and $\bar{\bromga}(s)$ for all $s >0$ as below (as $h^{a,\ta}_\t(\bromga(0)) = 0$ and using \eqref{Eqn_linear_h_satisfies}):
\begin{eqnarray}
    \tilde{\bromga}(s) &:=& \bromga(0)e^s+\b_i(1-e^s),  \text{ then } h_\t^{a,\ta}(\tilde{\bromga}(s)) = h_\t^{a,\ta}(\b_i)(1-e^s) \label{eqn_int_h_1},\\
    \bar{\bromga}(s) &:=& \bromga(0)e^s+\b_j(1-e^s),  \text{  then } h_\t^{a,\ta}(\bar{\bromga}(s)) = h_\t^{a,\ta}(\b_j)(1-e^s).\label{eqn_int_h_2}
\end{eqnarray}
Since $\bromga(0) \in \overline\R_i \cap \overline \R_j$, one can choose $s$ sufficiently small such that $\tilde{\bromga}(s), \bar{\bromga}(s) \in \R_i \cup \R_j$; they are not in $\H^{a,\ta}_\t$ as, $h^{a,\ta}_\t(\bar{\bromga}(s)) \ne 0 $ and $h^{a,\ta}_\t(\tilde{\bromga}(s)) \ne 0 $ by  \ref{asm_bj_notin_regions} and \eqref{eqn_int_h_1}-\eqref{eqn_int_h_2}. Further, by \ref{Asm_filippov_new}, $h_\t^{a,\ta}(\b_i) h_\t^{a,\ta}(\b_j) >0$ and hence using \eqref{eqn_int_h_1}-\eqref{eqn_int_h_2},
\begin{eqnarray}
    h_\t^{a,\ta}(\tilde{\bromga}(s))h_\t^{a,\ta}(\bar{\bromga}(s)) = h_\t^{a,\ta}(\b_j)    h_\t^{a,\ta}(\b_i)(1-e^s)^2 > 0.
\end{eqnarray}
Then, using  \eqref{eqn_same_diff_sign in_regions} and since $\J(\bromga(0))=\{i,j\}$, it follows that   both $\tilde{\bromga}(s)$ and $\bar{\bromga}(s)$ are in the same region, either in $\R_i$ or $\R_j$. Without loss of generality, consider that they both belong to the region $\R_i$.  Now, recall $\tilde \bromga(s) = \bromga(0) e^{s} + \b_i (1-e^s)$, which implies 
$$
\bromga(0) = (\tilde \bromga(s) -\b_i)e^{-s} + \b_i.
$$
Observe that, the above is a solution of the DI \eqref{eqn_general_DI} starting from $\tilde{\bromga}(s) \in \R_i$ which hits $\bromga(0)$ at $\tau = s.$ Now, using arguments of Lemma \ref{lem_solu_tau1_tau2}, there will exist a solution of the form \eqref{Eqn_solution_two_regions} to the DI \eqref{eqn_general_DI}, a segment of which can be seen as solution starting from $\bromga(0)$. One can extend further as before to obtain solution like \eqref{eqn_solution_under_G.2}.  \eop

\vspace{3mm}

\noindent\textbf{Proof of Theorem \ref{Cyclic_ICT}: }\label{proof_Cyclic_ICT} We complete the proof in two steps. 

To prove that $\CC^*$ is an ICT set (see Definition~\ref{Def_ICT_set}), we show that for any $\z, \z' \in \CC^*$ and any $T > 0$, there exists a solution $\bromga(t)$ of DI~\eqref{eqn_general_DI} such that 
\[
\bromga(0) = \z, \quad \bromga(t_1) = \z', \quad \text{and} \quad \bromga(t) \in \CC^* \;\; \forall \, 0 \leq t \leq t_1,
\]
for some $t_1 > T$. This is shown in \textbf{Step~(ii)}. Prior to that, in \textbf{Step~(i)}, we establish the existence of a solution $\bromga(t)$ of the form~\eqref{solution_DI_general}, lying entirely within $\CC^*$, that starts at $\z$ and returns to $\z$, for any $\z \in \CC^*$, which leads to the required proof of \textbf{Step~(ii)}. 

\vspace{2mm}
\noindent\textbf{Step (i):} 
Consider any $\z \in \CC^*$ implies  $\z \in{\S}_i$ for some $i$ with $1 \leq i \leq l$ (see \eqref{Eqn_Ic_cycle}). Hence, there exists a time epoch $t'< \tau^*_{j_i}$ such that 
\begin{eqnarray}\label{eqn_cycle_ict_first}
    \z = (\bromga_{j_{i-1}}^*-\b_{j_{i}})e^{-\left(t'-\tau_{j_{i-1}}^*\right)} + \b_{j_{i}},\text{ implies } t'=\log\left(\frac{\bromgac_{j_{i-1}}^{*a}-b^a_{j_{i}}}{z^a-b^a_{j_{i}}}\right)+\tau_{j_{i-1}}^* \text{ for any } a . 
\end{eqnarray} 
Further, under \ref{asm_cycle_one}, there exists a solution $\bromga(t)$ of DI \eqref{eqn_general_DI}  as in \eqref{solution_DI_general}  with $\bromga(0) = \z \in \S_i\subset\R_{j_i}$, which touches  the boundary of $\R_{j_i}$ at $\bromga_{j_i}^*$ at time $\tau_{j_i}^*-t'$, with $t'$ given in \eqref{eqn_cycle_ict_first}. From the structure of the solution in \eqref{solution_DI_general}, we have $\z=(z^a)_{a\in \A} \in \R_{j_i}$, while  \ref{asm_cycle_one} ensures $\b_{j_{i}}=(b^a_{j_{i}})_{a\in \A} \notin \R_{j_i}$, $\tau_{j_{i-1}}^* < \infty$. Together, these imply $t' < \infty$. Further, it is easy to verify that  $\frac{\bromgac_{j_{i-1}}^{*a}-b^a_{j_{i}}}{z^a-b^a_{j_{i}}} > 1,$ i.e., $t'>\tau_{j_{i-1}}^*.$

Moreover, based on the  structure of $\S_i$, it follows that $\bromga(t) \in \CC^*$ for all $t \in [0,\tau_{j_i}^*-t']$.
This solution $\bromga(t)$ can be extended beyond $t= \tau_{j_i}^*-t'$ by concatenating it with a new solution that starts at $\bromga^*_{j_i}$ again as in \eqref{solution_DI_general}: 
extend the function $\bromga(\cdot)$ from time $  \tau_{j_i}^*-t'$ to $ \tau_{j_i}^*-t'+ \tau_{j_{l}}^*-\tau_{j_i}^*$, we obtain:
  $$
  \bromga( \tau_{j_i}^*-t') = \bromga^*_{j_i} \mbox{ and } \bromga( \tau_{j_i}^*-t' + \tau_{j_{l}}^*-\tau_{j_i}^*) = \bromga^*_{j_l}.
  $$
To prove that this concatenated function $\bromga(\cdot)$ is indeed a solution of DI \eqref{eqn_general_DI} with $\bromga(0) = \z$,  it is sufficient to show that it is absolutely continuous (direct from the construction) and satisfies the DI \eqref{eqn_general_DI}, for almost all $t$. Clearly, this concatenated function $\bromga(\cdot)$ is also satisfy the DI \eqref{eqn_general_DI} for almost all $t \in [0,\tau_{j_i}^*-t' + \tau_{j_{l}}^*-\tau_{j_i}^*]$ with $\bromga(0) = \z$ and $\bromga\left(\tau_{j_i}^*-t' + \tau_{j_{l}}^*-\tau_{j_i}^*\right)=\bromga_{j_l}^*$
. Thus,  $\bromga(t) \in \CC^*$ for all such $t$.
Recall from \ref{asm_cycle_one} that $\bromga_{j_0}^* = \bromga_{j_l}^*$.
Therefore, by concatenating the function but now starting from $\bromga_{j_0}^*$, we obtain a solution 
$\bromga(t)$ for all $t \in [0, \ \tau_{j_{l}}^*]$ $\left(\text{observe~}   \tau_{j_{l}}^* = (\tau_{j_i}^*-t')+(\tau_{j_{l}}^*-\tau_{j_i}^*)+ \tau_{j_{i-1}}^*+(t'-\tau_{j_{i-1}}^*)\right)$, with initial and terminal values equal to $\z$. Moreover, $\bromga(t) \in \CC^*$ for all such $t$.


\vspace{2mm}
\noindent \textbf{Step (ii):}
Consider any $\z, \z' \in \CC^*$, and any $T>0$ (see Definition \ref{Def_ICT_set}) and define $\varrho := \lceil \nicefrac{T}{ \tau^*_{j_l}} \rceil$. 
First,  construct a DI solution as in \textbf{Step (i)}
such that $\bromga(0) = \z$ and $\bromga(\tau^*_{j_l}) = \z$. Then concatenate this solution repeatedly for $\varrho$-rounds, i.e., consider $\bromga (t)$ for all $t \in [0, \varrho  \tau^*_{j_l}]$ with the following properties:
$$
\bromga( t_\kappa   ) = \z, \text { for all }  t_\kappa := \kappa \tau^*_{j_l} \mbox{ with } \kappa  \in \{0, \cdots, \varrho\} \mbox{ and such that } \bromga(t) \in \CC^* \mbox{ for all } t \in [0, \varrho  \tau^*_{j_l}].
$$
Finally, concatenate the above with another DI solution that starts from $\z$ and ends at $\z'$ as before. Now, define $t_1 :=  \varrho  \tau^*_{j_l}$ such that $\bromga(t_1) = \z'$, then clearly $t_1 > T$ and $\bromga(t) \in \CC^*$ for all $0 \leq t \leq t_1$.
\eop

\subsection{Game with three actions} 
\noindent \textbf{Proof of Lemma \ref{lem_three_action_on_H}: }\label{proof_lem_three_action_on_H}
   We will first prove part (ii) and then part (i).

\noindent\textbf{Part (ii):} To begin with, assume that 
$\J(\bromga) = \{i,j\}$ for all $\bromga \in \H^{a,\ta}_\t$. We aim to show the existence of a solution $\bromga(t)$ of the DI \eqref{eqn_general_DI} such that $\bromga(t) \in \H^{a,\ta}_\t$ for all $t \geq 0$.

By the definition of a solution to a DI (see \ref{def_DI_solution}) and the characterization of the convex hull $\G$ in \eqref{eqn_general_DI}, this amounts to finding a solution of the following ODE:
\begin{equation}\label{eqn_find_lembda}
    \dot{\bromga} = \b_i + \lambda(\bromga)(\b_j - \b_i) - \bromga,
\end{equation}
where $\lambda(\bromga) \in [0,1]$ must be chosen appropriately so that the following condition holds:
\begin{equation*}
    h^{a,\ta}_\t(\bromga(t)) = 0 \quad \text{for all } t \geq 0.
\end{equation*}
The above implies that $\nicefrac{dh^{a,\ta}_\t(\bromga(t))}{dt}  = 0$ for all $t \geq 0$, which in turn gives (by the chain rule and using \eqref{eqn_find_lembda}):
\[
\nablaht^T \cdot \left( \b_i + \lambda(\bromga(t)) (\b_j - \b_i) - \bromga(t) \right) = 0.
\]
Using standard algebraic manipulation, we can now derive a candidate expression for $\lambda(\bromga(t))$. Since we are assuming that $\bromga(t) \in \H^{a,\ta}_\t$ for all $t$, we have $h^{a,\ta}_\t(\bromga(t)) = 0$ for all $t$. Substituting this into the equation yields:
\begin{equation}\label{eqn_val_lambda}
    \lambda(\bromga(t)) = \frac{\nablaht^T \cdot (\b_i - \bromga(t))}{\nablaht^T \cdot (\b_i - \b_j)} 
    = \frac{h^{a,\ta}_\t(\b_i)}{h^{a,\ta}_\t(\b_i) - h^{a,\ta}_\t(\b_j)} \quad \text{for all } \bromga(t) \in \H^{a,\ta}_\t, ~ \forall t \geq 0.
\end{equation}
Recall that $h^{a,\ta}_\t (\b_i) h^{a,\ta}_\t (\b_j) < 0$, and observe from \eqref{eqn_val_lambda} that, for every $\bromga(t) \in \H^{a,\ta}_\t$, there exists a $\lambda(\bromga(t))$ such that $0 \leq \lambda(\bromga(t)) \leq 1$.
Now, substitute $\lambda(\bromga(t))$ into \eqref{eqn_find_lembda}; then we obtain, using the definition of $\bromga_{ij}^\infty$ from \eqref{eqn_limit_sol_on_h},
\begin{eqnarray*}
    \dot{\bromga} &=& \b_i - (\b_i - \b_j) \frac{h^{a,\ta}_\t(\b_i)}{h^{a,\ta}_\t(\b_i) - h^{a,\ta}_\t(\b_j)} - \bromga = \bromga_{ij}^\infty - \bromga.
\end{eqnarray*}
The solution to the above linear ODE is given by:
\[
\bromga(t) = \bromga_{ij}^\infty(1 - e^{-t}) + \bromga(0)e^{-t}, \quad \text{for all } t \geq 0,
\]
with asymptotic limit $\bromga_{ij}^\infty$. Observe that $\bromga_{ij}^\infty$ defined in \eqref{Eqn_Solution_on_H} can be equivalently expressed as:
\begin{eqnarray*}
    \bromga_{ij}^\infty = \frac{h^{a,\ta}_\t(\b_i)}{h^{a,\ta}_\t(\b_i) - h^{a,\ta}_\t(\b_j)} \b_j + \frac{-h^{a,\ta}_\t(\b_j)}{h^{a,\ta}_\t(\b_i) - h^{a,\ta}_\t(\b_j)} \b_i.
\end{eqnarray*}
Recall again that $h^{a,\ta}_\t(\b_i) h^{a,\ta}_\t(\b_j) < 0$, which implies the two coefficients in the above expression are strictly positive and sum to one. Therefore, $\bromga_{ij}^\infty$ is a convex combination of $\b_i$ and $\b_j$, implying $\bromga_{ij}^\infty \in \D$ (since $\D$ is convex). Furthermore, $h^{a,\ta}_\t(\bromga_{ij}^\infty) = 0$, hence $\bromga_{ij}^\infty \in \H^{a,\ta}_\t$.
Furthermore, we have $\bromga_{ij}^\infty \in \co\{\b_i, \b_j\}$, which implies $\mathbf{0} \in \co\{\b_i - \bromga_{ij}^\infty, \b_j - \bromga_{ij}^\infty\}$. Therefore, $\bromga(t) = \bromga_{ij}^\infty$ for all $t \geq 0$ is a (constant) solution to the DI with initial condition $\bromga(0) = \bromga_{ij}^\infty$. Thus, $\{\bromga_{ij}^\infty\}$ forms a singleton ICT set (see Definition \ref{Def_ICT_set}).

Now consider the hypothesis of part (ii), i.e., $\J(\bromga_{ij}^\infty) = \{i,j\}$. By our assumption \ref{asm_card_ind_set_finite}, we have that $\J(\bromga) = \{i,j\}$ for all $\bromga \in \H \cap \B(\bromga_{ij}^\infty, \delta)$ for some $\delta > 0$. The above analysis applies to any initial condition $\bromga(0) \in \H \cap \B(\bromga_{ij}^\infty, \delta)$, which completes the proof of part (ii). 

\noindent\textbf{Part (i):} Now, suppose $\J(\bromga)$ is not constant over all $\bromga \in \H^{a,\ta}_\t$, but remains fixed (equal to $\{i,j\}$) for some initial portion of the solution $\bromga(t)$, say for $t \in [0, \tau]$, as specified in the hypothesis. Then, clearly, the trajectory $\bromga(t)$ given by \eqref{Eqn_Solution_on_H} is a valid solution to the DI \eqref{eqn_general_DI} on the interval $[0,\tau]$, as argued above.
\eop
\section{Connection of singleton ICT sets with MT-MFE}\label{append_connec_ict_equil}

\noindent\textbf{Proof of Theorem \ref{thm_connection_MT_MFE}}: 
We consider two cases to prove the theorem as outlined below.

\noindent\textbf{Case 1:} Suppose $\b_j \in \LL^*_c$, which implies $\b_j \in \R_j$ (see \eqref{eqn_set_AA}). Then, by equation \eqref{eqn_new_general_ode2}, we have
\[
b_j^{a} = \sum_{\t} \alpha_\t \prod_{\ta \ne a} \ind_{\{ h_\t^{a, \ta}(\b_j) > 0 \}} \quad \text{for all } a.
\]
Due to the strict inequality in the above expression, for each type $\t$, there exists a unique action $a_\t \in \A_\t$ such that
\[
\prod_{\ta \ne a_\t} \ind_{\left\{ h_\t^{a_\t, \ta}(\b_j) > 0 \right\}} > 0 \quad \text{and} \quad \arg\max_{a' \in \A_\t} u_\t(a', \b_j) = \{a_\t\}.
\]
Define $\bmu_\t^* := \delta_{\{a_\t\}}$ for each $\t$. Then, the tuple $(\bmu_\t^*)_\t$ constitutes an MT-MFE (see Definition~\ref{defn_MT-MFE}) with
\[
\b_j = \sum_\t \alpha_\t \bmu_\t^*.
\]
Hence, $\b_j$ is an aggregate-MFE (see Definition~\ref{defn_MT-MFE}).

 \noindent\textbf{Case 2:} Let $\bromga^* \in \LL_f^*$ as defined in \eqref{eqn_limit_sol_h_line_set}. By the definition of $\b_j$ in \eqref{eqn_new_general_ode2}, and the definition of any region $\R_j$, it follows that for each $j \in \J(\bromga^*)$ (see \eqref{Eqn_Jomega}) and each type $\t$, there exists a unique action $e_{j,\t} \in \A_\t$ such that
\begin{equation}
\label{Eqn_region_description}
h_\t^{e_{j,\t}, a}(\bromga) > 0 \quad \text{for all } a \ne e_{j,\t}, \text{ for all } \t, \text{ and all } \bromga \in \R_j.
\end{equation}
Moreover, the vector $\b_j = (b_j^a)$ is given by
\[
b_j^a = \sum_{\t} \alpha_\t \ind_{\{ a = e_{j,\t} \}} \quad \text{for any } a.
\]
Further, by \eqref{eqn_limit_sol_h_line_set}, we have $\bromga^* \in \H$ and it can be expressed as a convex combination of the vectors $\{\b_j\}$, that is,
\begin{equation}
\label{Eqn_convex_comb}
\bromga^* = \sum_{j \in \J(\bromga^*)} \lambda_j \b_j, \quad \text{with} \quad
\bromgac^{*a} = \sum_{j \in \J(\bromga^*)} \lambda_j \sum_\t \alpha_\t \ind_{\{ a = e_{j, \t} \}} \quad \text{for all } a \quad (\text{see } \eqref{Eqn_region_description}).
\end{equation}
Define, for each $\t$, the vector $\bmu_\t^* := (\mu_\t^{*a})_a$ as
\begin{equation}
\label{eqn_const_mu_the_agre}
\mu_\t^{*a} := \sum_{j \in \J(\bromga^*)} \lambda_j \ind_{\{ a = e_{j, \t} \}} \quad \text{for all } a.
\end{equation}
Then observe that
\[
\bromga^* = \sum_\t \alpha_\t \bmu_\t^*.
\]

Next, we show that $\bromga^*$ is an aggregate-MFE by proving that the tuple $(\bmu^*_\t)_\t$ constitutes an MT-MFE (see Definition~\ref{defn_MT-MFE}). In other words, we will establish that
\[
\supp(\bmu_\t^*) \subset \arg\max_{a \in \A_\t} u_\t(a, \bromga^*) \quad \text{for every } \t.
\]
Suppose, for the sake of contradiction, that this is not the case. Then there exists some type $\t$ and some action $\ta \in \supp(\bmu_\t^*)$ such that
\[
\ta \notin \arg\max_{a \in \A_\t} u_\t(a, \bromga^*).
\]
This implies that there exists an action $a^* \ne \ta$ such that
\[
u_\t(a^*, \bromga^*) > u_\t(\ta, \bromga^*), \quad \text{or equivalently,} \quad h^{\ta, a^*}_\t(\bromga^*) < 0.
\]
By continuity of the function $h^{\ta, a^*}_\t(\cdot)$ (see Assumption~\ref{A.1}), there exists an $\epsilon > 0$ such that for all $\bromga$ in a neighborhood $\mathcal{B}(\bromga^*, \epsilon)$, we have
\begin{equation}
\label{equilibrium}
h^{\ta, a^*}_\t(\bromga) < 0 \quad \text{for all } \bromga \in \mathcal{B}(\bromga^*, \epsilon).
\end{equation}
On the other hand, by \eqref{eqn_const_mu_the_agre}, the support of $\bmu_\t^*$ is given by
\[
\supp(\bmu_\t^*) = \{ a : a = e_{j,\t} \text{ for some } j \in \J(\bromga^*) \text{ with } \lambda_j > 0 \}.
\]
Hence, there exists some $\bar{j} \in \J(\bromga^*)$ such that $\ta = e_{\bar{j}, \t}$.
However, by the definition of the region $\R_{\bar{j}}$ (see \eqref{Eqn_region_description}), we must have
\[
h^{\ta, a^*}_\t(\bromga) > 0 \quad \text{for all } \bromga \in \R_{\bar{j}}.
\]
Moreover, by the definition of $\J(\bromga^*)$, the neighborhood $\mathcal{B}(\bromga^*, \epsilon)$ intersects $\R_{\bar{j}}$, i.e.,
\[
\mathcal{B}(\bromga^*, \epsilon) \cap \R_{\bar{j}} \neq \emptyset.
\]
This contradicts \eqref{equilibrium}.\eop

\section{Conditions for existence of a cyclic-ICT set} \label{subsub_sec_cond_exist_cycle}
Before proceeding further,  
we present some important observations related to $\{\R_j\}$ regions  and their boundaries $\{\H_\t^{a,\ta}\}$.
For any $\bromga_1,\bromga_2 \notin \H$ and any  index $ (a, \ta, \t)$, 
there will be no change in sign of the function values,  $h^{a,\ta}_\t (\bromga_1)$ and  $h^{a,\ta}_\t (\bromga_2)$  if the border region   $\H_\t^{a, \ta} $ does not intersect the line joining $\bromga_1$ and $\bromga_2$ (which follows from the continuity of  $h_\t^{a, \ta}(\cdot)$ and the intermediate value theorem). Thus, we have:
\begin{eqnarray}
    \label{Eqn_sign_changes}
\mbox{for any } \bromga_1, \bromga_2 \notin \H, \mbox{   }  \  
h_\t^{a, \ta} (\bromga_1) h_\t^{a, \ta} (\bromga_2) > 0 \text{  if }  \ \overleftrightarrow
{\bromga_1,\bromga_2} \cap  \H_\t^{a, \ta} = \emptyset, 
\end{eqnarray}
where $\overleftrightarrow{\bromga_1, \bromga_2}$ is the line segment joining $\bromga_1, \bromga_2$.  Secondly, if some initial condition $\bromga(0) \in \R_{j_1}$ for some $j_1 \notin \AA$ (see \eqref{eqn_set_AA}), then there exit a time epoch $\tau_{j_1}$ of the solution, defined in \eqref{Eqn_DI_Solution_in_Rj} and \eqref{eqn_hititing_time1}, is finite  and  it hits some boundary $\H^{a, \ta}_\t$.  In other words, $\bromga(\tau_{j_1}) \in \H^{a, \ta}_\t$ and then   by definition, $0=h^{a,\ta}_\t(\bromga(\tau_{j_1})) =  h(\bromga(0) ) e^{-\tau_{j_1}}  + h^{a,\ta}_\t(\b_j)(1-e^{-\tau_{j_1}}) $; further   
as  the sign of $h^{a,\ta}_\t$ does not change in $\R_{j_1}$, by \eqref{Eqn_sign_changes} it follows that: 
\begin{eqnarray}
\label{Eqn_sign_changes_two}
  h^{a, \ta}_\t(\bromga(0)) h^{a, \ta}_\t (\b_{j_1} )  < 0,  \mbox{ and hence, }  h^{a, \ta}_\t(\bromga) h^{a, \ta}_\t (\b_{j_1} )  < 0 \mbox{ for all }  \bromga \in \R_{j_1}. 
\end{eqnarray}

\subsection{For generic number of actions}
We aim to obtain the conditions under which \ref{asm_cycle_one} is satisfied, which provides the conditions for the existence of a cyclic ICT set.  From \eqref{solution_DI_general}, the conditions in \ref{asm_cycle_one} is satisfied if: 
\begin{itemize} [label=\textbf{C.2}, ref=\textbf{C.2}]
\item the inequalities in \eqref{Eqn_conditions_for_solution} are satisfied with $L = l$, and $j_{l} = j_0$, for some $\{j_0, \cdots, j_{l}\} \in \AA^c$, where $\AA^c$ is  the complementary set of $\AA$, the set of indices corresponding to the regions with classical attractors as defined in \eqref{eqn_set_AA}; and \label{asm_c.2}
\end{itemize}
\begin{itemize} [label=\textbf{C.3}, ref=\textbf{C.3}]
\item  there exists a (recursive) solution in terms of  $\{\bromga^*_{j_0}, \cdots, \bromga^*_{j_{l-1}}\}$ and 
 $\{\tau^*_{1}, \cdots, \tau^*_{l}\}$
for the following set of equations, with $\tau^*_{0} = 0$, $\bromga_{j_0}^* \in \D_{j_0} $ with $\D_{j_0} :=  \H_{j_0} \cap \overline{\R}_{j_0}\cap \overline{\R}_{j_{1}}$, and
{\small
\begin{eqnarray}
  \bromga^*_{j_{i}}  \in  \D_{j_i} :=  \H_{j_i} \cap \overline{\R}_{j_i}\cap \overline{\R}_{j_{i+1}}, \text{ and, } \bromga^*_{j_{i}} \hspace{-2mm}&=&\hspace{-2mm} (\bromga_{j_{i-1}}^*-\b_{j_{i}})e^{-(\tau^*_{i}-\tau^*_{i-1})} + \b_{j_{i}}  ,\substack{\forall \ i \in \{1,\ldots,l-1\}},  \label{eqn_bromega_i_iteration} \hspace{8mm}\\
\text{ and finally, }    \bromga^*_{j_0} \hspace{-2mm}&=&\hspace{-2mm} (\bromga_{j_{l-1}}^*-\b_{j_l})e^{-(\tau^*_{l}-\tau^*_{l-1})} + \b_{j_l}. \label{eqn_bromega_one_iteration}
\end{eqnarray}}\label{asm_c.3}
\end{itemize}
\vspace{-4mm}
If there exists a solution  of the above set of equations, then it would satisfy the following  set of equations (with $\oplus$ representing modulo $l$ addition): 
\begin{eqnarray}
\frac{\bromgac^{a *}_{j_i}-b^a_{j_{i\modl 1}}}{ 
        \bromgac^{a*}_{j_{i\modl 1}} - b_{j_{i\modl 1}}^a} 
        =
         \frac{\bromgac^{\ta *}_{j_i}-b^\ta_{j_{i\modl 1}}}{ 
        \bromgac^{\ta *}_{j_{i\modl 1}} - b_{j_{i\modl 1}}^\ta}  \text{  for each  } a \ne \ta \text{ and each } i. \label{eqn_the_diff_bromega_ratio}
\end{eqnarray}
We would assume {\bf C.2} and then 
obtain simpler conditions for the existence of solutions in {\bf C.3} by the existence of a fixed point of some appropriately constructed fixed-point equations.

Towards this, we define a sequence of functions $\{\psi_i\}_{i \leq l}$,  such that the composite function 
$\psi := \psi_l \circ \psi_{l-1}\circ \cdots \circ \psi_1$ represents the required fixed point equation. 

From \eqref{eqn_bromega_i_iteration} with $i=1$,  the second component of the solution $\bromga_{j_1}^*$ satisfies     $h_{j_1} (\bromga^*_{j_1}) =0$  and the first two components  of the solution, $\bromga^*_{j_0}, \bromga^*_{j_1}$ are related to each other via the second part of \eqref{eqn_bromega_i_iteration}. 
Hence, we define the first map $\psi_1$, that reflects exactly these two  relations, using a clever trick as in \eqref{eqn_the_diff_bromega_ratio} that avoids the corresponding $\tau$
components\footnote{\label{foot_psione} As discussed before, as $j_0, j_1 \notin \AA$, there exists a solution from any $\bromga(0) = \bromga \in \H_{j_0}\cap {\overline \R}_{j_1}$, which hits $\H_{j_1}$, and hence $h_{j_1} (\bromga)h_{j_1} \left(\b_{j_1}\right) < 0$ as in \eqref{Eqn_sign_changes_two}; this implies  $\nicefrac{ - h_{j_1}\left(\b_{j_1}\right)} { \left(h_{j_1} (\bromga ) - h_{j_1}  \left(\b_{j_1}\right) \right)} \in [0, 1] $. Thus $\psi_1(\bromga) \in \D.$ }:
\begin{eqnarray}\label{eqn_fun_sai_1}
 \psi_1 (\bromga  ) := \frac{ - h_{j_1}(\b_{j_1})} { h_{j_1} (\bromga ) - h_{j_1}  (\b_{j_1}) } \bromga  +     \left (1- \frac{ - h_{j_1}(\b_{j_1})} { h_{j_1} (\bromga ) - h_{j_1}  (\b_{j_1}) }     \right )    \b_{j_1} \text{ for all  } \bromga \in \H_{j_0} \cap {\overline \R}_{j_1}.
\end{eqnarray}
One can immediately observe that if there exists a solution as in {\bf C.3}, then it satisfies   $\bromga_{j_1}^*  =  \psi_1 (\bromga^*_{j_{0}}),$ and  $h_{j_1}(\bromga^*_{j_1}) = 0$; this can  be verified by direct substitution and using simple algebra. We consider the domain of $\psi_1$ as $\H_{j_0}\cap {\overline \R}_{j_1}$ (since  $\bromga^*_{j_0} \in \D_{j_0} \subset \H_{j_0}\cap {\overline \R}_{j_1}$), see also footnote\footref{foot_psione}. Furthermore, $h_{j_1} (\psi_1(\bromga)) = 0$ for any $\bromga  $ (can verify by directly simplifying $h_{j_1}(\psi_1(\bromga))$), the range of $\psi_1$ is a subset of $\H_{j_1}$, i.e., $\psi_1 : \H_{j_0} \cap {\overline \R}_{j_1} \to \H_{j_1}$.

Using a similar approach and for analogous reasons, we define the following functions, one for each $i \in \{1, \ldots, l\}$:   
\begin{eqnarray}\label{eqn_fun_sai_i}
    \psi_i (\bromga  ) := \frac{ - h_{j_i}(\b_{j_i})} { h_{j_i} (\bromga ) - h_{j_i}  (\b_{j_i}) } \bromga  +     \left (1- \frac{ - h_{j_i}(\b_{j_i})} { h_{j_i} (\bromga ) - h_{j_i}  (\b_{j_i}) }     \right )    \b_{j_i}, \  \mbox{ for  all } \bromga \in \H_{j_{i-1}} \cap \overline{\R}_{j_i}.
\end{eqnarray}
Once again  any solution of {\bf C.3}, satisfies   $\bromga_{j_i}^*  =  \psi_i (\bromga^*_{j_{i-1}}),$ and  $h_{j_i}(\bromga^*_{j_i}) = 0$; thus one can  further consider $\psi_i : \H_{j_{i-1}} \cap {\overline \R}_{j_i}\to \H_{j_i}$.

At $i=l$, the range of $\psi_l$ is $\H_{j_l} = \H_{j_0}$, by {\bf C.2} and hence the solution of {\bf C.3}, if  exists, satisfies the following fixed point equation constructed using the composite function formed using  $\{\psi_i\}_{i \leq l}$ functions (see~\eqref{eqn_bromega_i_iteration}):
\begin{eqnarray}
    \bromga^*_{j_0} = \psi (  \bromga^*_{j_0}  ) \text{, where, }  \psi := \psi_l \circ \psi_{l-1}\circ \cdots \circ \psi_1 \text{ and }  \psi: \H_{j_0} \cap \overline{\R}_{j_1} \to \H_{j_0}.
\end{eqnarray} 
The idea is to identify the conditions under which
the well-known Brouwer's fixed-point theorem can be applied to the function $\psi(\cdot)$  ---  that would establish the existence of a fixed point in $\D_{j_0} \subset \H_{j_0}\cap {\overline \R}_{j_1}$, which in turn establishes the existence of a solution in {\bf C.3} and hence that of the cyclic ICT set by Theorem \ref{Cyclic_ICT}. 

To begin with observe $\psi(\cdot)$ is a continuous function. 
And then to well define the function $\psi$ that satisfies the hypothesis of Brouwer's fixed point theorem,  we would  require an appropriate domain, constructed using  smaller (but sufficient to hold required points) domains for constituent functions $\{\psi_i\}$  and hence require: 
\begin{itemize}
   \item  a sequence of domains $\{\Z_i\}$ such that the corresponding range,  $\psi_i (\Z_i) \subset \Z_{i+1}$ for each $i$ (this would ensure the composite function $\psi$ is well defined with domain as $\Z_1$);
    \item each such domain should intersect with points of interest  (potential solution of \eqref{eqn_bromega_i_iteration}), i.e., $\Z_i \cap \D_{j_i} \ne \emptyset$ for each $i$; and
     \item finally, we require that $\psi(\Z_1) \subset \Z_1$.
\end{itemize}
Towards achieving the above, 
first define recursively $\bar \psi_i:=  \psi_i  \circ \bar \psi_{i-1}$, with $\bar \psi_1 = \psi_1$ and then
the first set of conditions are:
\begin{eqnarray}
  {\bar \psi}_i (\D_{j_{i-1}}) \subset \D_{j_i} \mbox{ for each }  i, \label{eqn_D_j_1_subset_D_j_i}
\end{eqnarray}
or an easier condition (but less general)
\begin{eqnarray}
   \psi_i (\D_{j_{i-1}}) \subset \D_{j_i} \mbox{ for each }  i \in \{1, \cdots, l \}, \mbox{ with, recall, } \D_{j_l} = \D_{j_0}. \label{eqn_D_j_1_subset_D_j_i_sai}
\end{eqnarray}
The above provides a test to check cycles, derivation of simpler conditions is for future study. 
In the current paper, 
we will delve into more details for the special case with three actions.
 
\subsection{For three actions: conditions to test the existence of cyclic ICT sets} 
We first assume \ref{asm_gen_2_implication} holds and then derive the required conditions.  

The equation \eqref{eqn_fun_sai_i} for the special case with three actions,  can be written using single dimensional equations:  for all $i \in \{1,\ldots,l\}$, one of the components (say $\bromgac^1_{j_{i-1}}$) of the input $\bromga_{j_{i-1}} = (\bromgac_{j_{i-1}}^1, \bromgac_{j_{i-1}}^2) \in \D_{j_{i-1}}$ to the function $\psi_i$  can be obtained from the second component $\bromgac^2_{j_{i-1}}$, as from \eqref{Eqn_h_linear}  $\bromgac^1_{j_{i-1}} = \eta_{j_{i-1}} \bromgac^2_{j_{i-1}} + c_{j_{i-1}}$; recall  $h_{j_{i-1}} (\bromga_{j_{i-1}}) = 0$ for all $\bromga_{j_{i-1}}$ that are inputs to function $\psi_i$.

By linearity of $h_{j_i}$ and connectedness of sets $\{\R_j\}$ (and since $\bromga \in \H $ is represented by $\bromgac^2$), we have that 
\begin{eqnarray}
\label{Eqn_Dji_in_twoD}
\D_{j_i} = \H_{j_i}\cap \overline{\R}_{j_i} \cap\overline{\R}_{j_{i+1}}  = \left \{ \bromga = (\bromgac^1, \bromgac^2): \bromgac^1 = \eta_{j_i} \bromgac^2- c_{j_i},  
\mbox{ and, }
\bromgac^2 \in [s_{j_i}, e_{j_i}]  \right \} \ \forall \ i,
\end{eqnarray}
for some $s_{j_i}, e_{j_i}$, where $0\le s_{j_i}\le e_{j_i} \le 1$ ensures the corresponding $\D_{j_i} \ne \emptyset$. 

We begin with $\psi_1$,   represent its inputs by $\bromga_{j_0} = (\bromgac_{j_0}^1, \bromgac_{j_0}^2) \in \D_{j_0}$ and considering only the second component of $\psi_1$ function (renaming this as $\psi_1$ with slight abuse of notation): 
\begin{eqnarray}
\psi_1(\bromgac_{j_0}^2) \hspace{-2mm} &=& \hspace{-2mm} \frac{\tilde\phi_{1,1} + \tilde\phi_{1,2} \bromgac_{j_0}^2}{\tilde\phi_{1,3} + \tilde\phi_{1,4} \bromgac_{j_0}^2}, \substack{\text{ with }  \tilde \phi_{1,1} =  (c_{j_1}-c_{j_0})b_{j_1}^2, \  \tilde\phi_{1, 2} =(\eta_{j_0} -\eta_{j_1})b_{j_1}^2 - h_{j_1}(\b_{j_1}), \\ \hspace{2cm} \tilde\phi_{1, 3}= c_{j_1}-c_{j_0}- h_{j_1}(\b_{j_1}), \ \tilde\phi_{1,4} = \eta_{j_0} -\eta_{j_1}.}\label{eqn_composition_sai_1}\\
\psi_i(\bromgac_{j_{i-1}}^2) \hspace{-2mm} &=&  \hspace{-2mm} \frac{\tilde\phi_{i,1} + \tilde\phi_{i,2} \bromgac_{j_{i-1}}^2}{\tilde\phi_{i,3} + \tilde\phi_{i,4} \bromgac_{j_{i-1}}^2}, \substack{\text{ with }  \tilde \phi_{i,1} =  (c_{j_i}-c_{j_{i-1}})b_{j_i}^2, \  \tilde\phi_{i, 2} =(\eta_{j_{i-1}} -\eta_{j_i})b_{j_i}^2 - h_{j_i}(\b_{j_i}),\\ \hspace{0.3cm} \tilde\phi_{i, 3}= c_{j_i}-c_{j_{i-1}}- h_{j_i}(\b_{j_i}), \tilde\phi_{i,4} = \eta_{j_{i-1}} -\eta_{j_i} \text{ for all } i \in \{2,\ldots,l\}.}\label{eqn_composition_sai_2}
\end{eqnarray}
 
Define $ \phi_{1,i}:=\tilde \phi_{1,i}$  for all $i \in \{1,\ldots,4\}$ and $ \bar{\psi}_1 := \psi_1$. Then, using simple induction, we get the following composition function:
\begin{eqnarray}
\hspace{-5mm}\bar{\psi}_{i}(\bromgac_{j_0}^2)\hspace{-3mm}&=&\hspace{-3mm} (\psi_i \circ\bar{\psi}_{i-1} ) (\bromgac_{j_0}^2) = \frac{ {\tilde \phi}_{i,1} +  {\tilde \phi}_{i,2} \bar{\psi}_{i-1}(\bromgac_{j_0}^2)}{ {\tilde \phi}_{i,3} +  {\tilde \phi}_{i,4} \bar{\psi}_{i-1}(\bromgac_{j_0}^2)} = \frac{ \phi_{i,1} +  \phi_{i,2}  \bromgac_{j_0}^2}{ \phi_{i,3} +  \phi_{i,4} \bromgac_{j_0}^2}, \ \substack{\forall \ i \in \{2,\ldots,l\},  \text{ where }} \label{eqn_composition_sai_i}\\
\hspace{-5mm}  \tilde \phi_{i,1} \hspace{-3mm}&=&\hspace{-3mm}  \substack{(c_{j_i} -c_{j_{i-1}}) b_{j_i}^2, \   \tilde\phi_{i, 2} =(\eta_{j_{i-1}} -\eta_{j_i})  b_{j_i}^2- h_{j_i}(\b_{j_i}), \ \tilde\phi_{i, 3}= c_{j_i}-c_{j{i-1}}- h_{j_i}(\b_{j_i}), \tilde\phi_{i, 4} = (\eta_{j_{i-1}} -\eta_{j_i}),} \label{cyclic_condition}\\
\hspace{-5mm} \phi_{i,1} \hspace{-3mm}&=& \hspace{-3mm} {\tilde \phi}_{i,1} \phi_{i-1,3} + {\tilde \phi}_{i,2} \phi_{i-1,1}, \ \phi_{i,2} = {\tilde \phi}_{i,1} \phi_{i-1,4} + {\tilde \phi}_{i,2} \phi_{i-1,2}, \label{eqn_relation_phi_til_phi_1} \text{ and} \\
\hspace{-5mm} \phi_{i,3}  \hspace{-3mm}&=& \hspace{-3mm} {\tilde \phi}_{i,3} \phi_{i-1,3} + {\tilde \phi}_{i,4}  \phi_{i-1,1}, \ \phi_{i,4} = {\tilde \phi}_{i,3} \phi_{i-1,4} + {\tilde \phi}_{i,4}  \phi_{i-1,2}.\label{eqn_relation_phi_til_phi_2}
\end{eqnarray}
We are now ready to provide two sets of conditions for the existence of a cyclic ICT set. The first set of conditions is based on \eqref{eqn_D_j_1_subset_D_j_i_sai}.  From equations \eqref{Eqn_Dji_in_twoD},\eqref{eqn_composition_sai_1} and \eqref{eqn_composition_sai_2} the function $\psi_i (\D_{j_{i-1}})$ is a connected set  and hence   it is clear that \eqref{eqn_D_j_1_subset_D_j_i_sai} is satisfied if
\begin{eqnarray}
s_{j_i}   \le  \psi_i (s_{j_{i-1}})  \le e_{j_i} \mbox{ and }  s_{j_i}   \le  \psi_i (e_{j_{i-1}})  \le e_{j_i}  \mbox{ for each }  i.  \nonumber
\end{eqnarray}
\underline{The first set of conditions for the existence of cyclic ICT using \eqref{eqn_composition_sai_2} is:}
\begin{eqnarray}
\label{eqn_cond_cyclic_ict_set-}
\substack{ \tilde \phi_{i,3} s_{j_i}  + \tilde \phi_{i,4}    s_{j_{i-1}} s_{j_i} - \tilde \phi_{i,1}   - \tilde \phi_{i,2}  s_{j_{i-1}} \le 0,  \mbox{ and, }
\tilde \phi_{i,3} e_{j_i}  + \tilde \phi_{i,4}   s_{j_{i-1}} e_{j_i}   - \tilde \phi_{i,1}   - \tilde \phi_{i,2}  s_{j_{i-1}}} \hspace{-3mm} &\ge&\hspace{-3mm} 0, \nonumber
\\
\hspace{-4mm}\substack{ \tilde \phi_{i,3} s_{j_i}  + \tilde \phi_{i,4}    e_{j_{i-1}} s_{j_i} - \tilde \phi_{i,1}   - \tilde \phi_{i,2}  e_{j_{i-1}} \le 0,  \mbox{ and, }
\tilde \phi_{i,3} e_{j_i}  + \tilde \phi_{i,4}  e_{j_{i-1}}  e_{j_i}  - \tilde \phi_{i,1}   - \tilde \phi_{i,2}  e_{j_{i-1}}} \hspace{-3mm}&\ge &\hspace{-3mm} 0, \nonumber  \substack{ \ \forall \  i \ \in \{1,\cdots,l\},}  \nonumber
\end{eqnarray}
where $\tilde \phi_{i,\kappa}$ for all $\kappa \in \{1,\ldots,4\}$ as in \eqref{eqn_composition_sai_1} and \eqref{eqn_composition_sai_2}.

From equations \eqref{Eqn_Dji_in_twoD},\eqref{eqn_composition_sai_1} and \eqref{eqn_composition_sai_i} the function $\psi_i (\D_{j_{i-1}})$ is a connected set  and hence   it is clear that \eqref{eqn_D_j_1_subset_D_j_i} is satisfied if
\begin{eqnarray}
s_{j_i}   \le  \bar{\psi}_i (s_{j_{0}})  \le e_{j_i} \text{ and }  s_{j_i}   \le  \bar{\psi}_i (e_{j_{0}})  \le e_{j_i}  \text{ for each }  i.  \nonumber
\end{eqnarray}
\underline{The second set of conditions for the existence of cyclic ICT  using \eqref{eqn_composition_sai_i} is:}
\begin{eqnarray}
{\small \phi_{i,3} s_{j_i}  + \phi_{i,4} s_{j_0}    s_{j_i} - \phi_{i,1}   - \phi_{i,2}  s_{j_{0}} \le 0,  \mbox{ and, }
 \phi_{i,3} e_{j_i}  +  \phi_{i,4}   s_{j_{0}} e_{j_i}   -  \phi_{i,1}   -  \phi_{i,2}  s_{j_{0}}} \hspace{-2mm}&\ge&\hspace{-2mm} 0 \nonumber
\\
 {\small \hspace{-4mm}\phi_{i,3} s_{j_i}  +  \phi_{i,4}    e_{j_{0}} s_{j_i} -  \phi_{i,1}   -  \phi_{i,2}  e_{j_{0}} \le 0,  \mbox{ and, }
 \phi_{i,3} e_{j_i}  +  \phi_{i,4}  e_{j_{0}}  e_{j_i}  -  \phi_{i,1}   -  \phi_{i,2}  e_{j_{0}}}\hspace{-2mm} &\ge & \hspace{-2mm} 0 \ \substack{\forall \
     i \ \in \{1,\cdots,l\},}   \nonumber
\end{eqnarray}
where $\phi_{i,\kappa}$ for all $\kappa \in \{1,\ldots,4\}$ as in \eqref{cyclic_condition},\eqref{eqn_relation_phi_til_phi_1} and \eqref{eqn_relation_phi_til_phi_2}. 

\section{Analysis of queuing game}\label{sec_append_Analysis_of_queuing_game}
\noindent\textbf{Proof of Theorem \ref{thm_appli_1}:}
Recall that $c \geq 1$ and consider $ \nicefrac{1}{(2+c)} < \rho < \nicefrac{1}{2}$. Then the $\{\Q_v\}$ regions are as shown in Figure \ref{fig:first_two} (Left) with borders in \eqref{eqn_cs_cm_bound_lines1}-\eqref{eqn_cs_cm_bound_lines2}, and corresponding $\{\e_v\}$ and $\{\b_v\}$ vectors are given in \eqref{eqn_b_val_CS_CM1}-\eqref{eqn_b_val_CS_CM2}. Hence, we have:
\begin{eqnarray}\label{eqn_new_h_equations}
\substack{h^{3,1}_{\CMr}(\b_1)h^{3,1}_{\CMr} (\b_2)} \hspace{-2mm} &=& \hspace{-2mm} \substack{h^{1,2}_{\CMr}(\b_2)h^{1,2}_{\CMr} (\b_3) = h^{1,2}_{\CMr}(\b_4)h^{1,2}_{\CMr} (\b_5) = 1,  h^{2,3}_{\CMr}(\b_3)h^{2,3}_{\CMr} (\b_1) =  (2-3\alpha_\CMr)(2-(c+2)\alpha_\CMr),} \hspace{4mm}\\
\substack{h^{1,2}_\CSr(\b_2)h^{(1,2)}_\CSr(\b_5)}\hspace{-2mm} &=& \hspace{-2mm}\substack{\rho(\alpha_\CSr(1-\rho)+\rho), \ \ h^{1,2}_{\CSr}(\b_3)h^{1,2}_{\CSr} (\b_4) = \alpha_\CMr(1+\rho)-\rho.}\label{eqn_new_h_equations1}
\end{eqnarray}
Using \eqref{Eqn_Qv_details} and \eqref{eqn_b_val_CS_CM1}-\eqref{eqn_b_val_CS_CM2}, we obtain:
\begin{eqnarray}\nonumber
    h^{3,1}_{\CMr}(\bromga) &=& h_{(1,2)}(\bromga), \  h^{1,2}_{\CMr}(\bromga)=h_{(2,3)}(\bromga)=h_{(4,5)}(\bromga), \  
    h^{2,3}_{\CMr}(\bromga)=h_{(3,1)}(\bromga), \\
    h^{1,2}_\CSr(\bromga) &=& h_{(3,4)}(\bromga)=h_{(2,5)}(\bromga), \nonumber
\end{eqnarray}
and hence the following holds using \eqref{eqn_new_h_equations}-\eqref{eqn_new_h_equations1}:
\begin{eqnarray*}
h_{(1,2)}(\b_1)h_{(1,2)}(\b_2) \hspace{-2mm} &=& \hspace{-2mm} h_{(2,3)}(\b_2)h_{(2,3)}(\b_3)=h_{(4,5)}(\b_4)h_{(4,5)}(\b_5)=1,  \\
h_{(3,1)}(\b_3)h_{(3,1)}(\b_1)\hspace{-2mm}&=&\hspace{-2mm}(2-3\alpha_\CMr)(2-(c+2)\alpha_\CMr), \ h_{(2,5)}(\b_2)h_{(2,5)}(\b_5)=\rho(\alpha_\CSr(1-\rho)+\rho), \\
\ h_{(3,4)}(\b_3)h_{(3,4)}(\b_4)\hspace{-2mm} &=& \hspace{-2mm}\alpha_\CMr(1+\rho)-\rho.
\end{eqnarray*}
If $\alpha> \nicefrac{2}{3}$, then $  \{(1,2),(2,3),(3,1)\} \subset \I^c$ (see \eqref{eqn_set_ind_Ic}) and  $h_{(1,2)}(\b_1)= h_{(2,3)}(\b_2) = -1 <0$ and  $h_{(3,1)}(\b_3)=2-3\alpha_\CMr<0$, implying that RV graph has cycle. It is easy to verify that there does not exist any other possibility of cycle.

On the other hand, if $\alpha_\CMr < \nicefrac{\rho}{(1+\rho)}$, then $(3,4) \in \I^*$ (see \eqref{eqn_set_ind_Ic}) and the point $\left(\frac{1}{1+\rho},\frac{\rho}{1+\rho}\right)$ lies at the intersection of the border line $h_{(3,4)}$ and the line segment $\overleftrightarrow{\b_3,\b_4}$. Thus, $\left(\frac{1}{1+\rho},\frac{\rho}{1+\rho}\right)$ is a Filippov attractor (see \eqref{eqn_limit_sol_h_line_set}). Similarly once can show that $\left(\frac{2c+1}{3(c+2)},\frac{1}{3}\right)$ is also a Filippov attractor if and only if $\frac{c+5}{3(c+2)} < \alpha_{\CMr}$.

A close inspection of Figure \ref{fig:first_two} (Left) and \eqref{eqn_b_val_CS_CM1}-\eqref{eqn_b_val_CS_CM2} reveals that   $\b_v \notin \Q_v$   for each $v \ne 3$; for example, $\b_5$ is on $\{\bromgac^1+\bromgac^2 = 1\}$, but $\Q_5$ is away from this line. Likewise, $\b_1$ lies on the $\bromgac^2 = 0$ line but $\Q_1$ is away from it. Now, observe that  $\b_3$ can belong to in its own region $\Q_3$ (or form a classical attractor) iff the point $(\alpha_{\CSr},\alpha_{\CMr})$ should lie (i) to the right side of the line $h^{2,3}_{\CMr}$ and $h^{1,2}_{\CMr}$, (ii) above of the line $h^{1,2}_{\CSr}$, which holds precisely when $\frac{\rho}{1+\rho} < \alpha_\CMr < \frac{2}{3}$.

Note that analogous results also hold for $\rho < \nicefrac{1}{(c+2)}$, which can be verified easily.
 \eop

\vspace{3mm}
 
\noindent\textbf{Proof of Theorem \ref{thm_appli_2}:}
Consider $\rho < \nicefrac{2}{(c+1)}$. Then the $\{\Q_v\}$ regions are as shown in Figure \ref{fig:next_two} (Left) with borders given in \eqref{eqn_cs_cm_bound_lines1}-\eqref{eqn_cs_cm_bound_lines2}, and $h^{a,\ta}_{\nuu} (\bromga) = \bromgac^\ta - \bromgac^a, \  \forall \ a\ne \ta$. The corresponding $\{\e_v\}$ and $\{\b_v\}$ vectors are given in 
\eqref{eqn_gen_e_vec_cs_cm_anu1}-\eqref{eqn_gen_e_vec_cs_cm_anu3}. Hence, we have:
\begin{eqnarray*}
 \substack{ h^{3,1}_{\CMr}(\b_1)h^{3,1}_{\CMr} (\b_2)} \hspace{-2mm}&=& \hspace{-2mm} \substack{h^{1,2}_{\CMr}(\b_3)h^{1,2}_{\CMr} (\b_4) = h_{\CMr}^{1,2}(\b_7)h_{\CMr}^{1,2}(\b_9) = h_{\CMr}^{3,1}(\b_6)h_{\CMr}^{3,1}(\b_{10})=1,  h^{1,2}_{\nuu}(\b_2) h^{1,2}_{\nuu}(\b_3) = \alpha_\CSr+\alpha_\CMr-\alpha_\nuu,}\\
  \substack{h^{2,3}_{\nuu}(\b_4) h^{2,3}_{\nuu}(\b_5)}\hspace{-2mm}&=&\hspace{-2mm}\substack{(\alpha_\CSr-1)(\alpha_\nuu- \alpha_\CMr), \
  h_{\CMr}^{2,3}(\b_5) h_{\CMr}^{2,3}(\b_6) = (3\alpha_\CSr-1)^2, \
h_{\nuu}^{3,1}(\b_6) h_{\nuu}^{3,1}(\b_1) = (2\alpha_\CSr-1)(1-2\alpha_\CMr),}\\
\substack{h_{\CSr}^{1,2}(\b_4) h_{\CSr}^{1,2}(\b_7)} \hspace{-2mm}&=&\hspace{-2mm} \substack{1-\alpha_\CSr(1+\rho), \ h_{\nuu}^{2,3}(\b_7) h_{\nuu}^{2,3}(\b_8) = 1-2\alpha_\nuu, \
h_{\CSr}^{1,2}(\b_5) h_{\CSr}^{1,2}(\b_8) = (-\alpha_\CSr+\rho\alpha_\CMr+p)(1-\rho\alpha_\nuu),}\\
\substack{h_{\CSr}^{1,2}(\b_9) h_{\CSr}^{1,2}(\b_3)} \hspace{-2mm}&=& \hspace{-2mm} \substack{(-\alpha_\CMr(1+\rho)+1)(\alpha_\nuu(1+\rho)-\rho), \
h_{\nuu}^{2,3}(\b_3)h_{\nuu}^{2,3}(\b_{10}) = -\alpha_\nuu^2, \ h_{\CMr}^{1,2}(\b_{10})h_{\CSr}^{1,2}(\b_5) = ((c+2)\alpha_\nuu-1)^2.}
\end{eqnarray*}
Using \eqref{Eqn_Qv_details}, \eqref{eqn_b_val_CS_CM1}-\eqref{eqn_b_val_CS_CM2} and $h^{a,\ta}_{\nuu} (\bromga) = \bromgac^\ta - \bromgac^a, \  \forall \ a\ne \ta$, we obtain:
\begin{eqnarray*}
     h^{3,1}_{\CMr}(\bromga) = h_{(1,2)}(\bromga), \ \ h^{1,2}_{\nuu}(\bromga) = h_{(2,3)}(\bromga), \ \ h_{\CMr}^{1,2}(\bromga) = h_{(3,4)}(\bromga),\\
      h^{2,3}_{\nuu}(\bromga)=h_{(4,5)}(\bromga), \ \  h_{\CMr}^{2,3}(\bromga) = h_{(5,6)}(\bromga), \ \ h_{\nuu}^{3,1}(\bromga) = h_{(6,1)}(\bromga).
\end{eqnarray*}
and hence the following holds:
\begin{eqnarray*}
   \substack{ h_{(1,2)}(\b_1)h_{(1,2)}(\b_2)} \hspace{-2mm}&=& \hspace{-2mm} \substack{h_{(2,3)}(\b_2) h_{(2,3)}(\b_3) = h_{(3,4)}(\b_3)h_{(3,4)}(\b_4) = 1, \ h_{(4,5)}(\b_4)h_{(4,5)}(\b_5) = (\alpha_\CSr-1)(\alpha_\nuu- \alpha_\CMr),} \\ \substack{h_{(5,6)}(\b_5)h_{(5,6)}(\b_6)} \hspace{-2mm}&= &\hspace{-2mm}\substack{(3\alpha_\CSr-1)^2, \ h_{(6,1)}(\b_6)h_{(6,1)}(\b_1) = (2\alpha_\CSr-1)(1-2\alpha_\CMr).}
\end{eqnarray*}
Hence, if $ \alpha_\CSr < \min\left\{\nicefrac{1}{2},\nicefrac{(c \alpha_\nuu+\alpha_\CMr)}{2}\right\} \text{ and } \alpha_\CMr > \nicefrac{1}{2}$ then $\{(1,2),(2,3),(3,4),(4,5),(5,6),(6,1)\} \\ \subset \I^c$ and $h_{(1,2)}(\b_1)= h_{(2,3)}(\b_2)=h_{(3,4)}(\b_3)=-1 <0$, $h_{(4,5)}(\b_4) = \alpha_\CSr-1 < 0$, $h_{(5,6)}(\b_5)=2\alpha_\CSr-\alpha_\CMr-c\alpha_\nuu < 0$ and $h_{(6,1)}(\b_6) = 1-2\alpha_\CSr <0$ implying that RV graph has cycle.

\hide{\rv{If $ \alpha_\CSr < \nicefrac{1}{1+\rho},~\alpha_\nuu < \nicefrac{\rho}{1+\rho}, \text{ and } \alpha_\CMr > \nicefrac{1}{1+\rho},$ then $\{(1,2),(2,3),(3,9),(9,7),(7,4),(4,5),(5,6),(6,1)\}\not\subset \I^c$ and $ \{(1,2),(2,3),(3,9),(9,7),(7,8),(8,5),(5,6),(6,1)\} \not\subset \I^c$ as $h^{3,1}_\CMr(\b_1)= h^{1,2}_\nuu(\b_2)=-1 <0$, $h^{1,2}_\CSr(\b_3)= (\alpha_\nuu(1+\rho)-\rho)<0, h^{1,2}_\CMr(\b_9)=-1<0, h^{1,2}_\CSr(\b_7)=1>0,h^{1,2}_\CSr(\b_8)=1-\rho\alpha_\nuu >0, h^{2,3}_\nuu(\b_4) = \alpha_\CSr-1 < 0$, $h^{2,3}_\CMr(\b_5)=2\alpha_\CSr-\alpha_\CMr-c\alpha_\nuu < 0$ and $h^{1,3}_\nuu(\b_6) = 1-2\alpha_\CSr <0$ implying that RV graph has no cycle in this case.}}

On the other hand if $\alpha_\nuu \neq 0, \alpha_\nuu > \alpha_\CMr, \ \nicefrac{c}{c+2} < \alpha_\CSr < \nicefrac{2-\rho}{2+\rho}$ then $(4,5) \in \I^*$ and $(\alpha_\CSr,\nicefrac{(1-\alpha_\CSr)}{2})$ lies on both $h_{(4,5)}$ and  $\overleftrightarrow{\b_4,\b_5}$. Thus, $(\alpha_\CSr,\nicefrac{(1-\alpha_\CSr)}{2})$ is a Filippov attractor. Other Filippov attractors can also be shown in a similar way.

A close inspection of Figure \ref{fig:next_two} (Left) and \eqref{eqn_gen_e_vec_cs_cm_anu1}-\eqref{eqn_gen_e_vec_cs_cm_anu3} reveals that   $\b_v \notin \Q_v$   for each $v \ne 5$; for example, $\b_3,\b_4$ is on $\{\bromgac^1+\bromgac^2 = 1\}$, but $\Q_3,\Q_4$ is away from that line. Hence, $\b_5$ can belong to in its own region $\Q_5$ iff the point $(\alpha_{\CSr},\alpha_{\CMr})$ should lie (i) to the right side of the line $h_{(5,6)}$, (ii) above of the line $h_{(4,5)}$, (iii) above  the line $h_{(5,8)}$, and (iv) to the right to the line $h_{(5,10)}$. Hence, the conditions simplify to (i) $3\alpha_\CMr + (c+2)\alpha_\nuu <2$, (ii) $\alpha_\nuu < \alpha_\CMr$, (iii) $\alpha_{\CSr}<p+\rho\alpha_\CMr$ iff $(\rho+1) \alpha_\CMr+\alpha_\nuu > \rho$ and (iv) $\alpha_\nuu < \nicefrac{1}{c+2}$.
\eop

\end{document}